\documentclass[11pt,reqno]{amsart}
\usepackage{amssymb}
\usepackage{amsfonts}
\usepackage{mathrsfs}
\usepackage{amsmath}
\usepackage{graphicx}
\usepackage{hyperref}
\usepackage{float}
\usepackage{epstopdf}
\usepackage{color}
\usepackage{bm}
\usepackage{comment}
\usepackage{soul}

\allowdisplaybreaks
\newtheorem{theorem}{Theorem}[section]
\newtheorem{proposition}[theorem]{Proposition}
\newtheorem{lemma}[theorem]{Lemma}
\newtheorem{corollary}[theorem]{Corollary}

\newtheorem{definition}[theorem]{Definition}

\newtheorem{thm}{Theorem}

\def\diam{\mathrm{diam}}

\def\dist{\mathrm{dist}}
\def\osc{\mathrm{osc}}

\def\mcA{\mathcal{A}}
\def\mcB{\mathcal{B}}

\def\mcD{\mathcal{D}}
\def\mcE{\mathcal{E}}
\def\mcF{\mathcal{F}}
\def\tf{\tilde{f}}
\def\mcM{\mathcal{M}}
\def\mcN{\mathcal{N}}
\def\R{\mathbb{R}}
\def\msG{\mathscr{G}}

\def\sn{\stackrel{n}{\sim}}
\def\sm{\stackrel{m}{\sim}}

\numberwithin{equation}{section}

\begin{document}
\title[Self-similar Dirichlet forms on polygon carpets]{Self-similar Dirichlet forms on polygon carpets}

\author{Shiping Cao}
\address{Department of Mathematics, Cornell University, Ithaca 14853, USA}
\email{sc2873@cornell.edu}
\thanks{}

\author{Hua Qiu}
\address{Department of Mathematics, Nanjing University, Nanjing, 210093, P. R. China.}
\thanks{The research of Qiu was supported by the National Natural Science Foundation of China, grant 12071213, and the Natural Science Foundation of Jiangsu Province in China, grant BK20211142.}
\email{huaqiu@nju.edu.cn}

\author{Yizhou Wang}
\address{Department of Mathematics, Nanjing University, Nanjing, 210093, P. R. China.}
\thanks{}
\email{1159955494@qq.com}

\subjclass[2010]{Primary 28A80, 31E05}

\date{}

\keywords{unconstrained Sierpinski carpets, Dirichlet forms, diffusions, self-similar sets}

\maketitle

\begin{abstract}
We construct symmetric self-similar diffusions with sub-Gaussian heat kernel estimates on two types of polygon carpets,
which are natural generalizations of planner Sierpinski carpets (SC). The first ones are called perfect polygon carpets that are natural analogs of SC in that any intersection cells are either side-to-side or point-to-point. The second ones are called bordered polygon carpets which satisfy the boundary including condition as SC but allow distinct contraction ratios.
\end{abstract}

\section{Introduction}\label{sec1}
We consider the existence of self-similar Dirichlet forms on polygon carpets, which are natural generalizations of planar Sierpinski carpets (SC), see Figure \ref{fig1}. In history, as a milestone in analysis on fractals \cite{B,s}, the locally symmetric diffusions with sub-Gaussian heat kernel estimates on SC were first constructed by Barlow and Bass in their pioneering works \cite{BB,BB1,BB2}, using a probabilistic method. By introducing the difficult coupling argument, the result was later extended to generalized Sierpinski carpets (GSC) \cite{BB3}, which are higher dimensional analogues of SC. In the mean time, a different approach using Dirichlet forms was introduced by Kusuoka and Zhou \cite{KZ}. The strategy is to construct self-similar Dirichlet forms on fractals as limits of averaged rescaled energies on cell graphs. The proof is analytic except a key step to verify that the resistance constants and the Poincare constants are comparable, which was achieved by the probabilistic ``Knight move'' argument of Barlow and Bass's. The two approaches are both based on the delicate geometry structure (for example, local symmetry) of SC (or GSC), and were shown to be equivalent in 2010 in the celebrated work by Barlow, Bass, Kumagai and Teplyaev \cite{BBKT}.

\begin{figure}[htp]
	\includegraphics[width=4.4cm]{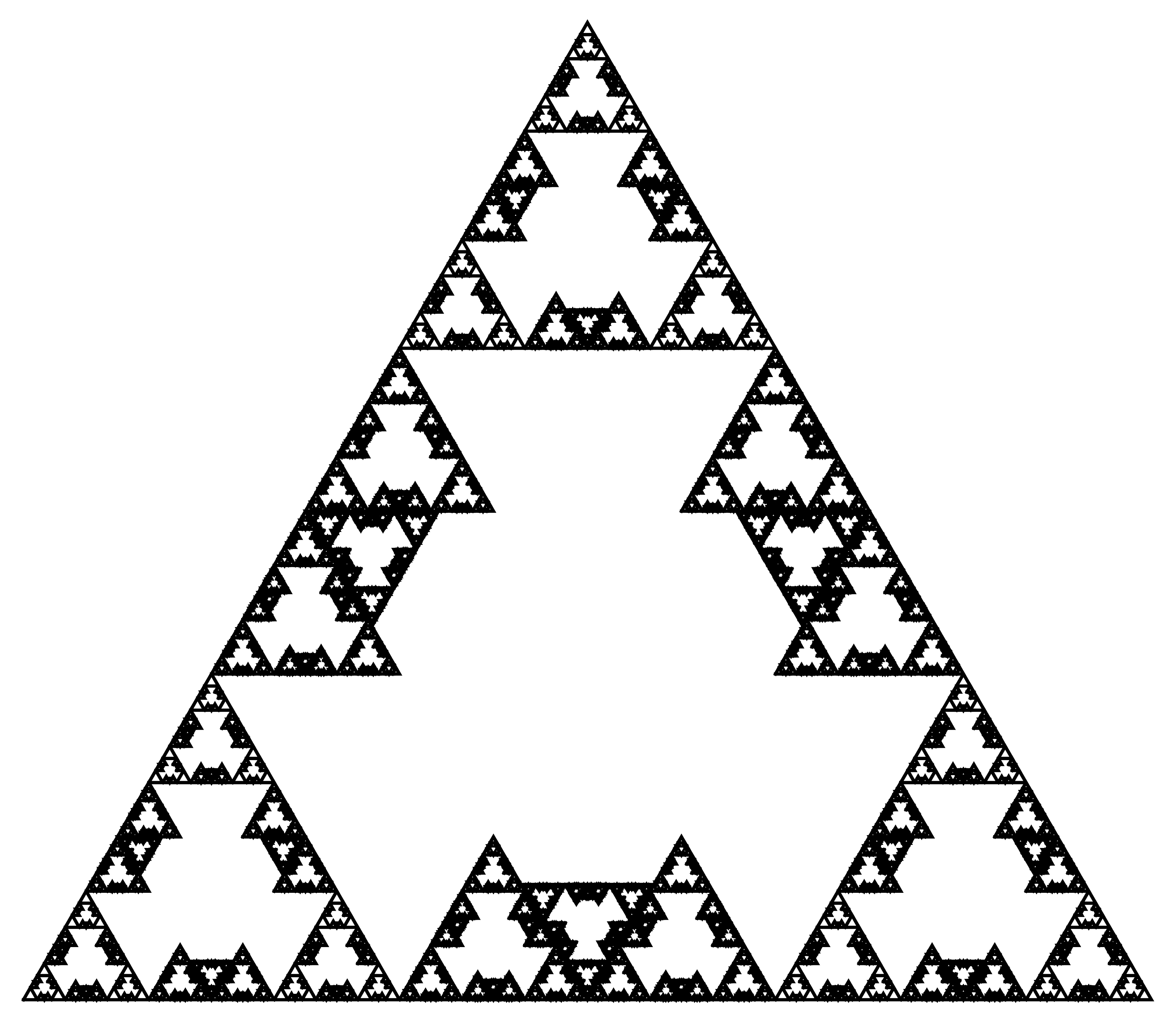}\hspace{0.1cm}
	\includegraphics[width=4cm]{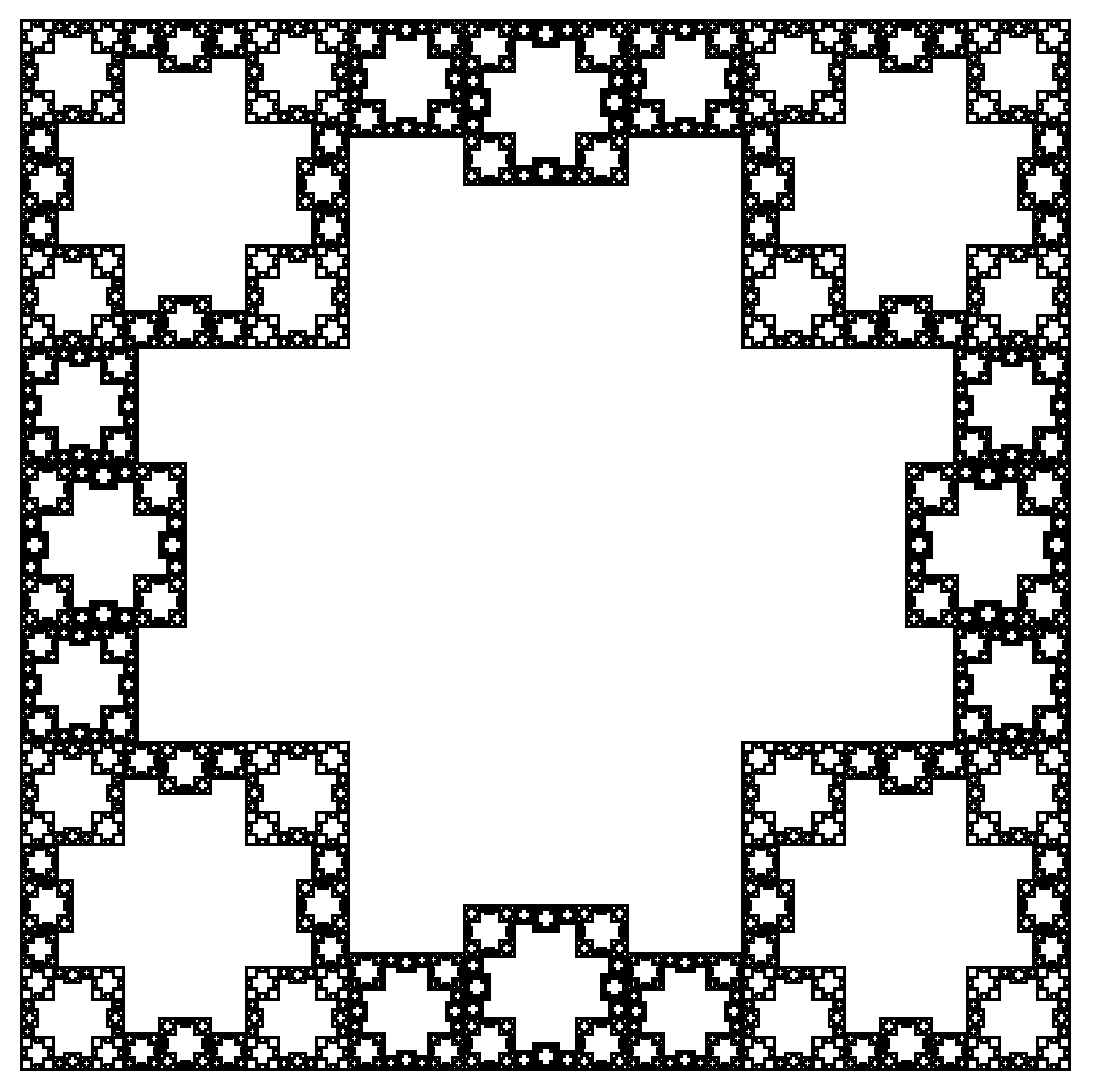}\hspace{0.1cm}
	\includegraphics[width=4.4cm]{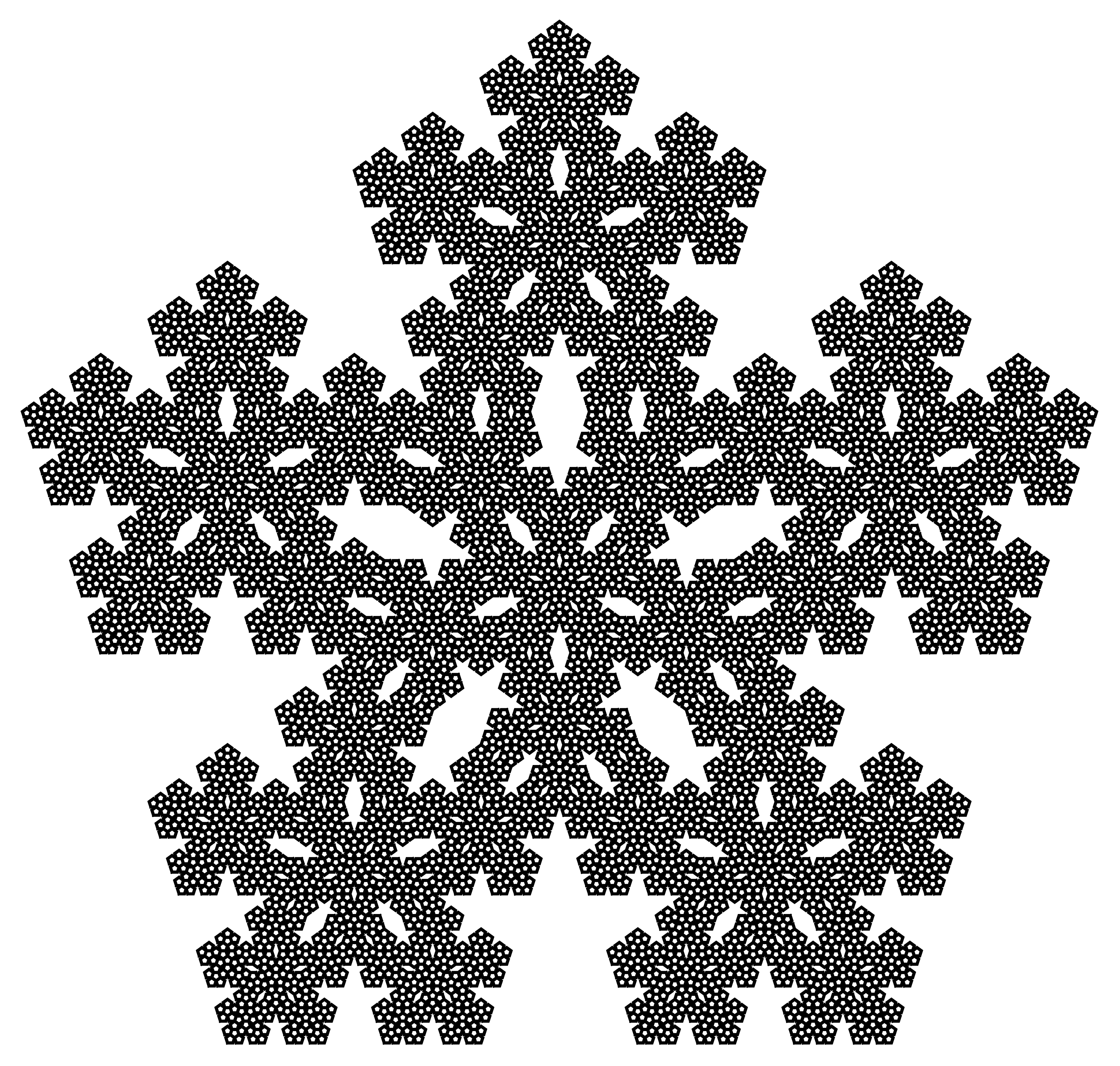}
	\caption{Some polygon carpets that have good self-similar Dirichlet forms.}\label{fig1}
\end{figure}	

Recently, two of the authors extended the results to unconstrained Sierpinski carpets (USC) in \cite{CQ3} based on the method of Kusuoka-Zhou \cite{KZ}, but replacing the probabilistic argument with a purely analytic chaining argument of resistances. The USC are more flexible in geometry as cells except those along the boundary are allowed to live off the grids, see the left picture in Figure \ref{fig2} for an example. To overcome the essential difficulty from the worse geometry, a ``building brick'' technique inspired by a reverse thinking of the trace theorem of Hino and Kumagai \cite{HK} was developed to construct  functions with good boundary values and controllable energy estimates.

Unexpectedly,  it was shown  in \cite{CQ4} by two of the authors that the existence of good diffusions on Sierpinski carpet like fractals is not always the truth, see the right picture in Figure \ref{fig2} for a counter-example. The construction of this example was partially inspired by the work of Sabot \cite{Sabot}, of which corner vertices loosely connected with inner cells which causes that the effective resistances between corner vertices are uncomparable with that between opposite sides. Naturally, it is of great interest to see how the geometry of the fractals plays a role.

\begin{figure}[htp]
	\includegraphics[width=4.2cm]{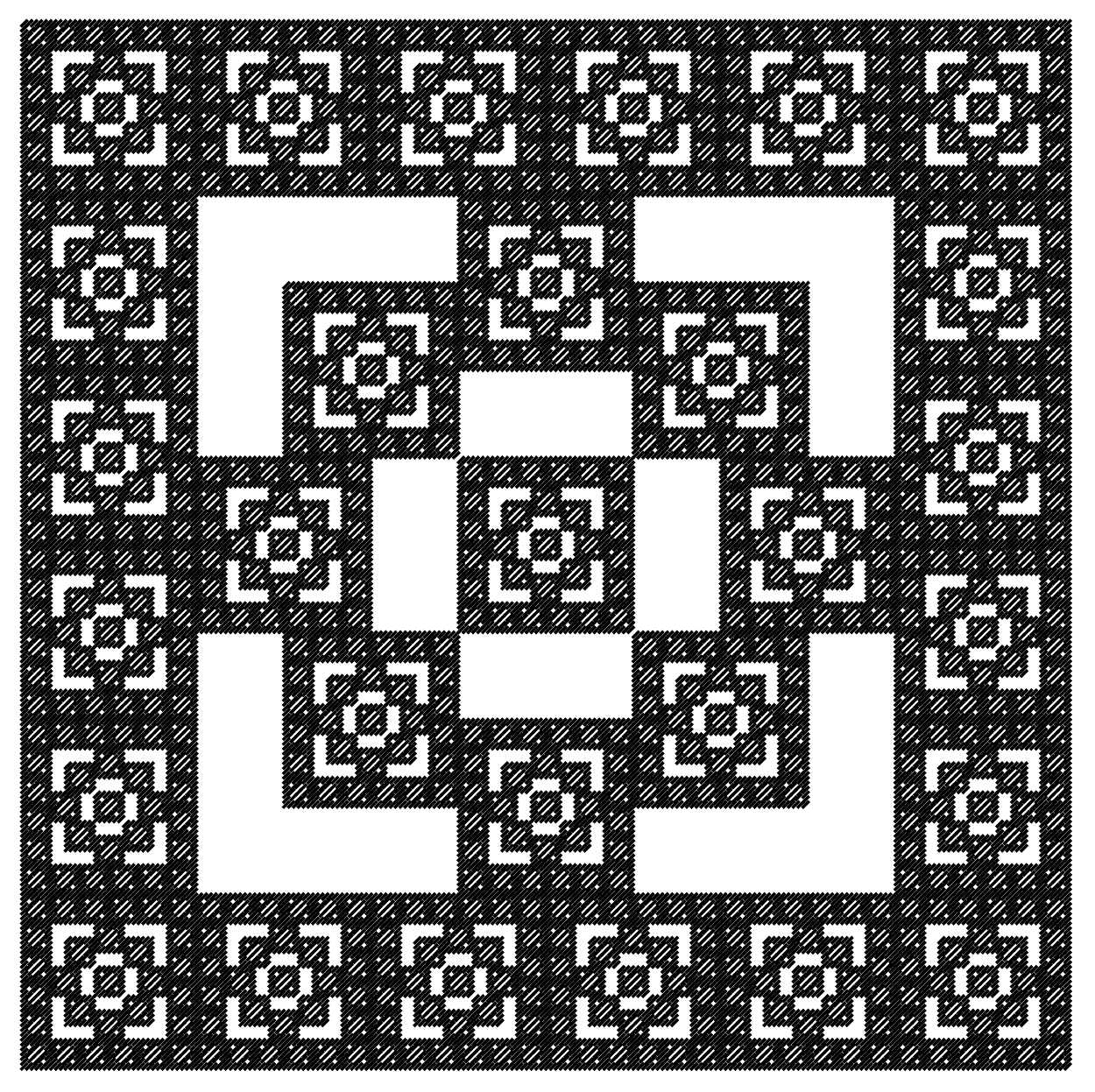}\hspace{1cm}
	\includegraphics[width=4.2cm]{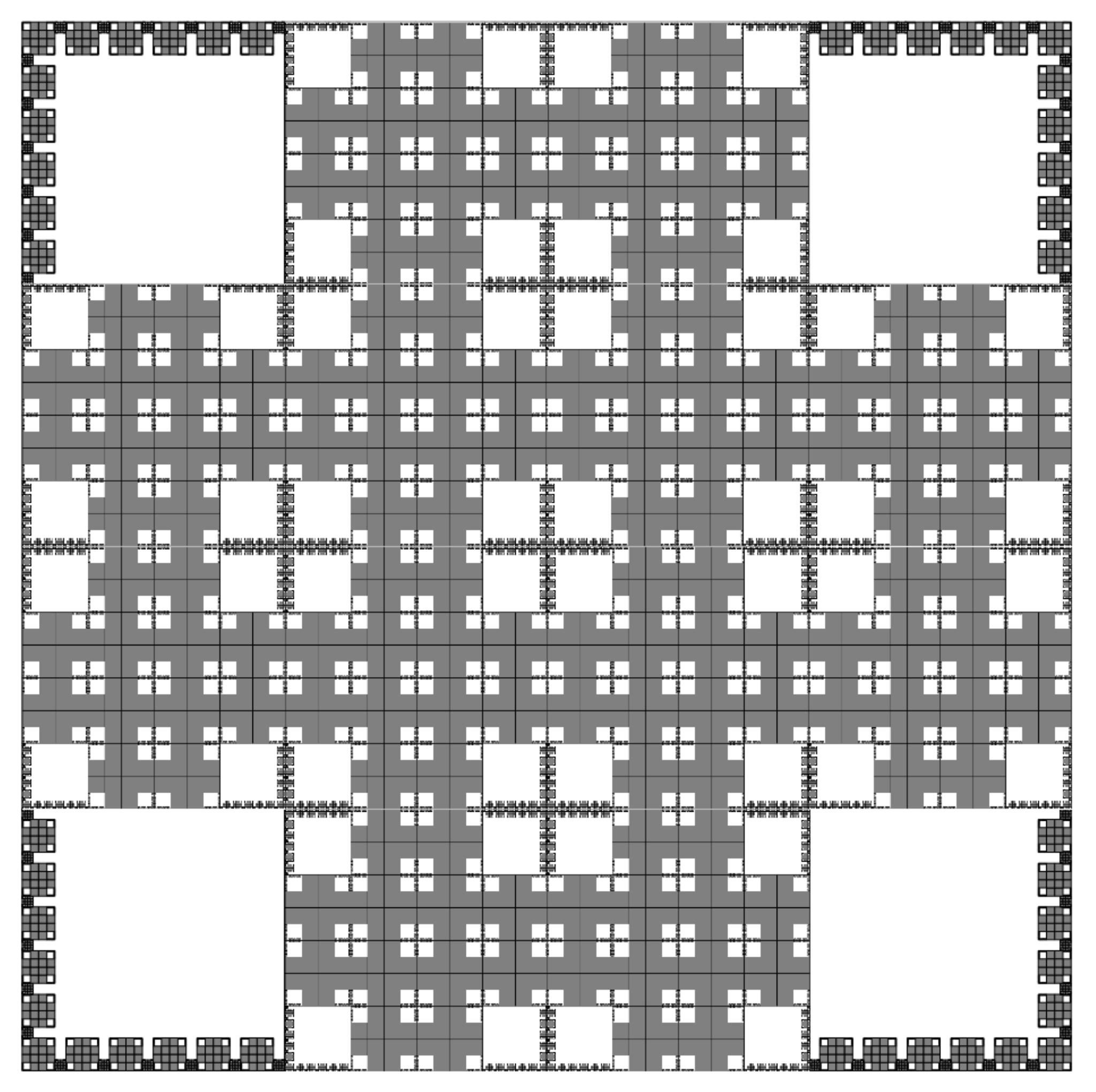}
	\caption{The left one is a USC. The right one is a Sierpinski carpet like fractal without good Dirichlet form (right).}\label{fig2}
\end{figure}

In this paper, as a sequel to \cite{CQ3}, motivated by \cite{CQ4},  the main aim of the authors is to extend the existence result to more general planar symmetric fractals. We consider two type of polygon carpets: perfect polygon carpets and bordered polygon carpets, see Definition \ref{def22} and also Figure \ref{fig1} for an illustration. Perfect polygon carpets are natural analogs  to SC  in that cells are side-to-side arranged, keeping the locally symmetric structure; while bordered polygon carpets insist the boundary including condition as SC (and USC), but allow distinct contraction ratios of the iterated function systems (i.f.s.), which include many irrationally ramified fractals (see the Sierpinski cross considered in \cite{ki35} by Kigami).

Indeed, the analysis on the second type of fractals is more challenging, and is of the main interest of the paper. Due to the counter-example constructed in \cite{CQ4}, it is no hope that the existence result holds for all bordered polygon carpets. A new technique in this paper is that we will show that two cells close in resistance metric can be connected by a set with small diameter in resistance metric, and in particular if this happens for two cells on the opposite sides of the fractal, there is a ``ring'' passing through the fractal with small diameter in resistance metric. Basing on this observation, we could extend the existence result to a large class of hollow bordered polygon carpets, where ``hollow'' means all the first generation cells are located along the boundary of the fractal.\vspace{0.2cm}

\begin{thm}\label{th1}
Let $K$ be a polygon carpet with i.f.s. $\{\Psi_i\}_{i\in S}$, contraction ratios $\{\rho_i\}_{i\in S}$, that satisfies either (1) or (2):

(1). $K$ is a perfect polygon carpet;

(2). $K$ is a bordered polygon carpets satisfying \textbf{(H)} and \textbf{(C)}.

\noindent Let $\mu$ be the normalized Hausdorff measure on $K$. Then, there is a local regular self-similar Dirichlet form $(\mathcal{E},\mathcal{F})$ on $L^2(K,\mu)$ with $r_i=\rho_i^\theta,i\in S$, such that
$$\mcE(f) = \sum_{i\in S}r_i^{-1}\mcE(f\circ \Psi_i),\quad \forall f\in \mathcal F,$$
for some $\theta>0$.
In addition, $\mathcal{F}\subset C(K)$. Moreover, there is a constant $C>0$ such that
$$\big|f(x) - f(y)\big|^2 \leq C\mathcal{E}(f)\cdot |x - y|^{\theta}, \quad\forall x,y\in K,\forall f\in \mathcal{F}.$$
\end{thm}

See the exact definition of \textbf{(H)} and \textbf{(C)} in Section \ref{sec6}. By applying \cite[Theorem 15.10 and 15.11]{ki4} by Kigami, we know that Theorem \ref{th1} implies that there exists a diffusion process on $K$ with sub-Gaussian heat kernel estimates. \vspace{0.2cm}

We will follow the strategy of Kusuoka and Zhou \cite{KZ}, and extend the ``building brick'' technique in \cite{CQ3}, so that the method will be purely analytic. Although the geometry of polygon carpets are much more complicated than post critically finite (p.c.f.) self-similar sets \cite{ki1,ki2,ki3} of Kigami, we can still take the advantage of the strong recurrence. In particular, we use the simplified model, resistance forms, to describe the limit form, though there is not  a compatible sequence argument as in \cite{ki3}.

 \vspace{0.2cm}

Finally, we briefly introduce the structure of the paper.

We recommend readers to read the definitions and notations in Section \ref{sec2} and \ref{sec3} carefully, and quickly go over the other parts. In Section \ref{sec2}, we introduce the definition of polygon carpets, with Proposition \ref{prop27} proved in Appendix \ref{AppendixA}. In Section \ref{sec3}, we show that once we have good resistance estimates, we can construct good self-similar Dirichlet forms. This section is not new, but a modification of \cite{KZ,CQ3} to include the distinct ratios case. Some well-known estimates in \cite{KZ} (which needs some modification) are provided in Appendix \ref{AppendixB}. Also, see Appendix \ref{AppendixC} for the proof of Proposition \ref{prop35}.

Section \ref{sec4}, \ref{sec5} ,\ref{sec6} will be the main parts of the paper. In Section \ref{sec4}, we will consider important properties about resistance metrics. The key observations are Propositions \ref{prop44} and  \ref{prop45}, in which we show if two cells on the boundary are far away in Euclidean metric, but close in resistance metric, one can find a ``ring'' connecting them with small diameter in resistance metric. Section \ref{sec5}  is a short section on the existence of good Dirichlet forms on perfect polygon carpets. In Section \ref{sec6}, we study bordered polygon carpets. Our arguments will be based on Corollary \ref{coro75} and the geometric conditions \textbf{(H)} and \textbf{(C)} of the fractals.

We end the story for hollow bordered polygon carpets in Section \ref{sec7}, where we will develop a more flexible ``building brick'' technique to construct functions with good boundary values and glue them together to verify the resistance estimates.

Throughout the paper, we will write $a\lesssim b$ for two variables (functions, forms) if there is a constant $C>0$ such that $a\leq C\cdot b$, and write $a\asymp b$ if both $a\lesssim b$ and $b\lesssim a$ hold. We will always abbreviate that $a\wedge b=\min\{a,b\}$ and $a\vee b=\max\{a,b\}$.

\section{Geometry of Polygon carpets}\label{sec2}

In this section, we introduce the definition of polygon carpets, and present some basic geometric properties of these fractals as well as their associated graph approximation sequences.

We consider fractals in $\mathbb{R}^2$ in this paper. For two points $x,y\in\mathbb{R}^2$, we write the line segment connecting $x,y$ as $\overline{x,y}$, and the Euclidean distance between $x,y$ as $|x - y|$. For sets $A,B\subset \R^2$, we write $\dist(A,B) = \inf\{|x - y|:x\in A,y\in B\}$ as the Euclidean \textit{distance} between $A,B$. It will always be positive providing that $A,B$ are disjoint compact sets. For $A\subset \mathbb{R}^2$, we write $\diam(A)=\sup\{|x-y|: x,y\in A\}$ as the \textit{diameter} of $A$.

We will always write $\mcA$ to be an \textit{equilateral polygon} in $\R^2$ with side length $1$. Let $N_0 \geq 3$ be the number of vertices of $\mcA$,  and $q_1,\cdots q_{N_0}$ be the vertices arranged counter-clockwise. Denote $S_0= \{1,\cdots,N_0\}$, and write $L_{i} = \overline{q_i,q_{i+1}},i = 1,\cdots N_0$ for the sides of $\mcA$ accordingly, where $q_{N_0 + 1} = q_{1}$. We denote the  Euclidean \textit{boundary} of $\mcA$ as $\partial \mcA := \bigcup_{i = 1}^{N_0}L_i$ and write $\mcA^{o} = \mcA - \partial \mcA$ for the \textit{interior} of $\mcA$. We denote the canonical symmetric group associated with $\mcA$ as $\msG$, generated from $N_0$ many axial symmetries $\Gamma_{i,i+1}$'s and $N_0$ many rotational symmetries $\Gamma_{i}$'s, where for $1\leq i,j\leq N_0$, we denote $\Gamma_{i,j}$ the axial symmetry that exchanges $q_i, q_j$, and for $0\leq i<N_0$, denote $\Gamma_i$ the rotational symmetry that shifts each $q_j$ to $q_{i + j}$ for $j\in S_0$. In particular, $\Gamma_0 = id|_{\mcA}$.

Let $S$ be a non-empty finite set with $N: = \# S \geq N_0$. For each $i\in S$, let $\Psi_i$ be a \textit{contracting similarity} on $\R^2$, defined by $\Psi_i(x)=\pm\rho_i x+c_i$
for some $0<\rho_i < 1$, $c_i\in\mathbb{R}^2$, and call $\rho_i$ the \textit{contraction ratio} of $\Psi_i$.
We require that for each $i\in S$, $\Psi_i\mcA\subset \mcA$.
Then there is a unique non-empty compact set $K\subset \mcA$ satisfying
\begin{equation}\label{eq21}
K = \bigcup_{i \in S} \Psi_i K.
\end{equation}
Call $\mathcal{I} := \{\Psi_i\}_{i\in S}$ the \textit{iteration function system} (i.f.s. for short) associated with $K$.

\begin{definition}[Perfectly touching]\label{def21}
  For $i\neq j\in S$, we say $\Psi_i$, $\Psi_j$ are \emph{perfectly touching} if $\Psi_i\mcA\cap \Psi_j\mcA = \Psi_iL_k=\Psi_jL_{k'}$ for some $k, k'$ in $S_0$.

  We say $\{\Psi_i\}_{i\in S}$ a \emph{perfect i.f.s.} if

  \noindent(a). for any $i\neq j\in S$, there exists a chain of indices $i_0,\cdots,i_l\in S$ so that $i_0 = i,i_l = j$, and $\Psi_{i_k}$, $\Psi_{i_k - 1}$ are perfectly touching for any $k = 1,\cdots,l$;

  \noindent(b). for any $i\neq j\in S$ with $\Psi_i\mcA\cap \Psi_j\mcA\neq \emptyset$, either $\Psi_i$, $\Psi_j$ are perfectly touching, or  $\Psi_i\mcA\cap \Psi_j\mcA = \Psi_iq_k = \Psi_jq_{k'}$ for some $k,k'\in S_0$.

\end{definition}

\noindent\textbf{Remark.} A perfect i.f.s. always has the same contraction ratios.

\begin{definition}[Polygon carpets]\label{def22}\quad
Suppose the i.f.s. $\mathcal I:=\{\Psi_i:\mcA\to \mcA\}_{i\in S}$ satisfies

\noindent\emph{(Open set condition).} $\Psi_i(\mcA^{o})\cap \Psi_{j}(\mcA^{o}) = \emptyset,\ \forall i\neq j\in S$;

\noindent\emph{(Connectivity). } $K$ is connected;

{\color{red}
}

\noindent\emph{(Symmetry).} $\Gamma\big(\bigcup_{i\in S} \Psi_i\mcA\big)=\bigcup_{i\in S} \Psi_i\mcA$ for any $\Gamma\in \msG$;

\noindent\emph{(Non-trivial).} $\bigcup_{i\in S}\Psi_i\mcA \neq \mcA$.\\
Call the unique compact set $K\subset \mcA$ associated with $\mathcal I$ as in (\ref{eq21}) a \emph{polygon carpet}.

If in addition $\mathcal I$ satisfies

\noindent\emph{(Perfectly touching).} $\mathcal I$ is a perfect i.f.s.,

then call $K$ a \emph{prefect polygon carpet};
alternatively, if $\mathcal I$ satisfies

\noindent\emph{(Boundary included).} $\partial \mcA \subset \bigcup_{i\in S}\Psi_i\mcA$,\\
then call $K$ a \emph{bordered polygon carpet}.

Call both these two types of carpets \emph{regular polygon carpets}.

\end{definition}

\begin{figure}[htp]
	\includegraphics[width=4.9cm]{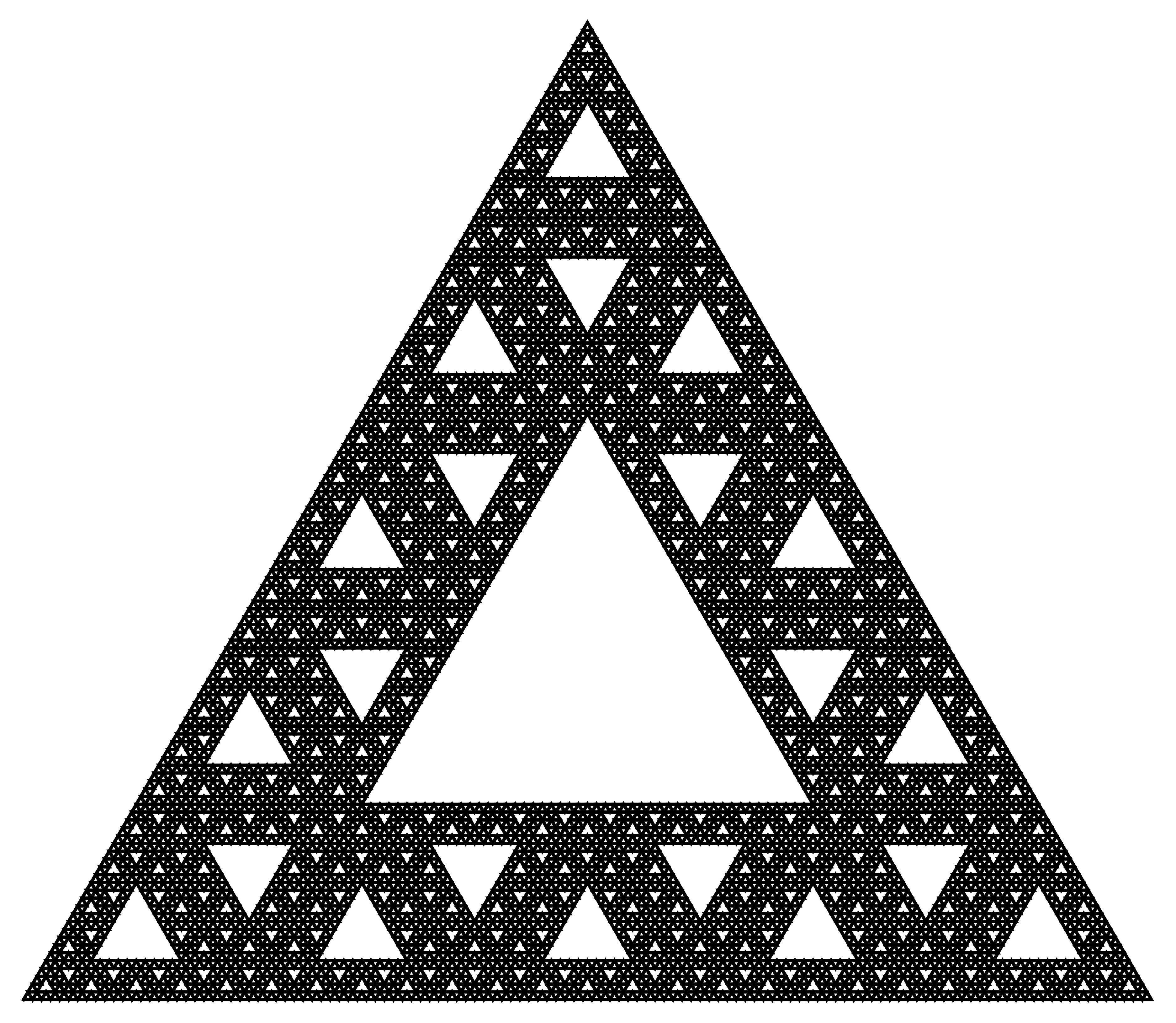}\hspace{0.1cm}
    \includegraphics[width=4.9cm]{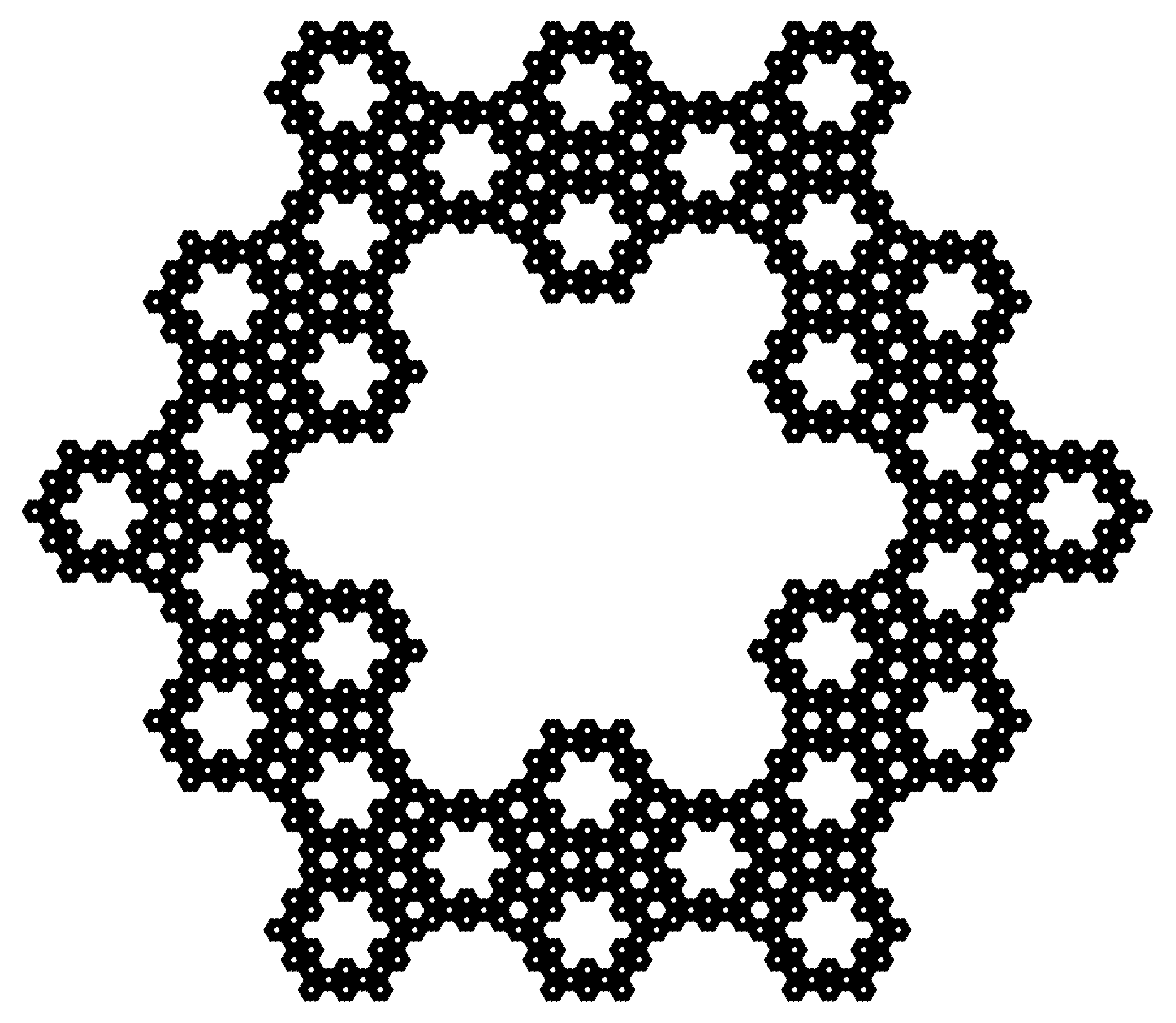} \hspace{0.1cm}    \includegraphics[width=4.4cm]{bpc4.pdf}\vspace{0.5cm}
	\caption{Examples of regular  polygon carpets: (a) is both perfect and bordered, (b) is only perfect, and (c) is only bordered.}
		\begin{picture}(0,0)
		\put(-148,45){(a)}\put(-1,45){(b)}\put(141,45){(c)}
	\end{picture}
	\label{fig3}
\end{figure}

See Figure \ref{fig3} for some examples of regular polygon carpets. Clearly, due to the open set condition, the boundary included condition can only hold when $N_0=3$ or $4$, but it allows the contraction ratios to be distinct. To deal with the possible distinct ratios case, we need to divide the fractal $K$ into cells of comparable sizes in later context. When $N_0=4$, the contraction ratios are the same, and the boundary included condition holds, $K$ is a USC considered in \cite{CQ3}. If in addition, $N=8$, $K$ is the standard Sierpinski carpet SC. See Figure \ref{fig4} for examples.

\begin{figure}[htp]
	\includegraphics[width=4.5cm]{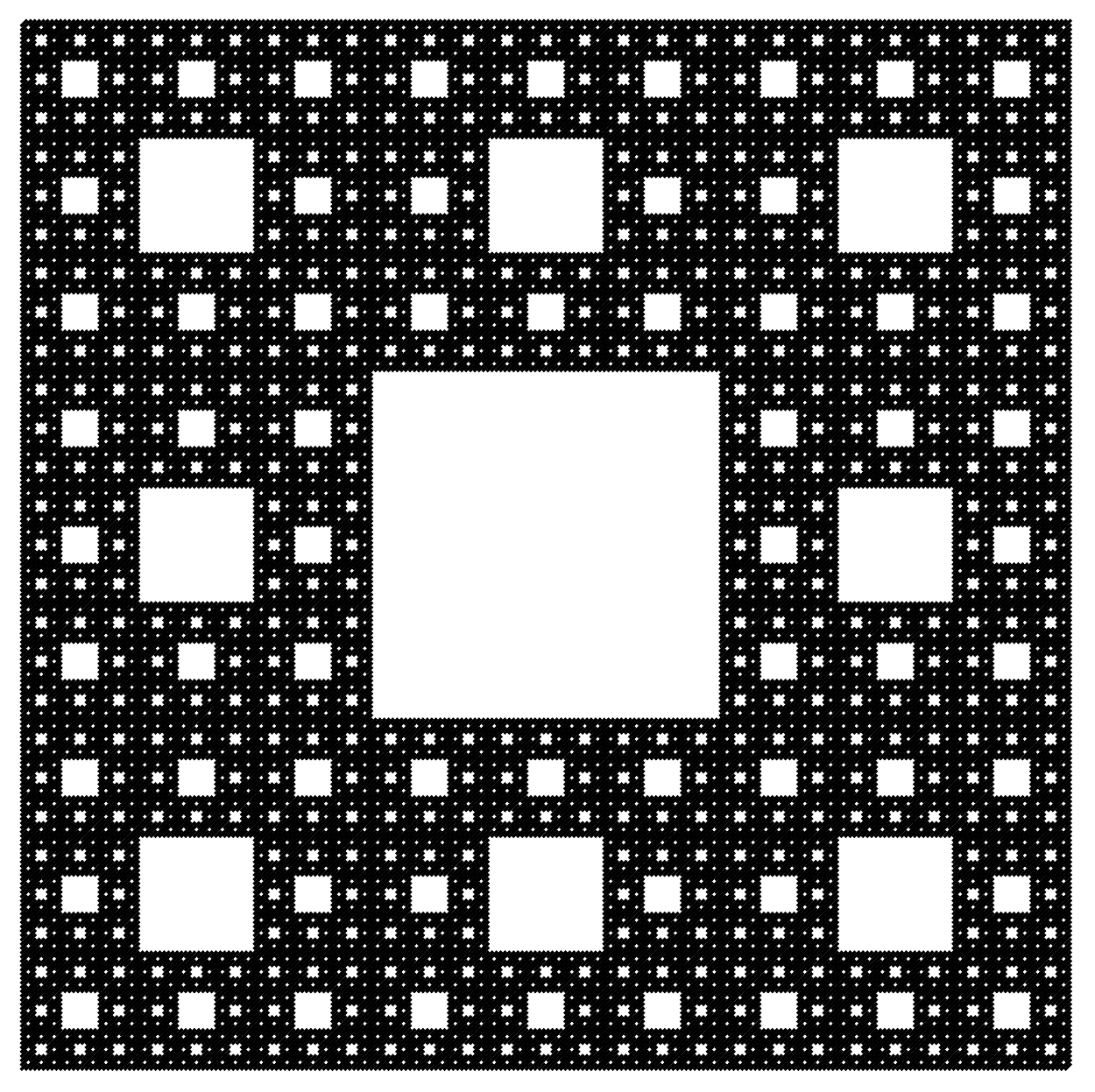}\hspace{1cm}
    \includegraphics[width=4.5cm]{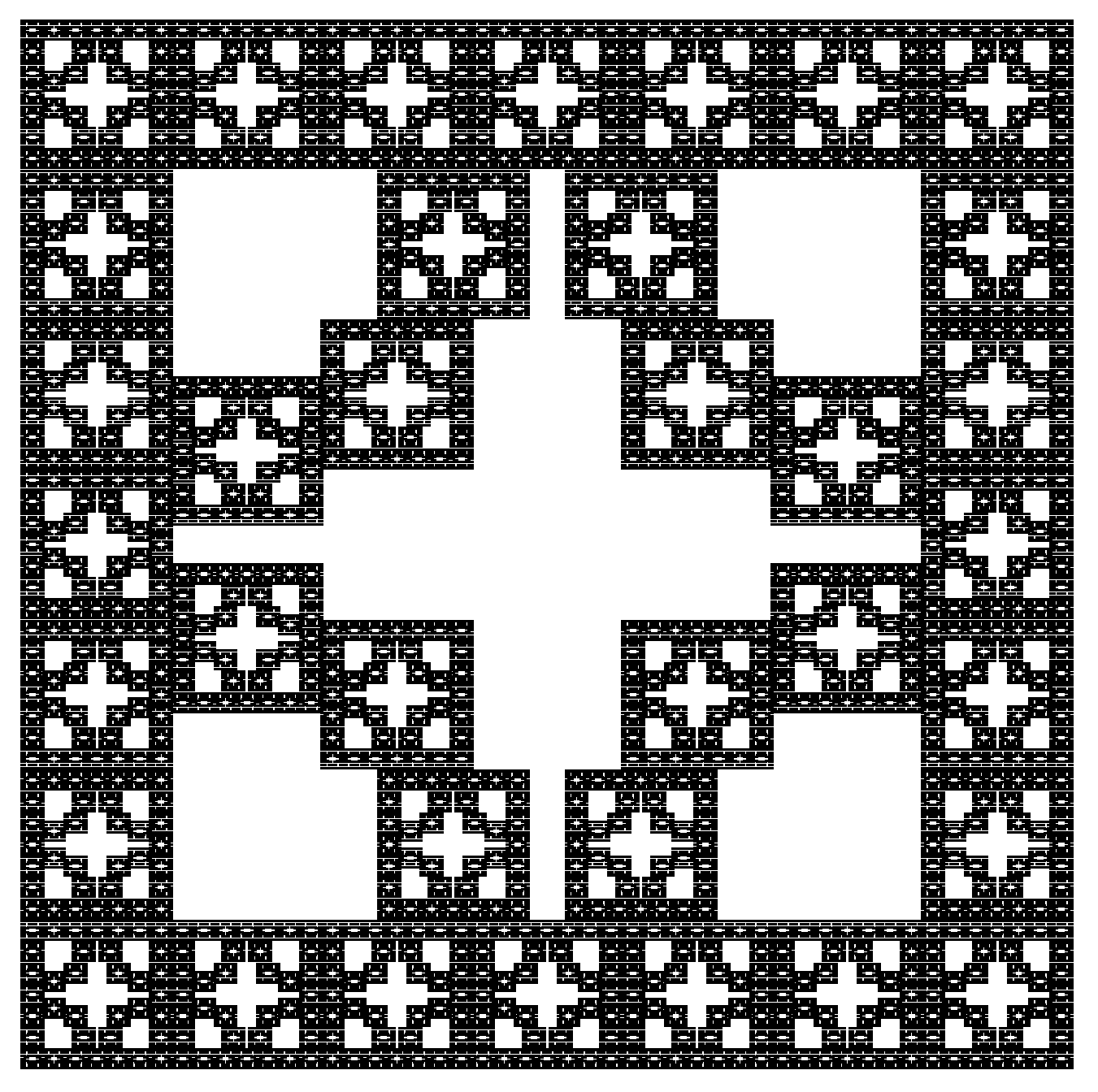}
	\caption{Examples of USC, and the left one is SC.}
	\label{fig4}
\end{figure}

From now on, we always assume $K$ to be a regular polygon carpet, and $\{\Psi_i\}_{i\in S}$ to be its i.f.s.. We denote
$$\partial K: = K\cap \partial \mcA.$$
Immediately, $\partial K=\bigcup_{i\in S_0}\partial_i K$ with $\partial_i K = \partial K\cap L_i$ for $i\in S_0$.

By open set condition, the Hausdorff dimension $d_H$ of $K$ is the unique solution of the equation $\sum_{i\in S}\rho_i^{d_H} = 1$. Since $\sum_{i\in S}\rho_i^2 < 1$ by non-trivial condition, it always holds that $d_H < 2$. We will always set $\mu$ to be the \textit{normalized $d_H$-dimensional Hausdorff measure} on $K$, i.e. $\mu$ is the unique self-similar probability measure on $K$ satisfying $\mu=\sum_{i\in S}\rho_i^{d_H}\mu\circ\Psi_i^{-1}$.\vspace{0.2cm}

\noindent\textbf{Basic notations.}

(1). Let $W_0 = \{\emptyset\}$, $W_n = S^n = \{w = w_1\cdots w_n:w_i\in S,i = 1,\cdots ,n\}$ for $n\geq 1$, and $W_* = \bigcup_{n\geq 0}W_n$. The elements in $W_*$ are called \textit{finite words}.  For each $w\in W_{n}$, $n\geq 1$, we write $|w|=n$ for the \textit{length} of $w$, write $\Psi_w = \Psi_{w_1}\circ \cdots\circ \Psi_{w_{n}}$ and $\rho_w = \prod_{i = 1}^{n}\rho_{w_i}$, and call $\Psi_wK$ an \textit{$w$-cell} in $K$. In particular, $|\emptyset| = 0$, $\Psi_{\emptyset} = id|_{\mcA}$, $\rho_{\emptyset} = 1$ and $K_\emptyset =K$.

(2). For $w\in W_n,v\in W_m$, we denote $w\cdot v = w_1\cdots w_{n}v_1\cdots v_{m}\in W_{n + m}$. For $w,v\in W_*$,  by the open set condition, $\Psi_{w}\mcA\subset \Psi_v\mcA$ if and only if $w\cdot w' = v$ for some $w'\in W_*$. For $A,B\subset W_*$, denote $A\cdot B = \{w\cdot v:w\in A,v\in B\}$. In particular, write $w\cdot B=\{w\}\cdot B$ for short.

(3). Let $B\subset W_*$ and $v\in W_*$, we define $v^{-1}\cdot B=\{w\in W_*:v\cdot w\in B\}.$
Clearly, $v^{-1}\cdot v\cdot B=B$, however, it is often false that $v\cdot v^{-1}\cdot B=B$ (we still have $v\cdot v^{-1}\cdot B\subset B$).

(4). We say a finite set $\Lambda\subset W_*$  a \textit{partition} of $W_*$ if  $\bigcup_{w\in \Lambda}\Psi_wK = K$ and $\mu(\Psi_w K \cap \Psi_vK)=0, \forall w\neq v\in \Lambda$. Let $\Lambda,\Lambda'\subset W_*$ be two partitions. We say $\Lambda'$ is \textit{finer} than $\Lambda$, if for any $w\in \Lambda'$, there is some $v\in \Lambda$ such that $\Psi_w\mcA \subset \Psi_v\mcA$.

(5).  Let $\rho_* = \min_{i\in S}\rho_i$. Let $\sigma: W_* \to W_*$ be the operator defined as $\sigma(w) = w_1\cdots w_{n - 1}$ for $w=w_1\cdots w_n\in W_n,n\geq 1$ and $\sigma(\emptyset) = \emptyset$. We define $\Lambda_0 = W_0$, and for $n\geq 1$,
  $$\Lambda_n = \big\{w\in W_*:\rho_w\leq \rho_*^{n} < \rho_{\sigma(w)}\big\}.$$
Write $\Lambda_* = \bigcup_{n\geq 0}\Lambda_n$.
Clearly, for each $n\geq 0$, $\Lambda_n$ forms a partition of $W_*$, and $\Lambda_n$ is finer than $\Lambda_{m}$ for $n \geq m$. For each $w\in \Lambda_n$, call $\Psi_wK$ a \emph{level}-$n$ \emph{cell} ($n$-\emph{cell} for short) in $K$ and write $\|w\| = n$. In addition, if $\rho_i=\rho_*$ for all $i\in S$, then $\Lambda_n=W_n$ for each $n\geq0$.

(6). For any $n,m\geq 0$, and $w\in \Lambda_n$, we  define
$$\mcB_m(w) = \{v\in \Lambda_{n + m}:\Psi_vK \subset \Psi_wK\}.$$ Clearly, $\mcB_m(w) = w\cdot w^{-1}\cdot\Lambda_{n+m}$  represents the collection of $(n+m)$-cells contained in $F_wK$.
 Write $\mcB_m(A) = \bigcup_{w\in A}\mcB_m(w)$ for $A\subset \Lambda_n$.

\begin{lemma}\label{lemma23}
For  $w\in \Lambda_*$ and $m\geq 1$,  $w^{-1}\cdot\mcB_{m}(w)$ is finer than $\Lambda_{m -1}$, and $\Lambda_{m + 1}$ is finer than $w^{-1}\cdot\mcB_{m}(w)$.
\end{lemma}
\begin{proof}
  The lemma follows from the fact that, for $w\in \Lambda_n$, $v\in w^{-1}\cdot\mcB_{m}(w)$, we always have
  $$\rho_*^{m+1}=\rho_*^{n+m+1}\rho_*^{-n}<  \rho_v=\rho_{wv}\rho_{w}^{-1} < \rho_*^{n + m}\rho_*^{-(n+1)} = \rho_*^{m-1}.$$
\end{proof}

\noindent\textbf{Remark.} Unlike the case that  $\rho_i = \rho_*$ for any $i\in S$, $w^{-1}\cdot\mcB_{m}(w)$ may not belong to $\{\Lambda_n\}_{n\geq 0}$.

\begin{proposition}\label{prop24}
  There exists $C > 0$ depending only on $\rho_*$ such that $C^{-1}\rho_*^{-md_H} \leq \#\mcB_m(w) \leq C\rho_*^{-md_H}$ for any $w\in \Lambda_*$ and $m\geq 0$. In particular,
  $$C^{-1}\rho_*^{-md_H} \leq \#\Lambda_m \leq C\rho_*^{-md_H}.$$
\end{proposition}
\begin{proof}
Suppose $w\in \Lambda_*,m\geq 0$, then we have
  $$ \sum_{v\in \mcB_m(w)}\rho_v^{d_H} = \sum_{v\in \mcB_m(w)}\mu(\Psi_vK) = \mu(\Psi_wK) = \rho_w^{d_H}.$$

  So on one hand
  $$ \rho_*^{\|w\| d_H} \geq \rho_w^{d_H} = \sum_{v\in \mcB_m(w)}\rho_v^{d_H} > \sum_{v\in \mcB_m(w)}\rho_*^{(m + \|w\| + 1)d_H}, $$
  which gives $\#\mcB_m(w) \leq \rho_*^{-( m + 1)d_H}$,
and on the other hand
  $$\rho_*^{(\|w\| + 1)d_H} < \rho_w^{d_H} = \sum_{v\in \mcB_m(w)}\rho_v^{d_H} \leq \sum_{v\in \mcB_m(w)}\rho_*^{(m + \|w\|)d_H},$$
  which gives $\#\mcB_m(w) \geq \rho_*^{-( m - 1)d_H}$.
\end{proof}

\begin{definition}[$m$-boundary of cells]\label{def25}
For $m,n\geq 0$ and $w\in \Lambda_n$, define
$$\partial \mcB_{m}(w) = \big\{v\in \mcB_{m}(w):\Psi_vK\cap \Psi_w\partial K \neq \emptyset\big\}.$$
In particular, we write $\partial \Lambda_m:=\partial\mcB_m(\emptyset)$.
\end{definition}

Obviously, $\Psi_wK\supset \bigcup_{v\in \partial\mcB_m(w)}\Psi_vK \supset \bigcup_{v\in \partial\mcB_{m + 1}(w)}\Psi_vK$ for any $w\in \Lambda_*$ and $m\geq 0$. We have
$$\Psi_w\partial K = \bigcap_{m\geq 0}\ \bigcup_{v\in \partial\mcB_m(w)}\Psi_v\mcA.\\$$

\begin{proposition}\label{prop26}
  The Hausdorff dimension of $\partial K$ is strictly smaller than $d_H$.
\end{proposition}
\begin{proof}
  By the symmetry condition, we see the dimension of $\partial K$ is equal to that of  $\partial_1 K$, denoted as $\dim_H(\partial_1 K)$. So we only need to show  $\dim_H(\partial_1 K) < d_H$.

Let $S' = \{i\in S: \Psi_i\mcA\cap L_1 \text{ is a line segment}\}, S'' = \{i\in S:\Psi_i\mcA\cap L_1\text{ is a point}\}$ (it may happen that $S''=\emptyset$). Let $E_0 = \bigcup_{i\in S''}(\Psi_i\mcA\cap L_1), E_{n} = \bigcup_{i\in S'}\Psi_iE_{n-1}$ for $n\geq 1$. Then $\partial_1 K = F_1\cup F_2$, where $F_1$ is the unique attractor of $\{\Psi_i\}_{i\in S'}$, $F_2 = \bigcup_{n = 0}^{\infty}E_n$ is a countable set. By taking the open interval $L_1^{o}$, $F_1$  satisfies the open set condition, which means
  $$\sum_{i\in S'}\rho_i^{\dim_H(\partial_1 K)} = 1.$$

  However, by the non-trivial condition and the symmetry condition, there is at least one map $\Psi_{i_0}$ such that $i_0\notin S'$. So we have
  $\sum_{i\in S'}\rho_i^{s} < \sum_{i\in S}\rho_i^{s}$
  for any $s > 0$, which gives $\dim_H(\partial_1 K) < d_H$.
\end{proof}

By this proposition, we see that the Hausdorff dimension of $\bigcup_{w\in \Lambda_{*}}\Psi_w\partial K$ is strictly less than  $d_H$, which implies that for any $n\geq 0$, for almost every $x\in K$, there is only one $w\in \Lambda_n$ such that $x\in \Psi_wK$.\vspace{0.2cm}

\noindent\textbf{Basic notations of graph approximation sequences.}

Let $\Lambda\subset W_*$ be a partition.

(1). For $w\neq v\in \Lambda$, define $w\sim_{\Lambda} v$ if $\Psi_w K\cap \Psi_vK\neq \emptyset$, then $(\Lambda,\sim_{\Lambda})$ is a graph. For $w,v\in \Lambda$, we write $d_{\Lambda}(w,v)\in \mathbb{Z}_+$ as the \textit{graph distance} between $w,v$ in $\Lambda$. In particular, $\{(\Lambda_n,\sn)\}_{n\geq 0}$ is a graph approximation sequence of $K$ where $\sn$ is a short of $\sim_{\Lambda_n}$, and we write $d_n:= d_{\Lambda_n}$. For $k\geq 0,w\in \Lambda_n$, say
$$\mcN_{k}(w) = \big\{v\in \Lambda_n:d_n(w,v)\leq k\big\}$$
 the \textit{$k$-neighborhood} of $w$ in $\Lambda_n$. Write $\mcN_k(A) = \bigcup_{w\in A}\mcN_k(w)$ for $A\subset \Lambda_n.$

(2). For $A\subset \Lambda$, say $A$ is \textit{connected} if each pair $w\neq v$ in $A$ is connected by a \textit{path} in $A$, i.e. there exists a chain of cells $\{\tau^{(i)}\}_{i=0}^k\subset A$ with $\tau^{(0)}=w$, $\tau^{(k)}=v$ and $\tau^{(i)}\sim_{\Lambda}\tau^{(i-1)}$ for $1\leq i\leq k$. Call $k$ the \textit{length} of the path. For a connected $A\subset \Lambda$, let $l(A) = \{f:A\to \R\}$, and define a \textit{non-negative bilinear form} $\mcD_{\Lambda,A}$ on $l(A)$ as
$$
\mcD_{\Lambda,A}(f,g) = \sum_{w\sim_{\Lambda} v\in A}\big(f(w) - f(v)\big)\big(g(w) - g(v)\big),\quad\forall  f,g\in l(A).
$$
We write $\mcD_{\Lambda,A}(f): = \mcD_{\Lambda,A}(f,f)$ for short, and  $\mcD_{\Lambda} := \mcD_{\Lambda,\Lambda}, \mcD_{n,A} := \mcD_{\Lambda_n,A}, \mcD_{n} := \mcD_{\Lambda_n}$. In this way, $\mcD_{\Lambda,A}$ can be viewed as a \textit{quadratic form} on $l(A)$.

(3). We define $\mathscr{F}_{\Lambda} $ as the $\sigma$-field generated by $\big\{\Psi_wK:w\in \Lambda\big\}$. There is a natural bijection $\pi_{\Lambda}$ from $l(\Lambda)$ to $L^2(K,\mathscr{F}_{\Lambda},\mu)$ as $\pi_{\Lambda}(f)(x) = f(w)$ for any $x\in \Psi_w K$ and $w\in \Lambda$. Notice that we ignore the conflict definition on $\bigcup_{w\in\Lambda}\Psi_w\partial K$ since by Proposition \ref{prop26} it is just a null set of $\mu$. Write $\mathscr{F}_{n} := \mathscr{F}_{\Lambda_n},\pi_n := \pi_{\Lambda_n}$ for short.

(4). Since each $L^2(K,\mathscr{F}_{\Lambda},\mu)$ is a closed subspace of $L^2(K,\mu)$, we define $P_{\Lambda}$ as the orthogonal projection from $L^2(K,\mu)$ to $L^2(K,\mathscr{F}_{\Lambda},\mu)$. Write $P_n:= P_{\Lambda_n}$ for short.

(5). With the operators $P_{\Lambda}$ and $\pi_{\Lambda}$, we can shift the domain of $\mcD_{\Lambda}$ from $l(\Lambda)$ to $L^2(K,\mu)$ by define
$$
\mcD_{\Lambda}(f,g): = \mcD_{\Lambda}\big(\pi_{\Lambda}^{-1}\circ P_{\Lambda}f,\pi_{\Lambda}^{-1}\circ P_{\Lambda}g\big),\quad \forall f,g\in L^2(K,\mu),
$$
still using the notation $\mcD_{\Lambda}$ with a slight abuse of notation. Then $\mcD_{\Lambda}$ is a continuous non-negative bilinear form on $L^2{(K,\mu)}$.\vspace{0.2cm}

The following proposition is almost the same in \cite{KZ,CQ3}. We leave its proof in Appendix \ref{AppendixA}.

\begin{proposition}\label{prop27}
A regular polygon carpet $K$ defined in Definition \ref{def22} always satisfies the condition \textbf{(A1)}-\textbf{(A4)} below.

\textbf{(A1).} There is an open set $O$ such that $\Psi_iO\cap \Psi_jO = \emptyset$ for any $i\neq j\in S$, and $\Psi_iO \subset O$ for any $i\in S$.

\textbf{(A2).} $(\Lambda_n,\sn)$ is a connected graph for any $n\geq 0$.

\textbf{(A3).} There is a constant $c_0>0$ satisfying
$$\min\big\{\dist (\Psi_wK,\Psi_vK):w,v\in \Lambda_n,d_n(w,v) > 2\big\} \geq c_0\rho_*^n$$
for any $n\geq 1$.

\textbf{(A4).} There exists $m\geq 1$ such that $\mcB_{m}(w) \neq \partial \mcB_{m}(w)$ for any $w\in \Lambda_*$.

\end{proposition}

\section{Self-similar forms of Kusuoka and Zhou}\label{sec3}

In this section, we will follow Kusuoka-Zhou's strategy \cite{KZ}: first, we introduce three kinds of \textit{Poincare constants} $\lambda_m$, $R_m$ and $\sigma_m$ (with slight modifications due to the possible distinctness of contraction ratios, and also note that there is another kind of constants $\lambda_m^{(D)}$ in \cite{KZ}); second, under a resistance assumption, we prove the existence of self-similar Dirichlet forms on regular polygon carpets.

Throughout this paper, we will fix a regular polygon carpet $K$ with Hausdorff dimension $d_H$.

\begin{definition}[Poincare constants]\label{def31}
Let $m\geq 0$, $A$ be a non-empty subset of $\Lambda_m$ and $f\in l(\Lambda_m)$. We write
$$ [f]_A = \big(\sum_{w\in A}\rho_w^{d_H}\big)^{-1}\sum_{w\in A}\rho_w^{d_H} f(w) $$
as the (weighted) \emph{average} of $f$ on $A$.

(a). For $m\geq 1,n\geq 0, w\in \Lambda_n$, we define
$$ \lambda_m(w) = \sup\big\{\rho_*^{md_H}\cdot\sum_{v\in \mcB_m(w)}(f(v) - [f]_{\mcB_m(w)})^2:f\in l(\mcB_m(w)), \mcD_{n+m,\mcB_m(w)}(f) = 1\big\}.$$

And for $m\geq 1$, define
$$ \lambda_m = \sup\big\{\lambda_m(w):w\in \Lambda_{*}\big\}. $$

(b). For $m\geq1, A,B\subset \Lambda_m$ and $A\cap B = \emptyset$, we define the \emph{resistance} between $A,B$ as
$$ R_m(A,B) = \big(\inf\big\{ \mcD_{m}(f): f\in l(\Lambda_m),f|_A = 1,f|_B = 0\big\}\big)^{-1}.$$
 In particular, we write $R_m(w,B) = R_m(\{w\},B)$ and $R_m(w,v)=R_m(\{w\},\{v\})$ for short, and write $R_m(A,B)=0$ if $A\cap B\neq\emptyset$. We define
$$ R_m = \inf\big\{ R_{n + m}\big(\mcB_m(w), \mcB_m(\mcN_2^c(w))\big):n\geq 1,w\in \Lambda_{n}\big\},$$
where $\mcN_{k}^c(w): = \Lambda_{n} - \mcN_{k}(w)$ for $w\in \Lambda_n$.

(c). For $m\geq 1, n\geq 1, w\sn w'\in \Lambda_n$, we define
$$ \sigma_m(w,w') = \sup\big\{ \big([f]_{\mcB_m(w)} - [f]_{\mcB_m(w')}\big)^2:f\in l\big(\mcB_m(\{w,w'\})\big),\mcD_{n+m,\mcB_m(\{w,w'\})}(f) = 1\big\}.$$

And for $m\geq 1$, define
$$ \sigma_m = \sup\big\{\sigma_m(w,w'):n\geq 1,w\sn w'\in \Lambda_{n}\big\}. $$

\end{definition}

One of the important result in \cite{KZ} is the comparison of the above Poincare constants basing on the conditions \textbf{(A1)-(A4)}.

\begin{proposition}[\cite{KZ}, Theorem 2.1]\label{prop32}
  There is a constant $C>0$ such that
  \begin{equation}\label{eq31}
 C^{-2} \rho_*^{(d_H-2)m}\lambda_n \leq C^{-1}R_m\lambda_n \leq \lambda_{n + m} \leq C\lambda_n\sigma_m
  \end{equation}
  for any $m,n\geq 1$. In addition, all the constants $\lambda_m$, $R_m$ and $\sigma_m$, $m\geq 1$ are positive and finite.
\end{proposition}

\noindent\textbf{Remark.} The first inequality in (\ref{eq31}) is exactly that $R_m\geq C^{-1}\rho_*^{(d_H-2)m}$. It was extensively explored by Kigami on more general compact metric spaces, see  \cite[Lemma 4.6.15]{ki5} for a generalized version. In particular, this inequality implies that the process is recurrent since $d_H<2$ by the non-trivial condition. On the other hand, it is not hard to verify that the second and third inequalities in (\ref{eq31}) implies
\begin{equation}\label{eq32}
C^{-1}R_m\leq \lambda_m\leq C\sigma_m
\end{equation}
for some $C>0$ independent of $m$. In fact, by taking $m=1$, we see $\lambda_n\asymp\lambda_{n+1}$, then taking $n=1$, we see (\ref{eq32}). \vspace{0.2cm}

The following is another important observation.

\begin{proposition}[{\cite[Theorem 7.2]{KZ}}]\label{prop33}
  There is $C >0 $ such that
  $$\big(f(v) - [f]_{\mcB_{m}(w)}\big)^2 \leq C\lambda_m\mcD_{n + m,\mcB_{m}(w)}(f)$$
   for any $m\geq 1, n\geq 0, w\in \Lambda_n, v\in \mcB_{m}(w)$ and $f\in l(\mcB_m(w))$.
\end{proposition}

Since the proof of Proposition \ref{prop32} and \ref{prop33} are essentially  same as in \cite{KZ} with suitable modification due to the distinctness of contraction ratios, we leave them in Appendix \ref{AppendixB}.

Following Kusuoka-Zhou's strategy, we will need another inequality:\vspace{0.2cm}

\noindent\textbf{(B).} There is a constant $C>0$ such that $\sigma_m\leq C R_m$ for any $m\geq 1$.\vspace{0.2cm}

Combining Proposition \ref{prop32} and  \textbf{(B)}, by a routine argument \cite{KZ}, there is  $0 < r \leq \rho_*^{2 - d_H} < 1$ such that
  $R_m\asymp \lambda_m\asymp \sigma_m \asymp r^{-m}. $ Kusuoka-Zhou's approach for SC is analytic except the verification of \textbf{(B)} (see  \textbf{(B2)} in \cite{KZ},
in a slightly different version). They achieved \textbf{(B)} by a probabilistic ``Knight moves'' method due to Barlow and Bass \cite{BB}, see   \cite[Theorem 7.16]{KZ}. Two of the authors  fulfilled this gap in a recent work \cite{CQ3}. In particular, they developed a pure analytic method for \textbf{(B)} for a more general planar (square) carpets USC. In some sense, USC are more flexible in geometry as cells except those along the boundary are allowed to live off the grids. Due to the  possible irrationally ramified situation for USC, the method in  \cite{CQ3} also non-trivially extends Barlow and Bass result \cite{BB} since the last one heavily depends on the local symmetry of SC. In this paper, we will extend the method to regular polygon carpets.

\subsection{Existence of self-similar forms under assumption \textbf{(B)}}\label{subsec31}

Let $K,\mu$ be the same as before. Recall that a Dirichlet form on $L^2(K,\mu)$ is a non-negative quadratic form on $L^2(K,\mu)$ which is densely defined, closed and satisfies the Markov property. Please refer to \cite{FOT} for the general definition of Dirichlet forms and some necessary  properties. In our situation, we only focus on the recurrent case.

\begin{definition}\label{def34}
  A Dirichlet form $(\mathcal{E,F})$ on $L^2(K,\mu)$ is called  \emph{self-similar} if

  \noindent(a). $f\in \mathcal{F}$ implies $f\circ \Psi_i\in \mathcal{F}$ for each $i\in S$;

  \noindent(b). $f\in C(K)$ and $f\circ\Psi_i\in\mathcal{F}$ for each $i\in S$ implies $f\in \mathcal{F}$;

  \noindent(c). the self-similar identity holds, i.e.
  $$\mathcal{E}(f) = \sum_{i\in S}r_i^{-1}\mathcal{E}(f\circ\Psi_i),\quad\forall f\in \mathcal{F},$$
  where $0<r_i<1, i\in S$ are called \emph{renormalization factors}.
\end{definition}

  In Kusuoka-Zhou's original strategy \cite{KZ} (where they dealt with the case that $\rho_i$'s are the same), the existence of $(\mathcal E,\mathcal F)$ (taking $r_i=r$ for all $i\in S$) follows by two steps: first, they construct a local regular Dirichlet form $(\bar{\mathcal E},\mathcal F)$ which is a  limit of $ r^{-n}{\mathcal D}_n$; then, they construct $(\mathcal E,\mathcal F)$ to be a limit of the Ces\`{a}ro  mean of $(\sum_{w\in \Lambda_n}r^{-n}\bar{\mathcal E}\circ \Psi_w,\mathcal F)$. We still follow this strategy but with little adjustment, since the renormalization factors $r_i$'s in our case may be distinct.
\vspace{0.2cm}

First, there is a limit Dirichlet form $(\bar{\mathcal E},\mathcal F)$ on $L^2(K,\mu)$.
\begin{proposition}\label{prop35}
Assume \textbf{(B)}.  There is a regular Dirichlet form $(\bar{\mathcal{E}},\mathcal{F})$ on $L^2(K,\mu)$ with $r\in (0,1)$, such that   for some constant $C > 0$,
  $$C^{-1}\sup_{n\geq 1}r^{-n}\mcD_n(f) \leq \bar{\mathcal{E}}(f) \leq C\liminf_{n\to\infty}r^{-n}\mcD_n(f),\quad \forall f\in \mathcal F.$$
 In addition, $\mathcal{F}\subset C(K)$.

 Moreover, denote $\theta = \frac{\log r}{\log \rho_*}$, then there is a constant $C'>0$ such that
  $$|f(x) - f(y)|^2 \leq C'\bar{\mathcal{E}}(f)\cdot |x - y|^{\theta}, \quad\forall x,y\in K,\forall f\in \mathcal{F}.$$
\end{proposition}

This proposition essentially follows from a combination of Theorem 5.4, 7.2 in \cite{KZ}. For the convenience of the readers, we provide its proof in Appendix \ref{AppendixC} but alternatively follows from \cite[Theorem 3.8]{CQ3} by two of the authors by using the technique of $\Gamma$-convergence. Note that the method of $\Gamma$-convergence in the construction of Dirichlet forms on self-similar sets was also used by Grigor'yan and Yang in \cite{GY}.\vspace{0.2cm}

Next, we transform $(\bar{\mathcal E},\mathcal F)$ into a self-similar one.

\begin{theorem}\label{thm36}
 Assume \textbf{(B)}.  There is a local regular self-similar form $(\mathcal{E},\mathcal{F})$ on $L^2(K,\mu)$ with $r_i=\rho_i^\theta,i\in S$, such that
  $$\mcE(f) = \sum_{i\in S}r_i^{-1}\mcE(f\circ \Psi_i),\quad \forall f\in \mathcal F.$$
   In addition, $\mathcal{F}\subset C(K)$ and $\mcE(f)\asymp \bar{\mcE}(f), \forall f\in \mcF$.
Here $\bar{\mcE}$ and $\theta$ are same as that in Proposition \ref{prop35}.

\end{theorem}

\begin{proof}
Let $(\bar{\mcE},\mcF)$,  $r$ be the same in Proposition \ref{prop35}. For $w\in W_*$, denote $r_w = r_{w_1}\cdots r_{w_{|w|}}$. Obviously, $r_w\asymp r^n$ for any $w\in \Lambda_n, n\geq 0$. For any partition $\Lambda\subset W_*$, denote $\mcF_{\Lambda} = \big\{f\in C(K): f\circ \Psi_w\in \mcF, \forall w\in \Lambda\big\}$ and
$$\tilde{\mcE}_{\Lambda}(f) = \sum_{w\in \Lambda}r_w^{-1}\bar{\mcE}(f\circ \Psi_w),\qquad \forall f\in \mcF_{\Lambda}.$$
In particular, for $n\geq 0$, denote $\mcF_n = \mcF_{\Lambda_n}, \mcF'_{n} = \mcF_{W_n}, \tilde{\mcE}_n = \tilde{\mcE}_{\Lambda_n}, \tilde{\mcE}'_n = \tilde{\mcE}_{W_n}$ for short.

For each $n\geq 0$, let $\Lambda'_{n} = \big\{w\in W_*: \rho_{\sigma(w)} > \rho_*^{n}\big\}$. Then for any $n\geq 1$, $\Lambda_n$ is the ``finest" partition of $W_*$ contained in $\Lambda'_n$, and  $\Lambda_n'\cap \Lambda_{n+1} = \emptyset$. We divide the proof into four claims. \vspace{0.2cm}

\noindent \textit{Claim 1. For each $m\geq 1$, there is a constant $c_m>0$ such that for any $n\geq m$, for any partition $\Lambda\subset \Lambda'_n\setminus \Lambda'_{n-m}$ and $f\in L^2(K,\mu)$, it satisfies $c_m^{-1}\mcD_{n-m}(f)\leq \mcD_{\Lambda}(f) \leq c_m\mcD_{n}(f)$.}\vspace{0.2cm}

Let $w\sim_{\Lambda} v\in \Lambda$, then there is constant $c_m'>0$ independent of $n,w,v,\Lambda$, such that
$$\begin{aligned}
\big(\pi_{\Lambda}^{-1}f(w) - \pi_{\Lambda}^{-1}f(v)\big)^2
=& \bigg(\sum_{w'\in w\cdot w^{-1}\cdot\Lambda_n}\sum_{v'\in v\cdot v^{-1}\cdot\Lambda_n}\frac{\mu(\Psi_{w'}K)}{\mu(\Psi_wK)}\frac{\mu(\Psi_{v'}K)}{\mu(\Psi_vK)}\big(\pi_n^{-1} f(w') - \pi_n^{-1} f(v')\big) \bigg)^2\\
\leq & \sum_{w'\in w\cdot w^{-1}\cdot\Lambda_n}\sum_{v'\in v\cdot v^{-1}\cdot\Lambda_n}\big(\pi_n^{-1} f(w') - \pi_n^{-1} f(v')\big)^2\\
\leq & c'_m\mcD_{n,w\cdot w^{-1}\cdot\Lambda_n\cup v\cdot v^{-1}\cdot\Lambda_n}(f).
\end{aligned}$$
By a summation over pairs $w\sim_{\Lambda} v\in \Lambda$, the right side of the claim follows. A similar argument gives the other side.\vspace{0.2cm}

\noindent \textit{Claim 2. $\mcF = \mcF_{\Lambda}$ for any partition $\Lambda$.} In particular, $\mcF  = \mcF_n = \mcF_n'$ for any $n\geq 0$.\vspace{0.2cm}

$``\mcF \subset\mcF_{\Lambda}$''.  For any $w\in \Lambda$, choose $m\geq 1$ so that $w\in \Lambda_m'\setminus \Lambda'_{m-1}.$ It's easy to check that for any $k \gg m,$ it satisfies $w^{-1}\cdot\Lambda_k \subset \Lambda_{k-m+1}'\setminus \Lambda_{k-m - 2}'$, which gives $\mcD_{k}(f) \geq \mcD_{w^{-1}\cdot \Lambda_k}(f\circ \Psi_w) \geq c_3^{-1}\mcD_{k-m -2}(f\circ \Psi_w)$ for any $f\in \mcF$ by Claim 1. Then by Proposition \ref{prop35}, there is constant $C_1 > 0$ such that
$$\bar{\mcE}(f\circ \Psi_w) \leq C_1\liminf_{k\to \infty}r^{-k}\mcD_{k}(f\circ \Psi_w) \leq C_1c_3\liminf_{k\to \infty}r^{-k}\mcD_{k + m + 2}(f) \leq C_1^2c_3r^{m+2}\bar{\mcE}(f) < +\infty,$$
thus $f\in \mcF_{\Lambda}$ follows.

``$\mcF\supset \mcF_n$, $n\geq 0$''. We first prove $r^{-n}\mcD_n(f) \leq C_2\tilde{\mcE}_n(f),\forall f\in \mcF_n$ for some $C_2 > 0$ independent of $n$. For each pair $w\sn w'\in\Lambda_n$, take $x\in \Psi_wK\cap \Psi_{w'}K$. Then for any $f\in \mcF_n, v \in \{w,w'\}$, we have $|\pi^{-1}_n \circ P_n f(v) - f(x)| \leq \int_{K}|f\circ \Psi_v(y) - f(x)|\mu(dy) \lesssim \sqrt{\bar{\mcE}(f\circ \Psi_v)}$ by Proposition \ref{prop35}, which gives $|\pi^{-1}_n \circ P_n f(w) - \pi^{-1}_n \circ P_n f(w')|^2 \lesssim \bar{\mcE}(f\circ \Psi_w) + \bar{\mcE}(f\circ \Psi_{w'})$, thus the desired estimate follows by a summation over $w\sn w'\in\Lambda_n$.

Next, it's easy to check that, for any $m\geq 0,k \geq 1$, it satisfies $\Lambda_m\cdot \Lambda_k\subset \Lambda_{m + k + 1}'\setminus \Lambda_{m+k-1}',$ so by Proposition \ref{prop35} and  Claim 1, there is constant $C_3 > 0$ such that for any $g\in \mcF,$ $$\tilde{\mcE}_m(g) \leq C_3r^{-m}\sum_{w\in \Lambda_m}\liminf_{k\to\infty}r^{-k}\mcD_k(g\circ \Psi_w)\leq C_3c_2r^{-m}\liminf_{k\to\infty}r^{-k}\mcD_{m + k + 1}(g) \leq C_1C_3c_2r\bar{\mcE}(g),$$
i.e. $\sup_{m\geq 0}\big\{\tilde{\mcE}_m(g)\big\} \leq C_4\bar{\mcE}(g)$, where $C_4 = C_1C_3c_2r$. Then
$$\begin{aligned} f\in \mcF_n \Rightarrow &  f\circ \Psi_w\in \mcF, \forall w\in \Lambda_n
\Rightarrow  \sup\{\tilde{\mcE}_m(f\circ \Psi_w): m\geq 0, w\in \Lambda_n\} < \infty\\
\Rightarrow & \sup\{\tilde{\mcE}_m(f): m\geq n\} < \infty
\Rightarrow  \sup\{r^{-m}\mcD_m(f): m\geq n\} < \infty \Rightarrow f\in \mcF.
\end{aligned}$$

``$\mcF\supset \mcF_{\Lambda}$''.   For any general partition $\Lambda$, take large $n$ so that $\Lambda \subset \Lambda'_n$. Then $f\in \mcF_{\Lambda}$ gives $f\circ \Psi_w\in \mcF$ for any $w\in \Lambda$, thus $f\circ \Psi_{wv}\in \mcF$ for any $wv\in \Lambda_n$, i.e. $f\in \mcF_n$. This together with the last paragraph gives that $\mathcal{F} \supset \mathcal{F}_{\Lambda}$.\vspace{0.2cm}

\noindent\textit{Claim 3. There is constant $C_5 > 0$ such that $\sum_{m = 1}^{n}\tilde{\mcE}'_m(f) \leq C_5\sum_{m = 0}^{n}\tilde{\mcE}_m(f),\forall n\geq 1,f\in \mcF$.}\vspace{0.2cm}

First, we prove there is constant $C_6 > 0$ such that $\tilde{\mcE}_{\Lambda}(f) \leq C_6\tilde{\mcE}_{m-1}(f)$ for any $m\geq 1$, partition $\Lambda\subset \Lambda'_m\setminus \Lambda'_{m-1}, f\in \mcF$. Note that for any $w\in \Lambda_{m-1}$, it satisfies $w^{-1}\cdot\Lambda\cdot\Lambda_k\subset \Lambda'_{k+2}\setminus \Lambda'_k$ for all $k\geq 1$. Then for any $f\in \mcF$, we have
$$\begin{aligned}
\tilde{\mcE}_{\Lambda}(f) =& \sum_{w\in \Lambda_{m-1}} r_w^{-1}\sum_{v\in w^{-1}\cdot \Lambda}r_v^{-1}\bar{\mcE}(f\circ \Psi_{w}\circ \Psi_{v})\\
\leq &C_7\sum_{w\in\Lambda_{m-1}}r_w^{-1}\sum_{v\in w^{-1}\cdot \Lambda}\liminf_{k\to\infty}r^{-k}\mcD_k(f\circ\Psi_w\circ\Psi_v)\\
\leq &C_7c_2\sum_{w\in \Lambda_{m-1}}r_w^{-1}\liminf_{k\to\infty}r^{-k}\mcD_{k+2}(f\circ \Psi_w)
\leq C_1C_7c_2r^2\tilde{\mcE}_{m-1}(f)
\end{aligned}$$
for some constant $C_7>0$, giving the desired estimate.

Next, since we can choose at most $M = [\frac{\log\rho_*}{\log\rho^*}] + 1$ partitions $\tilde{\Lambda}_i\subset \Lambda_m'\setminus \Lambda_{m-1}',1\leq i \leq M$ such that $\Lambda_m'\setminus \Lambda_{m-1}' = \bigcup_{i = 1}^{M}\tilde{\Lambda}_i $. Then by the estimate above, $\sum_{w\in \Lambda'_m\setminus\Lambda'_{m-1}}r_w^{-1}\bar{\mcE}(f\circ \Psi_w) \leq C_6M\tilde{\mcE}_{m-1}(f)$ for any $m\geq 1,f\in \mcF$. Thus $\sum_{m = 1}^{n}\tilde{\mcE}'_m(f) \leq \sum_{m = 1}^{n}\sum_{w\in \Lambda'_m\setminus \Lambda'_{m-1}}r_w^{-1}\bar{\mcE}(f\circ\Psi_w) \leq C_5\sum_{m = 0}^{n-1}\tilde{\mcE}_m(f),$ where $C_5 = C_6M$. The claim follows.\vspace{0.2cm}

Take a countable $\bar{\mcE}_1$-dense set $\hat{\mcF} \subset \mcF$ with $\{f\circ\Psi_w:f\in\hat{\mcF},w\in W_*\}\subset \hat{\mcF}$, where $\bar{\mcE}_1(\cdot) = \bar{\mcE}(\cdot) + \|\cdot \|_{L^2(K,\mu)}^2$. By Claim 3 and the fact that $\sup_{m\geq 0}\tilde{\mathcal{E}}(f) \lesssim \overline{\mathcal{E}}(f)$, we can choose a strictly increasing $\{n_l\}_l$ such that $\mcE(f):=\lim_{l\to\infty}\frac{1}{n_l}\sum_{m = 1}^{n_l}\tilde{\mcE}'_m(f) < \infty $ exists for any $f\in \hat{\mcF}$. \vspace{0.2cm}

\noindent \textit{Claim 4. $\mcE(f)\asymp \bar{\mcE}(f)$ for any $f\in \hat{\mcF}$.}\vspace{0.2cm}

Note that $\sum_{m = 1}^{n_l}\tilde{\mcE}'_{m}(f) \geq \sum_{m = 1}^{[\frac{n_l}{M}]}\tilde{\mcE}_m(f)$.
Combining this with Claim 3, it's enough to prove
$$\bar{\mcE}(f)\lesssim \liminf_{n\to\infty}\frac{1}{n}\sum_{m = 1}^{n}\tilde{\mcE}_m(f) \leq \limsup_{n\to\infty}\frac{1}{n}\sum_{m = 1}^{n}\tilde{\mcE}_m(f) \lesssim \bar{\mcE}(f),\qquad \forall f\in \hat{\mcF}.$$
Since in the proof of Claim 2, we have $r^{-n}\mcD_n(f)\lesssim \tilde{\mcE}_n(f) \lesssim \bar{\mcE}(f)$ for any $f\in \mcF,n\geq 0$, the above estimate follows by taking limit, which gives Claim 4. \vspace{0.2cm}

By Claim 4, we can continuously extend $\mcE$ to $\mcF$. By Proposition \ref{prop35}, $\mcE$ is a regular Dirichlet form on $L^2(K,\mu)$, since the extension keeps $\mcE(tf) = t^2\mcE(f), \mcE(f+g) + \mcE(f-g)=2\mcE(f)+2\mcE(g)$ and $\mcE((f\wedge 1)\vee 0) \leq \mcE(f)$ for any $t\in \mathbb{R}$ and $f,g\in \mcF$.

By the construction, for any $f\in\hat{\mcF}$,
$$\mcE(f) = \lim_{l\to\infty}\frac{1}{n_l}\sum_{m = 1}^{n_l}\tilde{\mcE}'_m(f) = \lim_{l\to\infty}\frac{1}{n_l}\sum_{m = 0}^{n_l-1}\sum_{i\in S}r_i^{-1}\sum_{w\in W_{m}}r_w^{-1}\bar{\mcE}(f\circ\Psi_i\circ\Psi_w) = \sum_{i\in S}r_i^{-1}\mcE(f\circ\Psi_i),$$ giving the self-similar property over $\mcF$ by continuously extension. The local property of $(\mcE,\mcF)$ follows immediately by the self-similar property.

\end{proof}

\noindent \textbf{Remark.}  Due to Proposition \ref{prop35} and Theorem \ref{thm36}, it remains to verify condition \textbf{(B)} to achieve our main goal, Theorem \ref{th1}.

\section{Resistance estimates between sets}\label{sec4}
In the rest of this paper, we will prove condition \textbf{(B)} for two classes of polygon carpets: all the perfect polygon carpets (Section \ref{sec5}) and some bordered polygon carpets (Section \ref{sec6}, \ref{sec7}). Before proceeding, we present some observations about resistances. \vspace{0.2cm}

We first restate Proposition \ref{prop33} in terms of oscillations of functions.

\begin{definition}\label{def41}
	For any non-empty set $A$ and $f\in l(A)$, we denote
	\[\osc|_A(f) = \sup\big\{f(x) - f(y):x,y\in A\big\}.\]
	For $m\geq 1$, define
	\[\delta_m := \sup\big\{\osc|^2_{\mcB_m(w)}(f):w\in \Lambda_*,f\in l\big(\mathcal{B}_m(w)\big),\mcD_{\|w\|+m,\mathcal{B}_{m}(w)}(f)=1\big\}.\]
\end{definition}

\begin{lemma}\label{lemma42}
	$\lambda_m\asymp\sigma_m\asymp \delta_m$ for $m\geq 1$.
\end{lemma}
\begin{proof}
	First, for any $w\in \Lambda_*$, $m\geq 1$, and $f\in l\big(\mcB(w)\big)$ with $\mcD_{\|w\|+m,\mcB_m(w)}(f)=1$, it follows from Proposition \ref{prop33} that $\big(f(v)-[f]_{\mcB_m(w)}\big)^2\leq C_1\lambda_m$ for some constant $C_1>0$ for any $v\in \mcB_m(w)$, and so $\osc|_{\mcB_m(w)}(f)\leq 2\sqrt{C_1\lambda_m}$. This gives that $\delta_m\leq 4C_1\lambda_m$.
	
	Next, for $m\geq 1$, pick a pair $w\sn v\in \Lambda_n, n\geq 1$, and a function $f\in l\big(\mcB_{m}(\{w,v\})\big)$ such that $\mcD_{n + m,\mcB_{m}(\{w,v\})}(f) = 1$ and $[f]_{\mcB_m(w)}-[f]_{\mcB_m(v)}\geq\frac{1}{2}{\sqrt{\sigma_m}}$. Pick $\iota\in \mcB_m(w)$ and $\kappa\in \mcB_m(v)$ so that $\iota\stackrel{n+m}{\sim}\kappa$, then
	\[\big|f(\iota) - f(\kappa)\big|\leq \sqrt{\mcD_{n + m,\mcB_{m}(\{w,v\})}(f)}=1.\]
	Hence,
	\[\begin{aligned}
	\frac 12\sqrt{\sigma_m}\leq [f]_{\mcB_m(w)}-[f]_{\mcB_m(v)}&\leq \osc|_{\mcB_{m}(\{w,v\})}(f)	\\
	&\leq \osc|_{\mathcal{B}_m(w)}(f)+\osc|_{\mathcal{B}_m(v)}(f)+\big|f(\iota)-f(\kappa)\big|\leq2\sqrt{\delta_m}+1.
	\end{aligned}\]
	By Proposition \ref{prop32}, $\sigma_m\gg1$ for large $m$. This gives that there exists a constant $C_2>0$ independent of $m$, such that $\delta_m\geq C_2 \sigma_m$.
	
	Finally,  the above estimates give that $C^{-1}\sigma_m \leq \delta_m \leq C\lambda_m$ for some constant $C>0$. Combing this with formula (\ref{eq32}), we immediately see that $\lambda_m\asymp\sigma_m\asymp \delta_m$ for $m\geq 1$.
\end{proof}

The following lemma will help us find a lower bound estimate of resistances between sets with only resistance estimates between points.

\begin{lemma}\label{lemma43}
Let $m\geq 1$ and $A,B\subset \Lambda_m$. For each $C>0$, there is $C'>0$ (depending only on $C$ and $K$) such that if $R_m(w,v)\geq C\sigma_m,\forall w\in A,  v\in B$, then
\[R_m(A,B)\geq C'\sigma_m.\]
\end{lemma}
\begin{proof}
It suffices to consider large $m$, so we can choose $n>0$ independent of $m$ such that $\delta_{m-n}<\frac{C}{9}\sigma_m$ for each $m>n$ according to Proposition \ref{prop32} and Lemma \ref{lemma42}.

We introduce $A',B'\subset \Lambda_n$ as ``covers'' of $A,B$ by
\[A'=\{w'\in \Lambda_n:\mcB_{m-n}(w')\cap A\neq \emptyset\},\ B'=\{v'\in \Lambda_n:\mcB_{m-n}(v')\cap B\neq \emptyset\}.\]
One can see that
\[R_m\big(\mcB_{m-n}(w'),\mcB_{m-n}(v')\big)\geq \frac{C}{9}\sigma_m,\quad \forall w'\in A',v'\in B'.\]
In fact, one choose $w\in \mcB_{m-n}(w')\cap A$ and $v\in \mcB_{m-n}(v')\cap B$, and define $f\in l(\Lambda_m)$ such that $f(w)=1,f(v)=0$ and $\mcD_m(f)=R^{-1}_m(w,v)\leq (C\sigma_m)^{-1}$. Then by the choice of $n$, one has $f|_{\mcB_{n-m}(w')}\geq \frac{2}{3}$ and $f|_{\mcB_{n-m}(v')}\leq \frac{1}{3}$. The estimate of $R_m\big(\mcB_{m-n}(w'),\mcB_{m-n}(v')\big)$ follows immediately.

By the above estimate, for each pair $w'\in A',v'\in B'$, we can find $g_{w',v'}\in l(\Lambda_m)$ such that $0\leq g_{w',v'}\leq 1$ and
\[g_{w',v'}|_{\mcB_{m-n}(w')}=1,\ g_{w',v'}|_{\mcB_{m-n}(v')}=0,\ \mcD_m(g_{w',v'})\leq 9(C\sigma_m)^{-1}.\]
Define $g\in l(\Lambda_m)$ by
\[g(w)=\max_{w'\in A'}\min_{v'\in B'}g_{w',v'}(w),\quad \forall w\in \Lambda_m. \]
Then it is direct to check that $g|_{\mcB_{m-n}(A')}=1$ and $g|_{\mcB_{m-n}(B')}=0$. Finally, noticing that $A\subset \mcB_{m-n}(A')$, $B\subset \mcB_{m-n}(B')$ and $\mcD_m(g)\leq \sum_{w'\in A',v'\in B'}\mcD_m(g_{w',v'})\leq (\#\Lambda_n)^2\cdot 9(C\sigma_m)^{-1}$, the lemma follows.
\end{proof}

Now, we state an important consequence of Lemma \ref{lemma43}, whose application Proposition \ref{prop45} will be the key tool in Section \ref{sec6}.

\begin{proposition}\label{prop44}
	There is a function $\eta:(0,\infty)\to (0,\infty)$ (depending only on $K$) such that $\lim\limits_{c\to 0}\eta(c)=0$,
	and in addition, for each $m\geq 1$ and $w,v\in l(\Lambda_m)$ satisfying $R_m(w,v)\leq c\sigma_m$, one can find a connected subset $A\supset \{w,v\}$ of $\Lambda_m$ such that  $R_m(w',v')\leq\eta(c)\sigma_m,\forall w',v'\in A$.
\end{proposition}
\begin{proof}
Instead of considering connected $A$ directly, we consider \textit{cut sets} that separate $w,v$. More precisely, we say a set $B$ separates $w,v$ if $\Lambda_m\setminus B$ is disconnected and $w,v$ belongs to different connected components of $\Lambda_m\setminus B$. Then, one can see that it suffices to take
\[\eta(c)= 4\gamma(c)+5c,\]
where
\[\begin{aligned}
	\gamma(c):=\sup\big\{c': \ &\text{there exists $m\geq 1$ and $w,v\in \Lambda_m$ such that }R_m(w,v)\leq c\sigma_m,\\
	&\text{\qquad\qquad\qquad and }\{w'\in \Lambda_m:R_m(w',\{w,v\})>c'\sigma_m\}\text{ separates } w,v\big\}.
\end{aligned}\]
In fact, for any $m\geq 1$ and $w,v\in \Lambda_m$ satisfying $R_m(w,v)\leq c\sigma_m$, if we let $B=\big\{w'\in \Lambda_m:R_m(w',\{w,v\})>\big(c+\gamma(c)\big)\sigma_m\big\}$, by the definition of $\gamma(c)$, we can see that $B$ doesn't separate $w,v$. Thus, we can find a connected component $A$ of $\Lambda_m\setminus B$ containing both $w,v$, and it is not hard to see that $R_m(w',v')\leq 4\big(c+\gamma(c)\big)\sigma_m+c\sigma_m$ for any $w',v'\in A$.

It remains to show $\gamma(c)\to 0$ as $c\to 0$. We will show that for each $\varepsilon>0$, there exists $\delta>0$ such that $\gamma(\delta)\leq\varepsilon$, which will finish the proof since $\gamma$ is a non-decreasing function. Let $m\geq 1$ and $w,v\in \Lambda_m$, and assume that $B=\big\{w'\in \Lambda_m:R_m(w',\{w,v\})>\varepsilon\sigma_m\big\}$ seperates $w,v$. Then, by Lemma \ref{lemma43}, there exists $\delta$ depending only on $\varepsilon$ (and $K$) such that
\[R_m\big(B,\{w,v\}\big)\geq \delta\sigma_m. \]
Hence, we can find a function $f\in l(\Lambda_m)$ such that
\[f|_B=1,\ f|_{\{w,v\}}=0\text{ and }\mcD_m(f)\leq (\delta\sigma_m)^{-1}.\] Let $A_v$ be the component of $\Lambda_m\setminus B$ which contains $v$, and define $g\in l(\Lambda_m)$ as
\[g(w')=\begin{cases}
	f(w'),&\text{ if }w'\in \Lambda_m\setminus A_v,\\
	1,&\text{ if }w'\in A_v.
\end{cases}\]
Then, we have $g(w)=0,g(v)=1$ and $\mcD_m(g)<(\delta\sigma_m)^{-1}$. Hence, \[R_m(w,v)>\delta\sigma_m.\]
Notice that the above argument holds for any $m\geq 1$ and $w,v\in \Lambda_m$ such that $B=\big\{w'\in \Lambda_m:R_m(w',\{w,v\})>\varepsilon\sigma_m\big\}$ seperates $w,v$, so we conclude that $\gamma(\delta)\leq\varepsilon$.
\end{proof}

\subsection{An application of Proposition \ref{prop44}} In this subsection, we consider an application of Proposition \ref{prop44}. We will use the symmetry of the polygon carpets. \vspace{0.2cm}

\noindent\textbf{Half-sides and related notations.}

(1). We define $S_1 := S_0\cup\{i + 1/2:i\in S_0\}$ equipped with a distance $d_{S_1}$ as
\[d_{S_1}(i,j) := |i - j|\wedge \big(N_0 - |i - j|\big) \text{ for any }i,j\in S_1.\] Denote $\dist_{S_1}(I_1,I_2) = \min\big\{d_{S_1}(i,j):i\in I_1,j\in I_2\big\}$ for $I_1,I_2\subset S_1$.  In particular, write $\dist_{S_1}(i,I) = \dist_{S_1}(\{i\},I)$ for short. For $i\in \frac12\mathbb Z:=\{j/2: j\in\mathbb Z\}$, we identify $i$ with the unique index $\tilde i\in S_1$ satisfying $(i-\tilde i)/N_0\in\mathbb Z$.

(2). For $i\in S_0$, denote the midpoint of $L_i$ as $q_{i + 1/2}$. We define $L'_{i} = \overline{q_i,q_{i + 1/2}}$ for $i\in S_1$ with $q_{N_0 + 1} = q_1$, the \textit{half-sides} of $\mcA$. Clearly, $L_i = L'_i\cup L'_{i + 1/2}$ for each $i\in S_0$.

(3). For $i\in S_0$, $m\geq 0$, we denote
\[\partial_i\Lambda_m = \{v\in \Lambda_m:\Psi_vK\cap L_i\neq \emptyset\}\]
and
\[\begin{cases}
	\partial'_i\Lambda_m = \{v\in \partial_i\Lambda_m:\dist(\Psi_vK,q_i) \leq \dist(\Psi_vK,q_{i+1})\},\\
	\partial'_{i+1/2}\Lambda_m = \{v\in \partial_i\Lambda_m:\dist(\Psi_vK,q_i) \geq \dist(\Psi_vK,q_{i+1})\}.	
\end{cases}\]
There exists $m_0\geq 1$ such that for any $m \geq m_0$, $ i,j\in S_1$ with $d_{S_1}(i,j)\geq 1$, $\partial'_i\Lambda_m \cap \partial'_j\Lambda_m = \emptyset$.\vspace{0.2cm}

We also need the following notation about symmetry. \vspace{0.2cm}

\noindent\textbf{Symmetry on $W_*$.}

For each $\Gamma\in \msG$, denote $\Gamma^*$ the induced symmetry of $\Gamma$ on $W_*$, i.e. $\Gamma^*:W_*\to W_*$ satisfying
$$\Psi_{\Gamma^*(w)}K = \Gamma(\Psi_wK),\quad \forall w\in W_*.$$
In particular, denote $\Gamma_{i,j}^*,\Gamma_i^*$ on $W_*$ induced by $\Gamma_{i,j},\Gamma_i$ respectively, for $i\neq j\in S_0$.

By the symmetry, we have $R_m(\partial'_i\Lambda_m,\partial'_j\Lambda_m) = R_{m}\big(\Gamma^*(\partial'_i\Lambda_m\big),\Gamma^*(\partial'_j\Lambda_m)\big)$ for any $\Gamma\in \msG$, $m \geq m_0$, and $i,j\in S_1$.

\begin{proposition}\label{prop45}
Let $m\geq m_0$ and $i,j\in S_1$ with $d_{S_1}(i,j)\geq 1$. For each $w\in \partial'_i\Lambda_m$ and $v\in \partial'_j\Lambda_m$, one can find a connected subset $A\subset \Lambda_m$ such that
\[w\in A=\Gamma^*(A),\quad\forall \Gamma\in\msG,\]
and
\[R_m(w',v')\leq 2N_0\eta\big(\frac{R_m(w,v)}{\sigma_m}\big)\cdot\sigma_m,\ \forall w',v'\in A, \]
where $\eta$ is the same function introduced in Proposition \ref{prop44}.

In particular, $R_m(w,\Gamma^* w)\leq 2N_0\eta\big(R_m(w,v)/\sigma_m\big)\sigma_m$ for any $\Gamma\in \mathscr{G}$.
\end{proposition}
\begin{figure}[htp]
	\includegraphics[width=4.9cm]{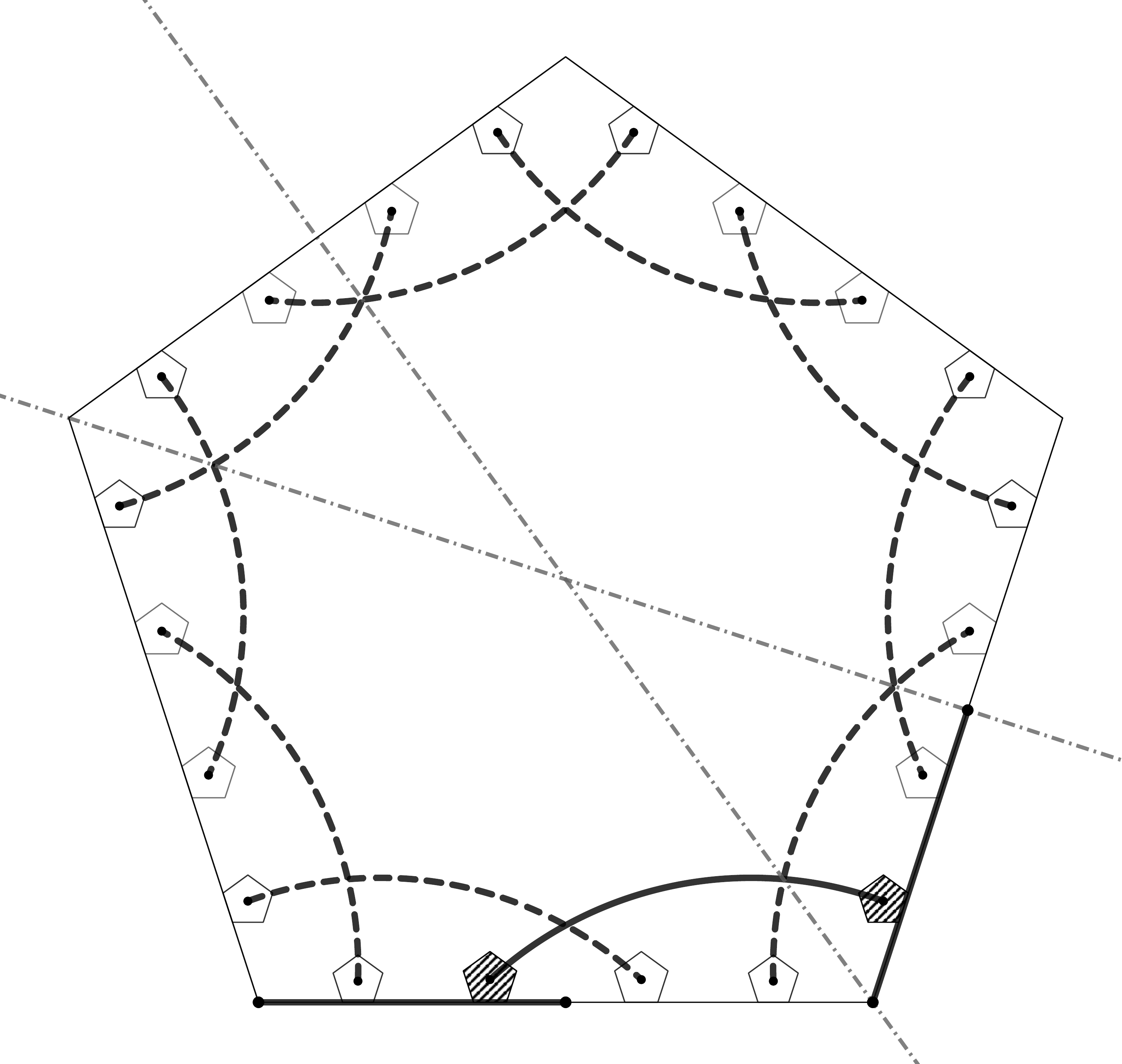}
		\caption{An illustration of the connected sets $A$ and $A'$ in $\Lambda_m$.}
		\begin{picture}(0,0)
		\put(8,55){$A'$}\put(45,55){$w$}\put(-12,38){$v$}\put(-52,160){$\Gamma^*_{k_1,l_1}$}\put(-89,123){$\Gamma^*_{k_2,l_2}$}
	\end{picture}
	\label{fig5}
\end{figure}
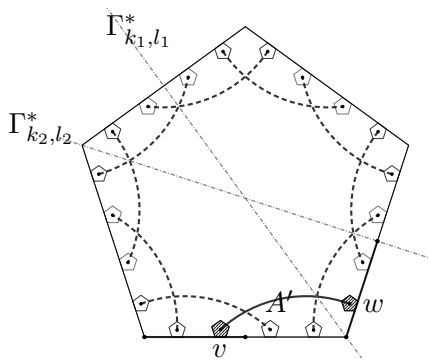

\begin{proof}
According to Proposition \ref{prop44}, we can find a connected $A'$ such that $\{w,v\}\subset A'\subset \Lambda_m$ and $R_m(w',v')\leq \eta\big(R_m(w,v)/\sigma_m\big)\sigma_m,\forall w',v'\in A'$. It suffices to let
\[A=\bigcup_{\Gamma\in \mathscr{G}}\Gamma^*(A').\]
We only need to show that $A$ is connected. Readers can find an illustration of the proof in Figure \ref{fig5}. Without loss of generality, we assume $i=j+d_{S_1}(i,j)$, and we let
\[\begin{cases}
	k_1=\lfloor i-\frac{1}{2}\rfloor,\\
	l_1=\lfloor i\rfloor+1,
\end{cases}\text{ and }
\begin{cases}
	 k_2=\lfloor i\rfloor,\\
	 l_2=\lfloor i+\frac{1}{2}\rfloor+1.
\end{cases}
\]
Then $\frac{k_1+l_1}{2}=i$ and $\frac{k_2+l_2}{2}=i+\frac{1}{2}$, so that one can then check
\[A'\cap \Gamma^*_{k_1,l_1}(A')\neq \emptyset,\ \Gamma^*_{k_2,l_2}\circ \Gamma^*_{k_1,l_1}(A')\cap \Gamma^*_{k_1,l_1}(A')\neq \emptyset,\ \Gamma^*_{k_2,l_2}\circ \Gamma^*_{k_1,l_1}=\Gamma^*_1.\]
Thus $A''=A'\cup \Gamma^*_{k_1,l_1}(A')\cup \Gamma^*_1(A')$ is connected, and as a consequence $A=\bigcup_{k\in S_0}\Gamma^*_k(A'')$ is connected.
\end{proof}

\noindent\textbf{Remark.} We need the full strength of Proposition \ref{prop45} for the development of Section \ref{sec6}, while in Section \ref{sec5} we only need the estimate $R_m(w,\Gamma^* w)\leq 2N_0\eta\big(R_m(w,v)/\sigma_m\big)\sigma_m$ for any $\Gamma\in \mathscr{G}$, which accutally can be derived by an easier way.

\section{Condition \textbf{(B)} for perfect polygon carpets}\label{sec5}
In this section, we prove the condition \textbf{(B)} for perfect polygon carpets. Noticing that $\rho_* =\rho_i$ for any $i\in S$ in this case, we will use an adaptation of the pure analytic argument in \cite{CQ3} developed for USC by two of the authors.

The proof takes two steps: first we prove a resistance estimate between half-sides using Proposition \ref{prop45} and Lemma \ref{lemma43}, then we construct bump functions by taking advantage of the good symmetry of the perfect polygon carpets. \vspace{0.2cm}

\noindent For $w,v\in \Lambda_n$, $n\geq 1$, we say $\Psi_w, \Psi_v$ are \textit{perfectly touching} if $\Psi_w\mcA\cap\Psi_v\mcA = \Psi_wL_i = \Psi_vL_{j}$ for some $i,j\in S_0$.\vspace{0.2cm}

For convenience of readers, we  recall the basic facts about resistances below. \vspace{0.2cm}

\noindent\textbf{Basic Facts.}

Let $\Lambda$ be a partition, $A\subset \Lambda$ connected and $A_1,A_2\subset A$. Denote \[R_{\Lambda, A}(A_1,A_2) = \big(\inf\{\mcD_{\Lambda,A}(f):f|_{A_1} = 1,f|_{A_2} = 0,f\in l(A)\}\big)^{-1}.\]
Write $R_{\Lambda,A}(w,v) := R_{\Lambda,A}(\{w\},\{v\})$ for short. Then the following holds:

\noindent(1). for $A\subset B\subset \Lambda,A_1,A_2\subset A$, we have $R_{\Lambda,A}(A_1,A_2) \geq R_{\Lambda,B}(A_1,A_2)$;

\noindent(2). for $A\subset \Lambda, w\sim_{\Lambda} v\in A$, we have $R_{\Lambda,A}(w,v) \leq 1$.\vspace{0.2cm}

   \noindent\textbf{Remark.} Since $\rho_*=\rho_i$ for all $i\in S$ for perfect polygon carpets, we always have $\Lambda_m=W_m$ for all $m\geq 0$, and in addition, for any $w\in W_*$, $m\geq 0$, $w^{-1}\cdot\mcB_{m}(w)$  coincides with $W_m$.\vspace{0.2cm}

First, we consider the resistance between disjoint half-sides.

\begin{lemma}\label{lemma51}
  There exists a constant $C > 0$ and $m_0\geq 1$ such that $R_m(w,v) \geq C\sigma_m$ for any $m\geq m_0$, $w\in\partial'_1\Lambda_m,v\in \partial'_i\Lambda_m$ with $i\in S_1$ and $d_{S_1}(1,i) \geq 1$.
\end{lemma}
\begin{proof} Notice that $\delta_m=\max_{\iota,\iota'\in \Lambda_m}R_m(\iota,\iota')$ for any $m\geq 1$.
	
It suffices to consider $m$ large enough. By Proposition \ref{prop32} and Lemma \ref{lemma42} we can choose $n>0$ independent of $m$ so that $R_m(\iota,\iota')\leq \frac{\delta_{n+m}}{4}, \forall \iota,\iota'\in \Lambda_{m}$.
By the basic fact (1) above, we then have
\begin{equation}\label{eqn51}
R_{n+m}(\tau\cdot \iota,\tau\cdot \iota')\leq \frac{\delta_{n+m}}{4}, \qquad\forall \tau\in \Lambda_n, \forall \iota,\iota'\in \Lambda_{m}.
\end{equation}
In addition, by Proposition \ref{prop32}, when $m$ is large enough, we always have $\#\Lambda_n\leq\frac{\delta_{n+m}}{4}$.

Now for $w\in\partial'_1\Lambda_m,v\in \partial'_i\Lambda_m$, we fix a pair $\iota,\kappa\in \Lambda_{n+m}$. One can find a path $\tau^{(0)},\tau^{(1)},\cdots,\tau^{(L)}$ with $L<\#\Lambda_n$ such that $\Psi_{\tau^{(k)}},\Psi_{\tau^{(k-1)}}$ are perfectly touching for each $1\leq k\leq L$, and $\iota\in \mathcal{B}_m(\tau^{(0)})$, $\kappa\in \mathcal{B}_m(\tau^{(L)})$. Hence, we can pick a sequence $\iota^{(j)},0\leq j\leq 2L+1$, such that $\iota^{(0)}=\iota$, $\iota^{(2N+1)}=\kappa$, and
\begin{equation}\label{eqn52}
\begin{cases}
	\{\iota^{(2k)},\iota^{(2k+1)}\}\subset \tau^{(k)}\cdot \{\Gamma^*(w): \Gamma\in \msG\},\qquad&\forall 1\leq k\leq L-1,\\
	\iota^{(2k-1)}\stackrel{n+m}{\sim}\iota^{(2k)},\qquad &\forall 1\leq k\leq L.
\end{cases}
\end{equation}
Then, by Proposition \ref{prop45} and by the basic fact (1) above, we have
\begin{equation}\label{eqn53}
R_{n+m}(\iota^{(2k)},\iota^{(2k+1)})\leq 2N_0\eta\big(\frac{R_m(w,v)}{\sigma_m}\big)\cdot\sigma_m,\quad \forall 1\leq k\leq L-1,
\end{equation}
where $\eta$ is the same function in Proposition \ref{prop44}.

Hence, combining (\ref{eqn51})-(\ref{eqn53}), by the basic fact (2), we can see
\[R_{n+m}(\iota,\kappa)\leq \sum_{j=0}^{2L}R_{n+m}(\iota^{(j)},\iota^{j+1})\leq \#\Lambda_n\cdot 2N_0\eta\big(\frac{R_m(w,v)}{\sigma_m}\big)\cdot\sigma_m+\frac{1}{2}\delta_{n+m}+\#\Lambda_n.\]
By taking the supreme over $\iota, \kappa$, and noticing that by Proposition \ref{prop32} and Lemma \ref{lemma42}, $\delta_{n+m}\geq C_1\sigma_m$ for some $C_1>0$, we see that $\eta\big(\frac{R_m(w,v)}{\sigma_m}\big)\geq C_2$ for some $C_2>0$ independent of $m$ and the choice of $w,v$. The lemma then follows immediately since $\eta(c)\to 0$ as $c\to 0$.
\end{proof}

\begin{corollary}\label{coro52}
There exists $C>0$ and $m_0\geq 1$ such that for any $m\geq m_0$,
\[R_m\big(\partial'
_1\Lambda_m\cup \partial'_{1/2}\Lambda_m, \bigcup_{k=4}^{2N_0-1}\partial'_{k/2}\Lambda_m\big)\geq C\sigma_m.\]
\end{corollary}
\begin{proof}
By Lemma \ref{lemma51}, there is $C_1>0$ so that $R_m(w,v)\geq C_1\sigma_m$ for any $w\in \partial'_1\Lambda_m\cup \partial'_{1/2}\Lambda_m$ and $v\in \bigcup_{k=4}^{2N_0-1}\partial'_{k/2}\Lambda_m$. Hence, by Lemma \ref{lemma43}, the corollary holds.
\end{proof}

\begin{theorem}\label{thm53}
  The condition \textbf{(B)} holds for perfect polygon carpets.
\end{theorem}
\begin{proof}
For each $m\geq m_0$, define $f_{m,1}\in l(\Lambda_m)$ as the unique function satisfying
\[f_{m,1}|_{\partial'_1\Lambda_m\cup \partial'_{1/2}\Lambda_m}=1,f_{m,1}|_{\bigcup_{k=4}^{2N_0-1}\partial'_{k/2}\Lambda_m}=0,\]
and
\begin{equation}\label{eqn54}
\mcD_m(f_{m,1})=R_m^{-1}\big(\partial'
_1\Lambda_m\cup \partial'_{1/2}\Lambda_m, \bigcup_{k=4}^{2N_0-1}\partial'_{k/2}\Lambda_m\big),
\end{equation}
where $m_0$ is the same number in Corollary \ref{coro52}.
Define $f_{m,i}=f_{m,1}\circ (\Gamma^{*}_{i-1})^{-1}$ for each $i\in S_0$.

Next, we fix $n\geq 1$ and $w\in \Lambda_n$. For each $i\in S_0$, we define $g_{w,m,i}\in l(\Lambda_{n+m})$ by gluing together scaled copies of $f_{m,j}$ so that $g_{w,m,i}$ is positive only in a neighbourhood of $\Psi_wq_i$. More precisely, for each $v\in \Lambda_n,v'\in \Lambda_m$, we define
\[g_{w,m,i}(v\cdot v')=
\begin{cases}
	0,&\text{ if }\Psi_wq_i\notin \Psi_v\mcA.\\
	f_{m,j}(v'),&\text{ if there is $j\in S_0$ s.t. }\Psi_wq_i=\Psi_vq_j.
\end{cases}\]
Noticing that there are at most $6$ many $v\in \Lambda_n$ such that $\Psi_wq_i\in \Psi_v\mcA$, we can easily see that for some $C_1>0$,
\[\mcD_{n+m}(g_{w,m,i})\leq 6\mcD_m(f_{m,1})\leq C_1\sigma_m^{-1},\qquad\forall m\geq m_0,\]
where the second inequality is due to (\ref{eqn54}) and Corollary \ref{coro52}.

Finally, let
\[g_{w,m}=\big(\sup_{i\in S_0}g_{w,m,i}\big)\vee 1_{\mcB_m(w)},\]
where $1_{\mcB_m(w)}\in l(\Lambda_{n+m})$ is the indicator function of $\mcB_m(w)$. Then one can check
\[\begin{cases}
	g_{w,m}|_{\mcB_m(w)}\equiv 1,\quad g_{w,m}|_{\mcB_m\big(\mcN_2^c(w)\big)}\equiv 0,\\
	\mcD_{n+m}(g_{w,m})\leq \mcD_{n+m}\big(\sup\limits_{i\in S_0}g_{w,m,i}\big)\leq N_0 C_1 \sigma_m^{-1}.
\end{cases}\]
Hence
\[R_{n+m}\big(\mcB_m(w),\mcB_m(\mcN_2^c(w))\big)\geq N^{-1}_0C^{-1}_1\sigma_m,\qquad \forall m\geq m_0.\]
Since the argument works for any $n\geq 1$ and $w\in W_n$, the condition (\textbf{B}) holds.
\end{proof}

\section{Half-side resistance estimates for bordered polygon carpets}\label{sec6}
The existence problem of standard self-similar Dirichlet forms on general polygon carpets is much more difficult and interesting, since we have seen evidences that the result depends on the geometry of the fractal. In particular, inspired by Sabot's work on p.c.f. fractals \cite{Sabot}, two of authors \cite{CQ4} found a Sierpinski carpet like fractal without a standard form, a bordered square carpet whose opposite sides are strongly connected with large cells in the middle, but the four corner vertices are loosely connected to the center, see Figure \ref{fig6}.

\begin{figure}[htp]
\includegraphics[width=4.5cm]{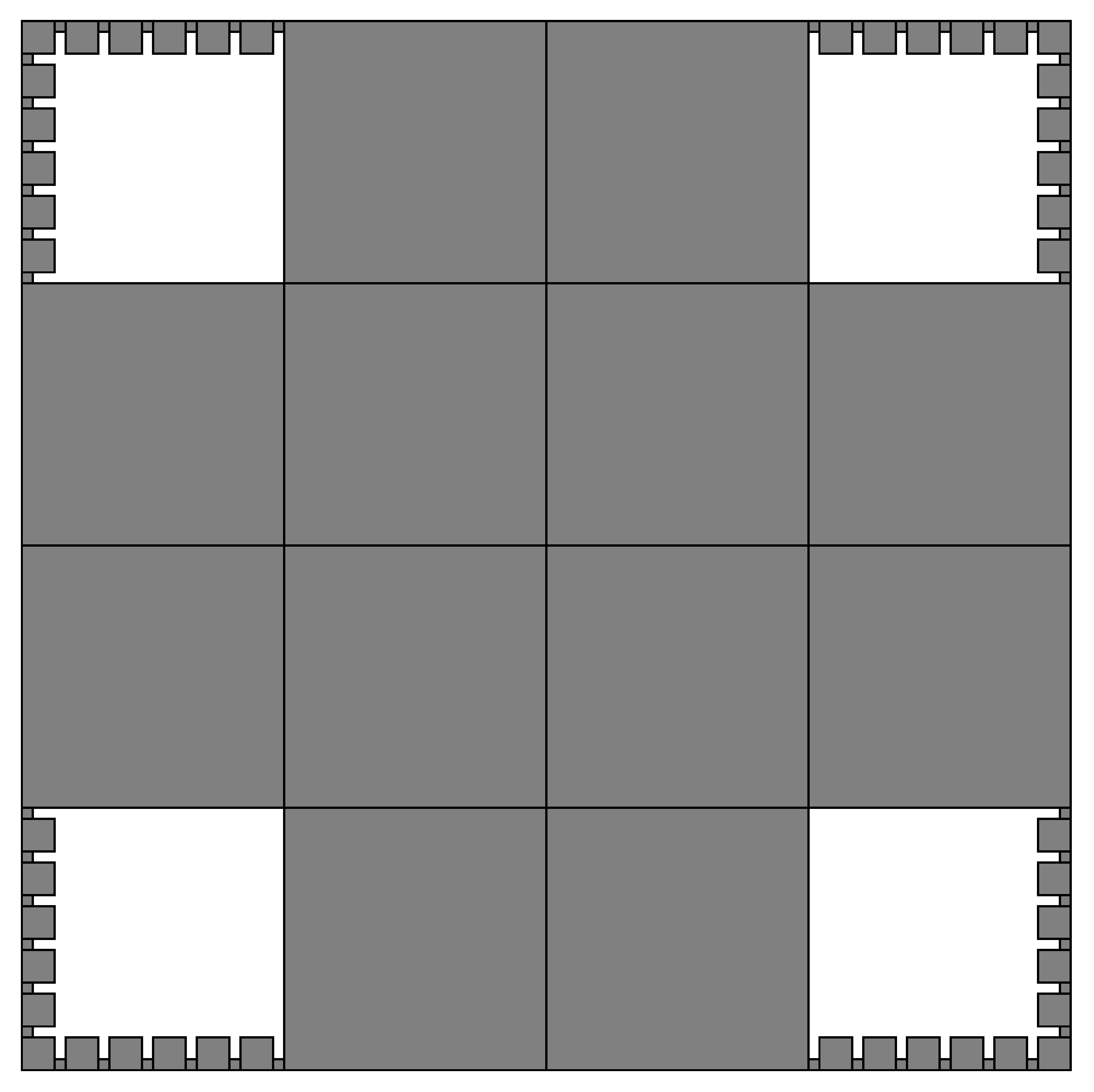}\hspace{1cm}
	\includegraphics[width=4.5cm]{sctlevel3}
	\caption{A  bordered square carpet without standard self-similar Dirichlet form.}
	\label{fig6}
\end{figure}

In this section, we introduce a class of bordered polygon carpets where we can obtain some resistance estimates. Notice that a bordered polygon carpet must have $N_0=3$ or $4$, so we are dealing with triangle carpets and square carpets. To  prove condition \textbf{(B)}, we need some technique and an analogous extension theorem from \cite{CQ3}, and require some extra conditions about the geometry. We hope our work will inspire further investigation.\vspace{0.2cm}

 \textit{For each $i\in S_0$, we specify $\Psi_i$ to be the contracting similarity which fixes the vertex $q_i$. }\vspace{0.2cm}

We introduce the following loop intersection condition for the development of this section. Specifically, we will focus on a class of fractals satisfying the  condition: \vspace{0.2cm}

\noindent\textbf{(LI).} Say a bordered polygon carpet satisfies the \textit {loop intersection condition}, if $A\cap \Psi_1A\neq \emptyset$ for any path connected closed $\msG$-symmetric $A\subset K$ such that $A\cap L'_i\neq\emptyset,\forall i\in S_1$.

\begin{theorem}\label{thm61}
	Let $K$ be a bordered polygon carpet and assume \textbf{(LI)} holds. Then, there exist $C > 0$ and $m_0\geq 1$ such that $R_m(w,v) \geq C\sigma_m$ for any $m\geq m_0$, $w\in\partial'_1\Lambda_m,v\in \partial'_i\Lambda_m$ with $i\in S_1$ and $d_{S_1}(1,i) \geq 1$.
\end{theorem}

\subsection{Examples of bordered polygon carpets with \textbf{(LI)}}
Before proving Theorem \ref{thm61}, let's first see some classes of bordered polygon carpets that satisfy (\textbf{LI}).

In particular, we consider the following conditions \textbf{(H)}, \textbf{(C-3)} and \textbf{(C-4)} imposed on bordered polygon carpets. Here $3,4$ stands for the different cases $N_0=3$ or $N_0=4$. The condition \textbf{(H)} is called the \textit{hollow condition}, while \textbf{(C-3)}, \textbf{(C-4)} indicate how $\Psi_iK$, $i\in S_0$ is connected to the outside. See the left two pictures in Figure \ref{fig1} for two carpets satisfying these conditions.  \vspace{0.2cm}

\noindent\textbf{(H)}. For any $ i\in S$, $\partial\mathcal{A}\cap \Psi_i\mathcal{A}\neq \emptyset$. In addition, for any $i\neq j\in S\setminus S_0$, $\Psi_i\mcA\cap \Psi_j\mcA$ is either empty or a line segment. \vspace{0.2cm}

\noindent\textbf{(C-4)}. $N_0=4$, and $(\Psi_1K)\cap \text{cl}\big(K\setminus \Psi_1K)\subset \Psi_1(L'_2\cup L'_{7/2})$.\vspace{0.2cm}

\noindent\textbf{(C-3)}. $N_0=3$, and $(\Psi_1K)\cap \text{cl}\big(K\setminus \Psi_1K)=\{\Psi_1q_2,\Psi_1q_3\}$.

\begin{proposition}\label{prop62}
(a). If a bordered polygon carpet $K$ satisfies conditions \textbf{(H)} and \textbf{(C-4)}, it satisfies \textbf{(LI)}.

(b). If a bordered polygon carpet $K$ satisfies conditions \textbf{(H)} and \textbf{(C-3)}, it satisfies \textbf{(LI)}.
\end{proposition}
\begin{proof}
Let $A$ be a path connected $\msG$-symmetric closed subset of $K$ such that $A\cap L'_i\neq\emptyset,\forall i\in S_1$. If $q_1\in A$, then there is nothing to prove. Hence in the following, we always assume $q_1\notin A$.\vspace{0.2cm}

\noindent\textit{Observation 1. If $z_1,z_2$ belong to a same connected component of $K\setminus A$, then there is a simple curve $\gamma:[0,1]\to K\setminus A$ so that $\gamma(0)=z_1,\gamma(1)=z_2$. }\vspace{0.2cm}

Let $z\in K\setminus A$, and let $C_z=\{z'\in K\setminus A:\text{ there is a simple curve }\gamma:[0,1]\to K\setminus A\text{ s.t. }\gamma(0)=z,\gamma(1)=z'\}$. Then it is not hard to see that $C_z$ is open in $K$, and $C_z$ is closed in $K\setminus A$. Hence $C_z$ is a connected component of $K\setminus A$. Observation 1 follows immediately. \vspace{0.2cm}

\noindent\textit{Observation 2. If $z_1,z_2$ belong to a same connected component of $K\setminus (\bigcup_{i\in S_0}\Psi_iA)$, then there is a simple curve $\gamma:[0,1]\to K\setminus (\bigcup_{i\in S_0}\Psi_iA)$ so that $\gamma(0)=z_1,\gamma(1)=z_2$. } \vspace{0.2cm}

The proof of Observation 2 is exactly the same as Observation 1.\vspace{0.2cm}

\noindent\textit{Observation 3. Let $i,j\in S_1$ with $d_{S_1}(i,j)\geq 1$. Then for any $z_1\in L'_i\setminus A$ and $z_2\in L'_j\setminus A$, $z_1, z_2$ belong to different connected components of $K\setminus A$.  }\vspace{0.2cm}

We prove Observation 3 by contradiction. Assume $z,z'$ belong to a same component of $K\setminus A$, then one find a simple curve $\gamma$ in $K\setminus A$ connecting $z,z'$ by Observation 1. One can extend $\gamma$ to be a simple closed curve $\tilde{\gamma}$ by gluing it with $\gamma':[0,1]\to \text{cl}(\mathbb{R}^2\setminus \mathcal{A})$ that connects $z,z'$. Assume $i<j$, then one can see that $(\bigcup_{i<k<j}L'_k)\cap A$ and $(\bigcup_{S_1\setminus [i,j]}L'_k)\cap A$ belong to different components of $\mathbb{R}^2\setminus \tilde{\gamma}([0,1])$. A contradiction to the fact that $A$ is connected.  \vspace{0.2cm}

(a). We prove (a) by contradiction. Assume $A\cap (\bigcup_{i\in S_0}\Psi_iA)=\emptyset$. Noticing that $\bigcup_{i\in S_0}\Psi_iK$ is not connected, $A$ contains some point $z\in K\setminus (\bigcup_{i\in S_0}\Psi_iK)$. Let $z'=\Gamma_1 (z)$, then since $A$ is path connected, by Observation 2, we can find a path $\gamma:[0,1]\to K\setminus (\bigcup_{i\in S_0}\Psi_iA)$ such that $\gamma(0)=z$, $\gamma(1)=z'$.

Let $X$ be the unique connected component of $K\setminus \bigcup_{i\in S_0}\Psi_iK$ containing $z$. Then by letting
\[t_1=\sup\{0<t<1:\gamma(t)\in X\},\quad t_2=\inf\{t>t_1:\gamma(t)\in K\setminus (\bigcup_{i\in S_0}\Psi_iK)\},\]
one can see $\gamma([t_1,t_2])\subset \Psi_iK$ for some $i\in S_0$ by \textbf{(H)}, and by \textbf{(C-4)}, $\Psi_i^{-1}\circ \gamma$ connects $L'_{i+1}$ and $L'_{i+5/2}$. Hence, $\gamma$ intersects $\Psi_iA$ by using Observation 3. A contradiction.

(b) can be proved with a same argument as (a).
\end{proof}

\subsection{From \textbf{(LI)} to half-side resistance estimates}
We can prove a same result as Lemma \ref{lemma51} for bordered polygon carpets satisfying \textbf{(LI)}.

\begin{definition}\label{def62}
Let $m\geq 0$ and $A\subset K$.

(a). We write
	\[\mathcal{I}_m(A)=\{w\in \Lambda_m:\Psi_wK\cap A\neq \emptyset\}.\]
	For convenience, we write $\mathcal{I}_mx=\mathcal{I}_m(\{x\})$ for $x\in K$.

(b). For $w\in W_*$, we write
\[\mathcal{I}_m(A,w)=\{v\in \mathcal{I}_m(A):\Psi_vK\subset \Psi_wK\}.\]
\end{definition}

\begin{lemma}\label{lemma63}
	For any $n\geq 1$, there exists $C(n)>0$ such that
	\[R_m\big(\mathcal{I}_m(\Psi_w x,w),\mathcal{I}_m(\Psi_w y,w)\big)\leq C(n)\cdot R_{m}\big(\mathcal{I}_mx,\mathcal{I}_my\big),\]
	for any $m>n$, $w\in W_*$ satisfying $\rho_w\geq \rho_*^n$ and  $x,y\in K$.
\end{lemma}
\begin{proof}
	It suffices to consider the case that $R_m\big(\mathcal{I}_m(\Psi_wx,w),\mathcal{I}_m(\Psi_w y,w)\big)>0$. Choose $f\in l(\Lambda_m)$ so that $f|_{\mathcal{I}_m(\Psi_w x,w)}=0$, $f|_{\mathcal{I}_m(\Psi_w y,w)}=1$ and $\mcD_m(f)=R_m^{-1}\big(\mathcal{I}_m(\Psi_w x,w),\mathcal{I}_m(\Psi_w y,w)\big)$.
	Denote $\mathcal{C}_m(w)=\{v: w\cdot v\in \Lambda_m\}$. Define 	$g\in l(\Lambda_m)$  by
	\[g(\iota)=f(w\cdot v),\qquad\text{ for each }\iota\in\Lambda_m \text{ such that } \Psi_\iota K\subset \Psi_vK \text{ with }v\in\mathcal C_m(w).\]
	Clearly $g|_{\mathcal{I}_mx}=0$, $g|_{\mathcal{I}_my}=1$. In addition, for each $v\in\mathcal C_m(w)$, noticing that $\rho_{v}\leq \rho_*^m\rho_w^{-1}$, by a volume calculation, the collection $\{\iota\in\Lambda_m: \Psi_\iota K\subset \Psi_vK\}$ consists of at most $(\rho_*^m\rho_w^{-1}/\rho_*^{m+1})^{d_H}\leq \rho_*^{-(n+1)d_H}$ many elements. Thus by the construction of $g$, there exists $C(n)>0$ depending only $n$ such that $\mcD_m(g)\leq C(n) \mcD_{m}(f)$. The lemma follows immediately.
\end{proof}

Next, we  estimate the resistances between half-sides. We will take two steps.
In the first step, we use a similar argument as in \cite[Lemma 4.11]{CQ3} on USC to show a lower bound estimate of resistances between boundary vertices. This step holds for general bordered polygon carpets, which also motivates the construction of the counter-example considered in \cite{CQ4}. In the second step, we will apply \textbf{(LI)} and Proposition \ref{prop45}.

\begin{lemma}\label{lemma64}
Let $K$ be a bordered polygon carpet. Then there exists $C>0$ such that \[R_m\big(\mathcal{I}_mq_1,\mathcal{I}_mq_2\big)\geq C\sigma_m,\qquad\forall m\geq 1.\]	
\end{lemma}
\begin{proof}
By Lemma \ref{lemma42}, noticing that $\delta_m=\max_{\iota,\kappa\in\Lambda_m}R_m(\iota,\kappa)$, it is easy to see that there exist some $\iota,\kappa\in \Lambda_m$ so that $R_m(\iota,\kappa)\geq C_1\sigma_m$ for some $C_1>0$ independent of $m$. It suffices to consider $m$ large enough. By Proposition \ref{prop32} and Lemma \ref{lemma42} we can choose $n>0$ independent of $m$ so that $\delta_{m-n}\leq C_1\sigma_{m}/4$. In addition, by Proposition \ref{prop32}, when $m$ is large enough, we always have $\sigma_{m}\gg 1$.

One can find a path $\tau^{(0)},\tau^{(1)},\cdots,\tau^{(L)}$ with $L<\#\Lambda_n$ such that $\tau^{(k)}\sn\tau^{(k-1)}$ for each $1\leq k\leq L$, and $\iota\in \mathcal{B}_m(\tau^{(0)})$, $\kappa\in \mathcal{B}_m(\tau^{(L)})$. Hence, we can pick a sequence $\iota^{(j)},0\leq j\leq 2L+1$, such that $\iota^{(0)}=\iota$, $\iota^{(2N+1)}=\kappa$, and
\begin{equation}\label{eqn61}
	\begin{cases}
		\{\iota^{(2k)},\iota^{(2k+1)}\}\subset \mcB_{m-n}(\tau^{(k)}),\qquad&\forall 0\leq k\leq L,\\
		\iota^{(2k-1)}\stackrel{n+m}{\sim}\iota^{(2k)},\qquad &\forall 1\leq k\leq L.
	\end{cases}
\end{equation}
Then by a same argument in the proof of Lemma \ref{lemma51}, one can find $1<k<L$ such that
\[R_m(\iota^{(2k)},\iota^{(2k+1)})\geq \frac{C_1\sigma_m/2-\#\Lambda_n}{\#\Lambda_n}\geq C_2\sigma_m,\]
where $C_2>0$ is a constant independent of large enough $m$. Since, $\{\iota^{(2k)},\iota^{(2k+1)}\}\subset \mcB_{m-n}(\tau^{(k)})$, and noticing that $\iota^{(2k)},\iota^{(2k+1)}$ are on the boundary of $\mcB_{m-n}(\tau^{(k)})$, one can apply Lemma \ref{lemma63} to find $w,v\in \partial\Lambda_m$ so that
\[R_m(w,v)\geq C_3\sigma_m,\]
for some $C_3>0$ independent of large enough $m$.

 Next, by the triangle inequality, we can choose $w',v'$ from $\{w,v\}\cup \bigcup_{j\in S_0} \mathcal{I}_m{q_j}$, so that $R_m(w',v')\geq C_3\sigma_m/3$ and both $w', v'\in \partial_i\Lambda_m$ for some $i\in S_0$.

Finally, we choose $n'$ large enough (independent of $m$) so that $\delta_{m-n'}\leq C_3\sigma_{m}/12$. Then, by a chaining argument as before (arranging the chain connecting $w'$ and $v'$ along $\partial_i\Lambda_m$), we can find $\tau\in \Lambda_{n'}$ so that
\[R_m\big(\mathcal{I}_m(\Psi_\tau q_i,\tau),\mathcal{I}_m(\Psi_\tau q_{i+1},\tau)\big)\geq C_4\sigma_m,\]
for some $C_4>0$ independent of large enough $m$. The lemma then follows by applying Lemma \ref{lemma63} again.
\end{proof}

Now we prove the main theorem of this section. For convenience of the readers, we provide a figure (Figure \ref{fig7}) sketching the idea of the proof.

\begin{figure}[htp]
	\includegraphics[width=6.5cm]{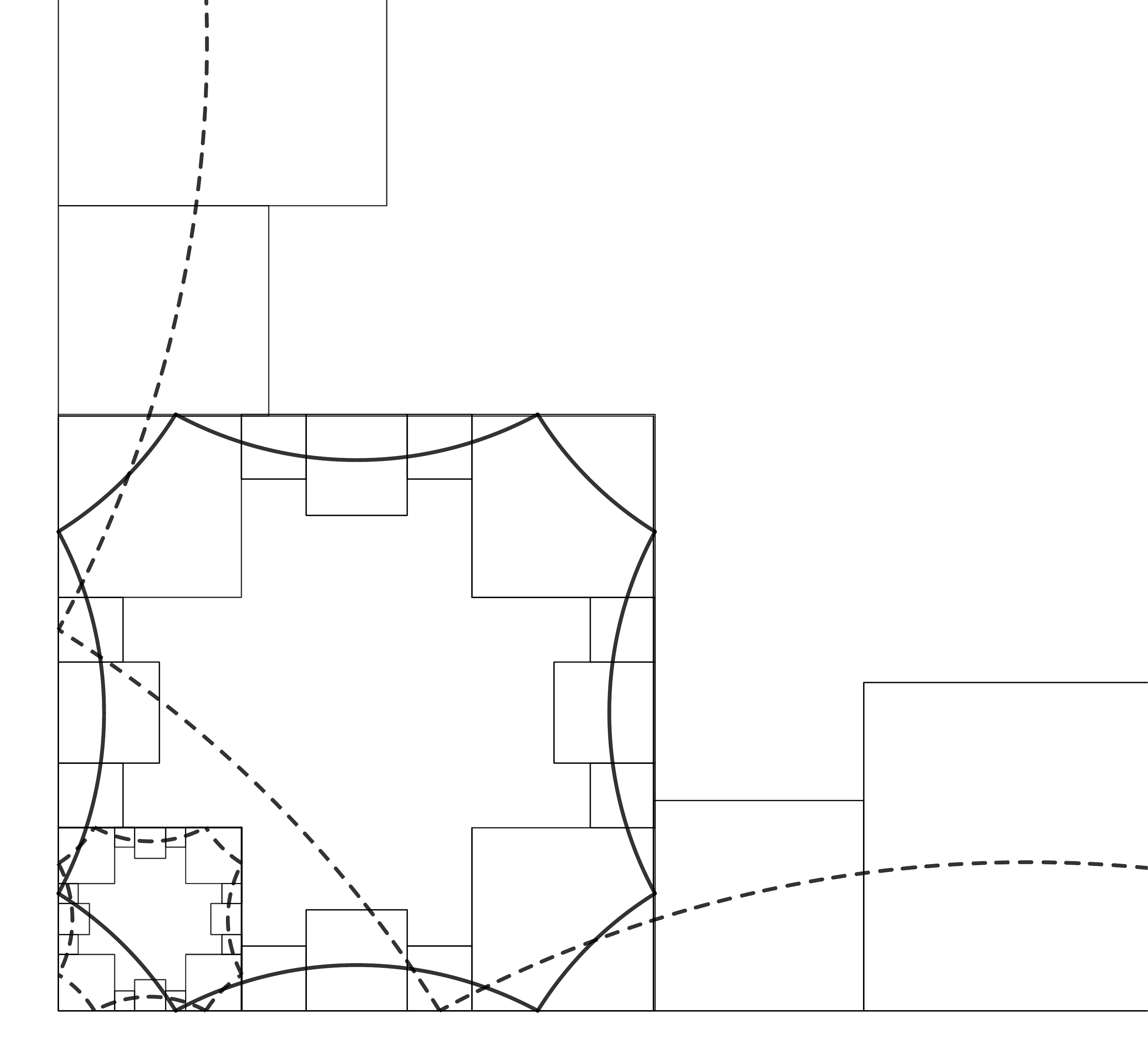}
	\caption{An illustration of the set $B$ in the proof of Theorem \ref{thm61}.}
	\label{fig7}
\end{figure}

\begin{proof}[Proof of Theorem \ref{thm61}]
	Take $w\in\partial'_1\Lambda_m$, $v\in\partial'_i\Lambda_m$ be as stated in the theorem. By Proposition \ref{prop45}, there exists a $\mathscr{G}$-symmetric $\tilde{A}$ such that $w\in\tilde{A}\subset \Lambda_m$ and $R_m(w',v')\leq \eta\big(\frac{R(w,v)}{\sigma_m}\big)\sigma_m,\forall w',v'\in \tilde{A}$, where $\eta$ is the same function stated in Proposition \ref{prop44}. Let $A=\bigcup_{w\in \tilde{A}}\Psi_wK$. Noticing that $\mathcal{I}_m(A)=\{w\in \Lambda_m:w\in \tilde{A}\text{ or }w\sm \tilde{A}\}$. We have
	\begin{equation}\label{eqn62}
	R_m(w',v')\leq \eta\big(\frac{R(w,v)}{\sigma_m}\big)\sigma_m+2,\qquad \forall w',v'\in \mathcal{I}_m(A).
	\end{equation}
	
	By Lemma \ref{lemma64}, $R_m\big(\mathcal{I}_mq_1,\mathcal{I}_mq_2\big)\geq C_1\sigma_m$ for some $C_1>0$ independent of $m$. We then choose $n$ large enough (independent of large $m$) so that $\delta_{m-n}\leq C_1\sigma_m/4$. Let  $n'\geq 0$ be the number determined by $(1)^{n'}\in \Lambda_{n}$, where we write $(1)^{n'}=111\cdots 1$ ($1$ repeats $n'$ times). By the \textbf{(LI)} condition, we can see that
    \[B=\big(\bigcup_{k=0}^{n'}\Psi_1^kA\big)\cup \big(\bigcup_{k=0}^{n'}\Psi_2^kA\big)\]
    is connected. Then by (\ref{eqn62}), and by using Lemma \ref{lemma63} to each $1^k$ and $2^k$ in $W_*$, $0\leq k\leq n'$, we can see that there exists $C_2>0$ depending only on $n$ so that
    \[R_m(w',v')\leq C_2\cdot \Big(\eta\big(\frac{R(w,v)}{\sigma_m}\big)\sigma_m+2\Big)+2n',\qquad \forall w',v'\in \mathcal{I}_m(B).\]
    Hence, by picking $w'\in \mathcal{I}_m(\Psi_{1^{n'}}A)$ and $v'\in \mathcal{I}_m(\Psi_{2^{n'}}A)$, we get
    \[
    \begin{aligned}
    	C_1\sigma_m\leq R_m\big(\mathcal{I}_mq_1,\mathcal{I}_mq_2\big)&\leq R_m(\mathcal{I}_mq_1,w')+R_m(w',v')+R_m(v',\mathcal{I}_mq_2)\\
    	&\leq C_2\cdot \Big(\eta\big(\frac{R(w,v)}{\sigma_m}\big)\sigma_m+2\Big)+2n'+\frac{C_1\sigma_m}{2}.
    \end{aligned}
    \]
    Noticing that $\sigma_m\gg 1$ for large $m$, we have $\eta(\frac{R(w,v)}{\sigma_m})\geq C_3$ for some $C_3>0$ independent of $m$ and the choice of $w,v$. So the theorem follows since $\eta(c)\to 0$ as $c\to 0$.
\end{proof}

\section{Condition \textbf{(B)} for some hollow bordered polygon carpets}\label{sec7}
In this section, we will prove condition \textbf{(B)} for all bordered polygon carpets satisfying \textbf{(H)} and \textbf{(C-4)}, or satisfying \textbf{(H)} and \textbf{(C-3)}. The main idea will be defining functions on $\mcB_m(w)$ with good boundary values, which has been developed for USC by two of the authors in the previous paper \cite{CQ3}.

\begin{theorem}\label{thm71}
	The condition (\textbf{B}) holds for bordered polygon carpets satisfying \textbf{(H)} and \textbf{(C-4)}, or satisfying \textbf{(H)} and \textbf{(C-3)}.
\end{theorem}

We divide the section into two parts. We will provide the main proof in the first part, except a fundamental lemma which will be proved in the second part.

\subsection{Proof of (\textbf{B}) by an extension argument} The key step of the proof of Theorem \ref{thm71} is to construct a function on $\mcB_m(w)$ that has linear boundary values. First, we introduce some notations. \vspace{0.2cm}

\noindent\textbf{Notations.}

 (1). We write $\tau\triangleleft \tau'$ or equivalently $\tau'\triangleright \tau$ if $\tau\in \tau'\cdot W_1$.

(2). Let $\Lambda\subset W_*$ be a partition, i.e. $\bigcup_{w\in\Lambda}\Psi_wK=K$ and $\mu(\Psi_wK\cap \Psi_vK)=0$ for any $w\neq v\in \Lambda$.

(2-1). We write
\[
\begin{cases}
	\partial_i\Lambda=\{\tau\in \Lambda:\Psi_\tau K\cap L_i\}\neq \emptyset,\ \forall i\in S_0,\\
	\partial'_i\Lambda=\{\tau\in \Lambda:\Psi_\tau K\cap L'_i\}\neq \emptyset,\ \forall i\in S_1,
\end{cases}
\]
and $\partial \Lambda = \bigcup_{i\in S_0}\partial_i\Lambda$.
For convenience, for $i\in S_0$, also denote $\tilde{\partial}_i\Lambda=w$, where $w\in \Lambda$ and $q_i\in \Psi_wK$. In addition, we write $\underline{\partial}_1\Lambda = \big\{\tau\in \partial_1\Lambda:\Psi_{\tau}\big([q_1,q_2]\big)\subset [q_1,q_2]\big\}$ and  $\underline{\partial}\Lambda = \bigcup_{i\in S_0}\Gamma^*_i(\underline{\partial}_1\Lambda)$, noting that only when $N_0 = 3$, it may happen that $\underline{\partial}_1\Lambda\subsetneq \partial_1\Lambda$.

(2-2). Write $\Theta_1(\Lambda)=\big\{v\in W_*: \text{ there exists }\tau\in \underline{\partial}_1\Lambda \text{ such that }\tau\in v\cdot (W_*\setminus W_0)\big\}$. Write $\#\Theta_1(\Lambda)=M$. \vspace{0.2cm}

(3). Write $q_c  = \frac{1}{N_0}\sum_{j\in S_0}q_j$ the center of $\mcA$.\vspace{0.2cm}

To construct functions with nice boundary values, we need the following two lemmas.

\begin{lemma}\label{lemma72}
There exist $0<\alpha<1$ and $C>0$ such that
\[\sigma_{n+m}\leq C \sigma_m\cdot \rho_*^{-n\alpha},\quad \forall m,n\geq 1. \]
\end{lemma}
\begin{proof}

Let $V_0=\{q_i\}_{i\in S_0}$ and define $D(g) = \sum_{i\in S_0}\big(g(q_i) - g(q_{i + 1})\big)^2,\forall g\in l(V_0)$.  For any partition $\Lambda$, let $V_\Lambda = \bigcup_{w\in \Lambda}\Psi_wV_0$ and denote $D_{\Lambda}(g) = \sum_{w\in\Lambda}\rho_w^{-1}D(g\circ\Psi_w),\forall g\in l(V_{\Lambda})$ as
a discrete energy form on $l(V_\Lambda)$. Abbreviate $V_{\Lambda_n}, D_{V_{\Lambda_n}}$ to $V_n,D_n$ for $n\geq 0$.

Note that for $\Lambda\neq \Lambda_0$, by using the knowledge of electrical networks, for each $g\in l(V_{\Lambda})$ and $w\in \underline{\partial}_1\Lambda$, we have
$\sum_{i\in S_0}\big(g(\Psi_wq_i) - g(\Psi_wq_{i+1})\big)^2\geq  \frac{4}{3}\big(g(\Psi_wq_1) - g(\Psi_wq_2)\big)^2,$
and thus
$$\begin{aligned}&\rho_{\tilde{\partial}_1\Lambda}^{-1}\sum_{j = 1,2}\big(g(\Psi_{\tilde{\partial}_j\Lambda}q_1) - g(\Psi_{\tilde{\partial}_j\Lambda}q_2)\big)^2 +\sum_{w\in \underline{\partial}_1\Lambda\setminus (\tilde{\partial}_1\Lambda\cup\tilde{\partial}_2\Lambda)}\rho_{w}^{-1}D(g\circ \Psi_w)\\
\geq &\rho_{\tilde{\partial}_1\Lambda}^{-1}\sum_{j = 1,2}\big(g(\Psi_{\tilde{\partial}_j\Lambda}q_1) - g(\Psi_{\tilde{\partial}_j\Lambda}q_2)\big)^2 + \frac{4}{3}(1 - 2\rho_{\tilde{\partial}_1\Lambda})^{-1}\big(g(\Psi_{\tilde{\partial}_1\Lambda}q_2) - g(\Psi_{\tilde{\partial}_2\Lambda}q_1)\big)^2\\
\geq &(\frac{3}{4} + \frac{1}{2}\rho_{\tilde{\partial}_1\Lambda})^{-1}\big(g(q_1) - g(q_2)\big)^2,
\end{aligned}$$
which gives $D_{\Lambda}(g) \geq (\frac{3}{4} + \frac{1}{2}\rho_{\tilde{\partial}_1\Lambda})^{-1}D(g|_{V_0})$ by symmetry. Choose $m_0$ large enough so that $w^{-1}\cdot\mcB_{m_0}(w)$ is finer than $W_2$ for any $w\in W^*$, then for any $n\geq 0$, $g\in l(V_{n + m_0})$, we have
$$D_{n+m_0}(g)= \sum_{w\in \Lambda_n}\rho_w^{-1}D_{w^{-1}\cdot \mcB_{m_0}(w)}(g\circ \Psi_{w}) \geq (\frac{3}{4} + \frac{1}{2}\rho_{\tilde{\partial}_1W_2})^{-1}D_n(g|_{V_n}) \geq \frac{8}{7}D_n(g|_{V_n}).$$
Then by induction, for any $n\geq 1$, $g\in l(V_n)$,
\begin{equation}\label{eq71}
D_n(g)\gtrsim (\frac{8}{7})^{\frac{n}{m_0}}D(g|_{V_0})\gtrsim (\frac{8}{7})^{\frac{n}{m_0}}\big(g(q_1)-g(q_2)\big)^2.
\end{equation}

Next, for $n,m\geq 1$, let us consider a function $f\in l(\Lambda_{n+m})$ that takes boundary values $f(\tilde{\partial}_1\Lambda_{n+m})=1$, $f(\tilde{\partial}_2\Lambda_{n+m})=0$, and
\begin{equation}\label{eq72}
\mcD_{n+m}(f)=R_{n+m}(\tilde{\partial}_1\Lambda_{n+m},\tilde{\partial}_2\Lambda_{n+m})^{-1}\asymp \sigma_{n+m}^{-1},
\end{equation}
where the estimate is due to Lemma \ref{lemma64}, Lemma \ref{lemma42}, and the fact that $\delta_{n+m}=\max\{R_{n+m}(\iota,\kappa): \iota, \kappa\in \Lambda_{n+m}\}$. We then write $\tilde f=\pi_n^{-1}\circ P_n\circ \pi_{n+m}f\in l(\Lambda_n)$, and define a function $g\in l(V_n)$ by
\[g(x)=\big(\#\{w\in\Lambda_n: x\in \Psi_w K\}\big)^{-1}\cdot\big(\sum_{w\in \Lambda_n,x\in \Psi_wK}\tilde f(w)\big).\]
Then, combining (\ref{eq71}) and (\ref{eq72}), one can easily check that
\[
\begin{aligned}
\sigma_{n+m}^{-1}\asymp\mcD_{n+m}(f)\gtrsim \sigma_m^{-1}\sum_{w\in \Lambda_n}D(g\circ\Psi_w)
\asymp \rho_*^n\sigma_m^{-1}D_n(g)
\gtrsim (\frac{8}{7})^{\frac{n}{m_0}}\rho_*^n\sigma_m^{-1},
\end{aligned}
\]
and thus the lemma follows.
\end{proof}

\begin{lemma}\label{lemma73}
Let $\Lambda\subset W_*$ be a non-trivial partition, i.e. $\Lambda\neq\{\emptyset\}$,  and assume there is $n\geq 0$ so that
\[\rho_*^{n+2}<\rho_\tau\leq \rho_*^n,\qquad\forall \tau\in \Lambda.\]

(a). There exists a function $h_\Lambda\in l(\Lambda)$ such that $\mcD_\Lambda(h_\Lambda)\leq C\sigma_n^{-1}$ where $C>0$ is independent of $\Lambda$, and for any $j\in \underline{\partial}_1S$,
\[
\begin{cases}
h_\Lambda(j\cdot \bullet)\big|_{\partial'_{1/2}(j^{-1}\cdot\Lambda)}=Li\big(\Psi_{j\cdot \tilde{\partial}_{1/2}(j^{-1}\cdot \Lambda)}(\frac{q_1+q_2}{2})\big),\\
h_\Lambda(j\cdot \bullet)\big|_{\partial'_2(j^{-1}\cdot\Lambda)}=Li\big(\Psi_{j\cdot \tilde{\partial}_2(j^{-1}\cdot \Lambda)}(\frac{q_1+q_2}{2})\big),\hspace{0.25cm}
\end{cases}
\]
and
\[h_\Lambda|_{\partial'_{1/2}\Lambda}=Li\big(\Psi_{\tilde{\partial}_1\Lambda}(\frac{q_1+q_2}{2})\big),\quad h_\Lambda|_{\partial'_2\Lambda}=Li\big(\Psi_{\tilde{\partial}_2\Lambda}(\frac{q_1+q_2}{2})\big),\]
where $Li$ is the linear function on $\overline{q_1,q_2}$ such that $Li(q_1)=0$ and $Li(q_2)=1$.

(b). There exists a function $h'_\Lambda\in l(\Lambda)$ so that $\mcD_\Lambda(h_\Lambda')\leq C'\sigma_n^{-1}$ where $C'>0$ is a constant independent of $\Lambda$, and
\[\begin{cases}
h'_\Lambda|_{j\cdot j^{-1}\cdot \Lambda}=0,&\forall j\in \partial_1S,\\
h'_\Lambda|_{\partial'_j\Lambda}=1,&\forall j\in S_1,\dist_{S_1}(j,\{1,3/2\})\geq 1.
\end{cases}
\]
\end{lemma}

The proof of Lemma \ref{lemma73} will be postponed to the next subsection.
Using Lemma \ref{lemma72} and Lemma \ref{lemma73}, we can construct functions with good boundary values.

\begin{proposition}\label{prop74}
Let $li$ be a linear function on $\mathbb{R}^2$. Also let $w\in \Lambda_*$ and $m\geq 1$. Then one can find $f\in l(\Lambda)$ where $\Lambda=w^{-1}\cdot \mcB_m(w)$ such that
\[\mcD_\Lambda(f)\leq C|\nabla li(0)|^2\sigma^{-1}_m\]
for some $C>0$ independent of $li,w,m$, and
\[f(\tau)=li\big(\Psi_\tau(q_c)\big),\qquad \forall \tau\in \underline\partial\Lambda.\]
\end{proposition}
\begin{proof}

Let's first construct a function  $g\in l(\Lambda)$ such that $\mcD_\Lambda(g)\leq C\sigma^{-1}_m$ for some $C>0$ independent of $w$, $m$,
\[g|_{\partial'_{1/2}\Lambda}=Li\big(\Psi_{\tilde{\partial}_1\Lambda}(\frac{q_1+q_2}{2})\big),\qquad g|_{\partial'_2\Lambda}=Li\big(\Psi_{\tilde{\partial}_2\Lambda}(\frac{q_1+q_2}{2})\big), \]
and
\[g(\tau)=Li\big(\Psi_\tau(\frac{q_1+q_2}{2})\big),\qquad \forall \tau\in \underline\partial_1\Lambda,\]
where $Li$ is the linear function on $\overline{q_1,q_2}$ such that $Li(q_1)=0$ and $Li(q_2)=1$.

We will start with $h_\Lambda \in l(\Lambda)$ (defined in Lemma \ref{lemma73}), and gradually improve the boundary value: \vspace{0.1cm}

\noindent\textbf{\textit{Algorithm}.}
1. Let $A_1=\emptyset$, and $g_1=h_\Lambda$.

2. For $k\geq 2$, if $A_{k-1}\neq \Theta_1(\Lambda)$, we pick $w\in \Theta_1(\Lambda)\setminus A_{k-1}$ such that there is $w'\in A_{k-1}$ satisfying $w\triangleleft w'$. Let $A_k=A_{k-1}\cup\{w\}$,We define $g_k'\in l(w^{-1}\cdot \Lambda)$ by (see Figure \ref{fig8})
\[g'_k(\tau)=g_{k-1}(w\cdot \tau)\cdot h'_{w^{-1}\cdot \Lambda}(\tau) +\big(Li(\Psi_wq_1)+\rho_w h_{w^{-1}\cdot\Lambda}( \tau)\big)\cdot \big(1-h'_{w^{-1}\cdot \Lambda}(\tau)\big)\]
for any $\tau\in w^{-1}\cdot\Lambda$, and define $g_k \in l(\Lambda)$ by
\[
g_k(\tau)=
\begin{cases}
g_{k-1}(\tau),\qquad \text{ if }\tau\notin w\cdot w^{-1}\cdot \Lambda,\\
g'_k(w^{-1}\tau),\qquad\hspace{-0.3cm} \text{ if }\tau\in w\cdot w^{-1}\cdot \Lambda.
\end{cases}
\]
3. If $A_k=\Theta_1(\Lambda)$, we stop the algorithm and let $g=g_k$.  \vspace{0.2cm}

\begin{figure}[htp]
	\includegraphics[width=6cm]{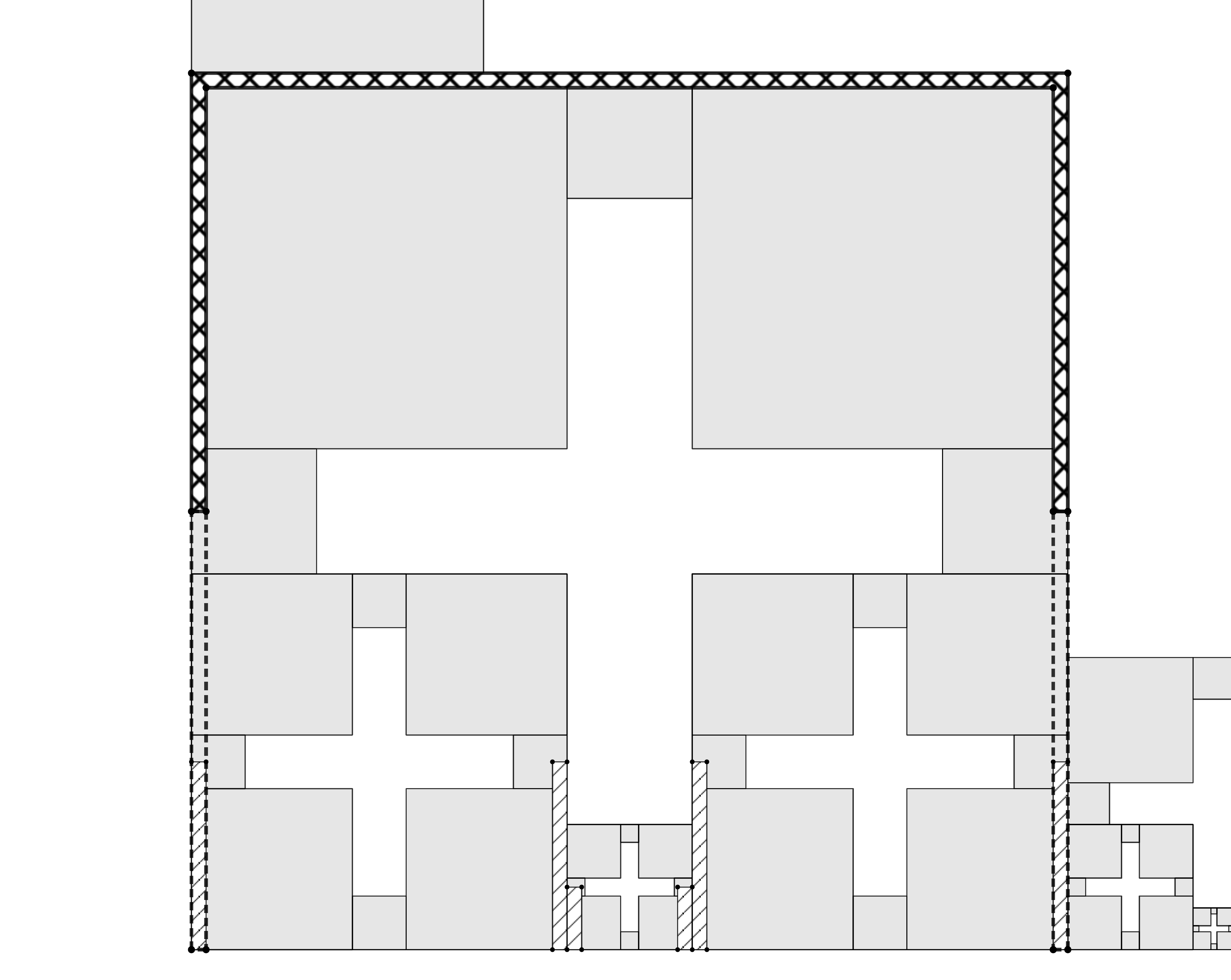}
	\caption{An illustration of the set $w\cdot w^{-1}\cdot\Lambda$.}
	\label{fig8}
\end{figure}
Clearly, the algorithm will stop when $k=M$ and $g=g_M$. One can then check that $g$ satisfies the desired boundary value.
  Furthermore, by the construction of Lemma \ref{lemma73} (b),  for any $2\leq k\leq M$, by letting $w$, $w'$, $g_{k-1}$, $g'_k$, $g_k$ be the same as in Step 2, we always have that $g_{k-1}(w\cdot\bullet)$ and $g_k(w\cdot\bullet)$ take the same boundary value at $\bigcup_{i\in S_0\setminus\{1\}}\partial_i(w^{-1}\cdot\Lambda)$, and in addition,
\[g_{k-1}(w\cdot \tau)=Li(\Psi_{w'}q_1)+\rho_{w'}h_{w'^{-1}\cdot \Lambda}(w'^{-1}\cdot w\cdot\tau),\qquad \forall \tau\in w^{-1}\cdot\Lambda.\]
So one can see for some $C_1, C_2>0$,
\[
\begin{aligned}
\big|\mcD_\Lambda(g_{k})-\mcD_\Lambda(g_{k-1})\big|&=\big|\mcD_{w^{-1}\cdot \Lambda}\big(g_{k}(w\cdot\bullet)\big)-\mcD_{w^{-1}\cdot \Lambda}\big(g_{k-1}(w\cdot \bullet)\big)\big|\\
&=\mcD_{w^{-1}\cdot \Lambda}\big(g_{k}'\big)+\mcD_{w^{-1}\cdot \Lambda}\big(\rho_{w'}h_{w'^{-1}\cdot \Lambda}(w'^{-1}\cdot w\cdot\bullet)\big)\\
&\leq C_1\rho_w^2\sigma_{m-[\log\rho_w/\log\rho_*]}^{-1}
\leq C_2\rho_{w}^{2-\alpha}\cdot\sigma_m^{-1},
\end{aligned}
\]
where the first inequality is due to Lemma \ref{lemma73} and the well-known estimate $\mathcal{E}(f g)\leq 2\mathcal{E}(f)\|g\|_\infty^2+2\mathcal{E}(g)\|f\|_\infty^2$ for any Dirichlet form $(\mathcal{E},\mathcal{F})$ and $f,g\in L^\infty\cap\mathcal{F}$, and the last inequality follows from Lemma \ref{lemma72} with the same constant $\alpha$.
Hence, by taking the sum of the above estimate over $1\leq k\leq M$ (letting $g_0\equiv 0$), we have
\[
\begin{aligned}
\mcD_\Lambda(g)
&\leq \sum_{w\in \Theta_1(\Lambda)}C_2\rho_{w}^{2-\alpha}\cdot\sigma_m^{-1}\\
&\leq C_2\sigma_m^{-1}\big(\sum_{k=0}^\infty\sum_{w\in (\underline\partial_1S)^k}\rho_w^{2-\alpha}\big)
\leq C_2\sigma_m^{-1}\sum_{k=0}^\infty\rho^{*k(1-\alpha)}\sum_{w\in(\underline\partial_1S)^k}\rho_w\leq C\sigma_m^{-1},
\end{aligned}\]
for some constant $C>0$, where $\rho^*=\max_{j\in S}\rho_j$. \vspace{0.2cm}

Next, we construct two basis functions $g',g''\in l(\Lambda)$, and $f$ can be constructed easily as the linear combination of  symmetric analogues of $g''$.

Let $g'\in l(\Lambda)$ be defined by $$g'(\tau) = \frac{1}{1 - \rho_{\tilde{\partial}_1\Lambda}}\cdot (g(\tau) - \frac{1}{2}\rho_{\tilde{\partial}_1\Lambda}),\qquad \forall \tau\in \Lambda,$$
so that $g'|_{\partial'_{1/2}\Lambda} = 0, g'|_{\partial'_2\Lambda} = 1$.\vspace{0.2cm}

\noindent\textbf{$N_0=4$ case:} Let $g_1\in l(\Lambda)$ be defined by
\[g_1(\tau)=\begin{cases}
	g'(\tau),&\text{ if }\dist(\Psi_{\tau}K,L_1)\leq \dist(\Psi_{\tau}K,L_3),\\
	g'(\Gamma^*_{1,4}\tau),&\text{ if }\dist(\Psi_{\tau}K,L_1)>\dist(\Psi_{\tau}K,L_3).
\end{cases}\]
Let $g''=(g_1\circ \Gamma^*_{1,2})\wedge (g_1\circ \Gamma_{1})$.

\noindent\textbf{$N_0=3$ case:} Let $g''\in l(\Lambda)$ be defined by
\[g''(\tau)=\begin{cases}
	1-g(\tau),&\text{ if }\dist(\Psi_{\tau}K,L_1)\leq \dist(\Psi_{\tau}K,L_3),\\
	1-g(\Gamma^*_{2,3}\tau),&\text{ if }\dist(\Psi_{\tau}K,L_1)>\dist(\Psi_{\tau}K,L_3).
\end{cases}\]\vspace{0.2cm}

Finally, we let $f= li\big(\Psi_{\tilde{\partial}_1\Lambda}(q_c) \big) + \sum_{j\in S_0}\big(li\big(\Psi_{\Gamma^*_{j}(\tilde{\partial}_1\Lambda)}(q_c) \big) - li\big(\Psi_{\tilde{\partial}_1\Lambda}(q_c) \big)\big)\cdot g''\circ \Gamma^{*-1}_{j}$. One can check $f$ has the desired boundary values, and its energy estimate follows from the energy estimate of $g$.
\end{proof}

\begin{corollary}\label{coro75}
	For each $x=(x_1,x_2)\in \mathbb{R}^2$ and $r>0$, let $\phi_{x,r}$ be a bump function defined by \[\phi_{x,r}(y)=\big(1-r^{-1}|y_1-x_1|-r^{-1}|y_2-x_2|\big)\vee 0, \quad\forall y=(y_1,y_2)\in \mathbb{R}^2.\]
	Also let $w\in \Lambda_*$ and $m\geq 1$.  Then one can find $f\in l\big(\mcB_m(w)\big)$ such that
	\[\mcD_{\|w\| + m, \mcB_m(w)}(f)\leq C(\rho_*^{\|w\|}/r)^2\sigma^{-1}_m\]
	for some $C>0$ independent of $x,r,w,m$, and
	\[f(\tau)=\phi_{x,r}\big(\Psi_\tau(q_c)\big),\qquad \forall \tau\in w\cdot\underline\partial\big(w^{-1}\cdot\mcB_m(w)\big).\]
\end{corollary}
\begin{proof}
 Let $g_{l,+}(y)=r^{-1}(y_l-x_l)$ and $g_{l,-}(y)=r^{-1}(x_l-y_l)$ for $l=1,2$. One applies Proposition \ref{prop74} to construct $f_{l,+},f_{l,-}\in l\big(w^{-1}\cdot \mcB_m(w)\big),l=1,2$ such that 	
\[\mcD_{w^{-1}\cdot\mcB_m(w)}(f_{l,+})\leq C_1(\rho_*^{\|w\|}/r)^2\sigma^{-1}_m,\quad \mcD_{w^{-1}\cdot\mcB_m(w)}(f_{l,-})\leq C_1(\rho_*^{\|w\|}/r)^2\sigma^{-1}_m,\]
for some $C_1>0$ independent of $x,r,w,m$, and
\[f_{l,+}(\tau)=g_{l,+}\big(\Psi_{w\cdot\tau}(q_c)\big),\quad f_{l,-}(\tau)=g_{l,-}\big(\Psi_{w\cdot\tau}(q_c)\big),\quad \forall \tau\in \underline\partial\big(w^{-1}\cdot\mcB_m(w)\big).\]
Finally, we let $f'=(1-f_{1,+}\vee f_{1,-}-f_{2,+}\vee f_{2,-})\vee 0$, and then define $f\in l\big(\mcB_m(w)\big)$ by $f(w\cdot \tau)=f'(\tau),\forall \tau\in w^{-1}\cdot\mcB_m(w)$. It is directly to check that this function $f$ satisfies the desired requirement of the corollary.
\end{proof}

\begin{proof}[Proof of Theorem \ref{thm71}]
	Let $n,m\geq 1$ and $w\in \Lambda_n$, we need to estimate $R_{n+m}\big(\mcB_m(w), \mcB_m(\mathcal N_2^c(w))\big)$. Assume $m$ large enough so that  $\rho_*^m \leq \frac{1}{4}c_0$ where $c_0$ is the same constant in condition $\textbf{(A3)}$. Let $k_0 = \big[(\frac{1}{4}c_0)^{-1}\big] + 1$ be an integer and choose $r = c_0\rho_*^{n}$.
	
	For each pair $j\in S_0, 0 \leq i \leq k_0$, denote $p_{i,j} = \Psi_w\big(\frac{(k_0 - i)\cdot q_{j} + i\cdot q_{j+1}}{k_0}\big)$. By Corollary \ref{coro75}, there is $f_{i,j}\in l(\Lambda_{n+ m})$ and  constant $C_1> 0$ independent of $n,m,w$ such that for any $v\in \Lambda_{n}$ and  $\tau\in v\cdot\underline{\partial}(v^{-1}\cdot \mcB_m(v))$, $f_{i,j}(\tau) = \phi_{p_{i,j},r}(\Psi_{\tau}q_c)$ and $\mcD_{n + m,\mcB_m(v)}(f_{i,j}) \leq C_1\sigma_m^{-1}$, where $\phi_{p_{i,j},r}$ is the same function in Corollary \ref{coro75}. Denote $g_j = \sum_{i = 0}^{k_0}f_{i,j}$ for $j\in S_0$. \vspace{0.2cm}
	
	\noindent\textit{Claim. For each $j\in S_0$, it holds that $g_j|_{\mcB_m(\mathcal N_2^c(w))} = 0$, $g_j|_{w\cdot(\Gamma_{j-1}^*\underline\partial_1w^{-1}\cdot\mcB_m(w))} \geq \frac{1}{2}$, and $\mcD_{n + m}(g_j) \leq C_2\sigma_m^{-1}$ for some constant $C_2 > 0$  independent of $n,m,w$. }\vspace{0.2cm}
	
In fact, $g_j|_{\mcB_m(\mathcal N_2^c(w))} = 0$ is an immediate consequence of \textbf{(A3)}. For each $\tau\in w\cdot\Gamma_{j-1}^*\underline{\partial}_1\big(w^{-1}\cdot \mcB_{m}(w)\big)$, there exists $0 \leq i \leq k_0$ such that $|y_1 - x_1| + |y_2 - x_2| \leq \rho_*^{n + m} + k_0^{-1}\rho_w \leq \frac{c_0}{2}\rho_*^{n}$, where $(x_1,x_2) = p_{i,j}, (y_1,y_2) = \Psi_{\tau}q_c$. Thus $f_{i,j}(\tau) \geq \frac{1}{2}$, and $g_j(\tau) \geq \frac{1}{2}$. As for the energy estimate, it suffices to estimate the energy of $f_{i,j}$ for each fixed $i,j$, first we have
$\sum_{v\in \Lambda_{n}}\mcD_{n + m,\mcB_m(v)}(f_{i,j}) \leq C_1\#\mcN_2(w)\sigma_m^{-1}:=C_3\sigma_m^{-1}$; second we have
	$$\begin{aligned}\sum_{v_1\neq v_2\in \Lambda_{n}}{\sum_{\begin{aligned}\substack{\tau_1\in \mcB_m(v_1), \tau_2\in \mcB_m(v_2)\\   \tau_1{\stackrel{n + m}{\sim}} \tau_2}\end{aligned}}}\big(f_{i,j}(\tau_1) - f_{i,j}(\tau_2)\big)^2 \leq C_4\rho_*^{-m}\cdot (r^{-1}\rho_*^{n+m})^2  \leq  C_5\rho_*^m
	\end{aligned}$$
	for some constant $C_4,C_5> 0$, which together gives that $\mcD_{n + m}(f_{i,j}) \leq C_3\sigma_m^{-1} + C_5\rho_*^m \leq C_6\sigma_m^{-1}$ for some $C_6>0$ for large $m$ by Lemma \ref{lemma72}. Thus the claim follows.
	
	Finally, taking $h = 2\sum_{j = 1}^{N_0}g_j$, by the claim, it holds that $h|_{\mcB_m(\mathcal N_2^c(w))} = 0, h|_{w\cdot\underline\partial(w^{-1}\cdot \mcB_m(w))}\geq 1$, and $\mcD_{n + m}(h) \leq C\sigma_m^{-1}$ for some $C>0$ independent of $n,m,w$. This gives that  $R_m=\inf\big\{R_{n+m}\big(\mcB_m(w), \mcB_m(\mathcal N_2^c(w))\big), n\geq 1, w\in \Lambda_n \big\}\geq C^{-1}\sigma_m$ and the theorem follows.
\end{proof}

\subsection{Proof of Lemma \ref{lemma73}}
We prove Lemma \ref{lemma73} for $N_0=4$ and $N_0=3$ cases separately. In particular, $N_0=4$ is an easy case, while we need to deal with a few possibilities when we deal with $N_0=3$ case.\vspace{0.2cm}

\noindent\textbf{The $N_0=4$ case.}

\begin{proof}[Proof of Lemma \ref{lemma73} for $N_0=4$]
It suffices to prove the lemma for $\Lambda=\Lambda_n$ for $n\geq 1$. In fact, for a general partition $\Lambda$ such that for some $n\geq 0$, $\rho_\tau\in (\rho_*^{n+2},\rho_*^n]$, $\forall\tau\in \Lambda$, one only needs to consider $\Lambda_{n+2}$ which is finer than $\Lambda$: let $h$ be the function to be defined (for (a) or (b)) on $\Lambda_{n+2}$; one can find $\hat{h}$ on $\Lambda$ such that $\hat{h}$ satisfies the desired boundary values and $\min\big\{h(w):w\in (\tau\cdot W_*)\cap \Lambda_{n+2}\big\}\leq \hat{h}(\tau)\leq \max\big\{h(w):w\in (\tau\cdot W_*)\cap \Lambda_{n+2}\big\}$ for any rest $\tau\in \Lambda$; it is easy to check that $\hat{h}$ satisfies the desired energy estimate by suitably adjusting the constant $C$.

(a). By a same reason as in the above arguments and by Theorem \ref{thm61}, Proposition \ref{prop62} and Lemma \ref{lemma43}, for any partition $\Lambda$ such that $\rho_\tau\in (\rho_*^{n+1},\rho_*^{n-1}]$ for each $\tau\in \Lambda$, one can find $g_\Lambda,g'_\Lambda, g''_\Lambda\in l(\Lambda)$ and constant $C_1>0$ independent of $\Lambda,n$ such that $0\leq g_\Lambda, g'_\Lambda, g''_\Lambda\leq 1$ and
\begin{align*}
g_\Lambda|_{(\partial'_{7/2}\Lambda)\cup (\partial_4 \Lambda)}=0,\ g_\Lambda|_{\partial'_2\Lambda}=1,\ \mcD_\Lambda(g_\Lambda)\leq C_1\sigma_n^{-1},\\
g'_\Lambda|_{\partial_4\Lambda}=0,\ g'_\Lambda|_{\partial_2\Lambda}=1,\ \mcD_\Lambda(g'_\Lambda)\leq C_1\sigma_n^{-1},\\
g_\Lambda|_{\partial'_{1/2}\Lambda}=0,\ g_\Lambda|_{(\partial'_{3}\Lambda)\cup (\partial_2 \Lambda)}=1,\ \mcD_\Lambda(g''_\Lambda)\leq C_1\sigma_n^{-1}.
\end{align*}
Let $\hat{g}_\Lambda=g'_\Lambda\circ \Gamma^*_3$ (which will be used in part (b)).
We write $Li_{j,k}=Li\big(\Psi_{j\cdot \tilde{\partial}_k(j^{-1}\cdot \Lambda_n)}(\frac{q_1+q_2}{2})\big)$ for $j\in \partial_1 S$ and $k=1,2$ for short.  One then define $h_{\Lambda_n}\in l(\Lambda_n)$ as for any $j\in S,\tau\in j^{-1}\cdot\Lambda_n$,
\[
h_{\Lambda_n}(j\cdot \tau)=
\begin{cases}
Li_{1,1}+(Li_{1,2}-Li_{1,1})\cdot g_{1^{-1}\cdot\Lambda_n}(\tau),\quad &\text{ if }j=1,\\
Li_{j,1}+(Li_{j,2}-Li_{j,1})\cdot g'_{j^{-1}\cdot\Lambda_n}(\tau),\quad &\text{ if }j\in \partial_1S\setminus \{1,2\},\\
Li_{2,1}+(Li_{2,2}-Li_{2,1})\cdot g''_{2^{-1}\cdot\Lambda_n}(\tau),\quad &\text{ if }j=2,\\
Li_{1,1},\quad &\text{ if }j\in \partial_4S\setminus \{1,4\},\\
Li_{2,2},\quad &\text{ if }j\in \partial_2S\setminus\{2,3\},\\
h_{\Lambda_n}\circ\Gamma^*_{1,4}(j\cdot\tau), \quad&\text{ if }j\in \partial_3S.
\end{cases}
\]
It is directly to check that $h_{\Lambda_n}$ satisfies the desired energy estimate and boundary conditions.

(b). Let $Li'$ be a linear function on $[\Psi_1q_4,\Psi_4q_1]$ such that $Li'(\Psi_1q_4)=0$, $Li'(\Psi_4q_1)=2$. Define $h'_{\Lambda_n}\in l(\Lambda_n)$ as follows, where $j\in S,\tau\in j^{-1}\cdot\Lambda_n$,
\[
h''_{\Lambda_n}(j\cdot \tau)=
\begin{cases}
	0,\quad &\text{ if }j\in \partial_1S,\\
	2,\quad &\text{ if }j\in \partial_3S,\\
	Li'(\Psi_jq_1)+\big(Li'(\Psi_jq_4)-Li'(\Psi_jq_1)\big)\hat{g}(\tau),\quad &\text{ if }j\in \partial_4S\setminus\{1,4\},\\
	h''_{\Lambda_n}\circ\Gamma^*_{1,2}(j\cdot\tau), \quad&\text{ if }j\in \partial_2S.
\end{cases}
\]
Then, we define $h'_{\Lambda_n}=\min\{1,h''_{\Lambda_n}\}\in l(\Lambda_n)$, which is the desired function.
\end{proof}

\noindent\textbf{The $N_0=3$ case.} \vspace{0.2cm}

To make things clear, we assume $q_1=(0,0),q_2=(1,0),q_3=(1/2,\sqrt{3}/2)$ be the three vertices of a triangle $\mcA$. So $L_1$ is the bottom boundary, $L_2$ is the right boundary and $L_3$ is the left boundary. As before, for $i\in S_0=\{1,2,3\}$, we assume $q_i$ is the fixed point of $\Psi_i$ by ordering $\Psi_j,j\in S$ suitably.

As in the proof for $N_0=4$ case, we only need to take care of $\Lambda_n$ (in the following lemmas). Then, for a general partition $\Lambda$ such that  for some $n\geq 0$,  $\rho_\tau\in (\rho_*^{n+2},\rho_*^n]$, $\forall\tau\in \Lambda$, a same conclusion holds with the constant $C$ slightly modified.

\begin{lemma}\label{lemma76}
Let $q=(1-t)q_1+tq_3$ for some $0\leq t<1$ and let $L'=\overline{q_1,q}$ be a sub-segment of $L_3$. Then there exists a constant $C_1>0$, such that for each $n\geq1$, one can find $f\in l(\Lambda_n)$ such that $f(w)=1$ for any $w\in \{w'\in\Lambda_n:\Psi_{w'}K\cap L'\neq \emptyset\}$, and $f(w)=0$ for any $w\in \partial_2\Lambda_n$, and $\mcD_n(f)\leq C\sigma_n^{-1}$.
\end{lemma}
\begin{proof}
Let's fix $m\geq 1$ and let $t_m=1-\rho_3^m$. In the following, for $n>m$, we construct a function $f_n\in l(\Lambda_n)$ that satisfies the boundary values and the energy estimate. The general case will then follow immediately. Before the construction, we  remark that the constant $C$  will depend on $t$, but will be independent of $n$.

For $n>m$ and $0\leq l<m$, we let
\[
Q_{l,n}=\bigcup_{w\in {(3)}^l\cdot \partial_1S}w\cdot w^{-1}\cdot \Lambda_n.
\]
Then by Theorem \ref{thm61}, Proposition \ref{prop62} and Lemma \ref{lemma43}, there is a function $g_{l,n}\in l(Q_{l,n})$ so that
\[g_{l,n}|_{\partial_3\Lambda_n\cap Q_{l,n}}=1,\quad g_{l,n}|_{\partial_2\Lambda_n\cap Q_{l,n}}=0,\quad \mcD_{n,Q_{l,n}}(g_{l,n})\leq C_1\sigma_n^{-1}\]
for some $C_1>0$ independent of $n$.

We also let $P_{n}=(3)^m\cdot (3)^{-m}\cdot \Lambda_n$ for $n>m$, where we write $(3)^m=333\cdots 3$ ($3$ repeats $m$ times). Then by the same reason as above, there is a function $h_{n}\in l(P_{n})$ so that
\[h_{n}\big((3)^m\cdot \tilde{\partial}_1\big((3)^{-m}\cdot \Lambda_n\big)\big)=1,\quad h_{n}|_{\partial_2\Lambda_n\cap P_{n}}=0,\quad \mcD_{n,P_{n}}(h_{n})\leq C_2\sigma_n^{-1}\]
for some $C_2>0$ independent of $n$.

We glue $f_{l,n},l<m$ and $h_{n}$ to get the desired function $f_n\in l(\Lambda_n)$: for any $w\in \Lambda_n$, let
\[
f_{n}(w)=\begin{cases}
	g_{l,n}(w), &\text{ if }w\in Q_{l,n}, 0\leq l<m,\\
	h_{n}(w), &\text{ if }w\in P_{n},\\
	0,  &\text{ otherwise and }dist(q_2,\Psi_wK)<dist(q_1,\Psi_wK),\\
	1,  &\text{ otherwise and }dist(q_1,\Psi_wK)<dist(q_2,\Psi_wK).\\
\end{cases}
\]
One can then check that $f_{n}$ satisfies the desired boundary values and the energy estimate.
\end{proof}

In the following, for each $i\in \partial_1S\setminus \underline{\partial}_1S$, we let $j_{-1}(i)\neq j_1(i)\in \underline{\partial}_1S$, and $j_{-2}(i)\neq j_2(i)\in  (\partial_1S\setminus \underline{\partial}_1S)\setminus \{i\}$ (if exist) such that $j_{-2}(i)\sim_{S} j_{-1}(i)\sim_{S} i\sim_{S} j_1(i)\sim_{S} j_2(i)$, and $\Psi_{j_{-1}(i)}K$,  $\Psi_{j_{-2}(i)}K$ are on the left of $\Psi_{i}K$. For short, we write $j_0(i)=i$, and for each $n\geq 1$, $L_{l,l'}^{(n)}=j_l(i)\cdot \partial_{l'}(j_l(i)^{-1}\cdot \Lambda_n)$, $L'^{(n)}_{l,l'}=j_l(i)\cdot \partial'_{l'}\big(j_l(i)^{-1}\cdot \Lambda_n\big)$ for $l\in\{0,\pm 1,\pm 2\}$ and $l'\in S_1$.

\begin{lemma}\label{lemma77}
There exists a constant $C>0$ such that for each $n\geq 1$, $i\in \partial_1S\setminus \underline{\partial}_1S$, there is a function $f_i\in l(\Lambda_n)$ supported on $\bigcup_{l=-2,-1,1,2}j_l(i)\cdot j_l(i)^{-1}\cdot \Lambda_n$ such that $f_i|_{L'^{(n)}_{-1,2}}=f_i|_{L'^{(n)}_{1,7/2}}=1$ and $\mcD_n(f_i)\leq C\cdot \sigma_n^{-1}$.
\end{lemma}
\begin{proof}
We only focus on the case that $j_{-2}(i)\neq j_2(i)$ exist. In fact, the lemma is easier to prove if one of them does not exist (or both do not exist). For short, we drop the supscript $(n)$ of $L^{(n)}_{l,l'}$ and $L'^{(n)}_{l,l'}$. Our goal is to construct $g_i: \bigcup_{l=0,1,2}j_l(i)\cdot j_l(i)^{-1}\cdot \Lambda_n\to [0,1]$ such that
\[g_i|_{L_{0,3}\cup L'_{1,7/2}}=1\text{ and }g_i|_{L_{2,3}}=0,\]
with estimate $\mcD_{n,\bigcup_{l=0,1,2}j_l(i)\cdot j_l(i)^{-1}\cdot \Lambda_n}\leq C_1\sigma_n^{-1}$ for some $C_1>0$ independent of $n,i$.
We also construct, by a same argument, $g_i':\bigcup_{l=0,-1,-2}j_l(i)\cdot j_l(i)^{-1}\cdot \Lambda_n\to [0,1]$ such that
\[g_i'|_{L_{0,2}\cup L'_{-1,2}}=1\text{ and }g_i'|_{L_{-2,2}}=0.\]
Then, one can define the desired $f_i\in l(\Lambda_n)$ by taking values of $g_i$ or $g_i'$ on $\bigcup_{l=-2,-1,1,2}j_l(i)\cdot j_l(i)^{-1}\cdot \Lambda_n$, and taking the minimum of $g_i,g_i'$ on $i\cdot i^{-1}\cdot \Lambda_n$ and $0$ outside.

For the construction of $g_i$, we need to consider all the possibilities based on the sizes of the cells. We only explain why the construction is feasible. Readers can fulfill the details easily. \vspace{0.15cm}
	
\noindent\textit{Case 1}. $\rho_i<\rho_{j_1(i)}$. In this case, we apply Lemma \ref{lemma76} to construct $g_i$ on $j_{1}(i)\cdot j_{1}(i)^{-1}\cdot \Lambda_n$ and extend $g_i$ such that $g_i=1$ on $i\cdot i^{-1}\cdot \Lambda_n$ and $g_i=0$ on $j_{2}(i)\cdot j_{2}(i)^{-1}\cdot \Lambda_n$ with desired energy estimate.

\noindent\textit{Case 2}. $\rho_i=\rho_{j_1(i)}$. In this case, we apply Theorem \ref{thm61} and symmetry to $\bigcup_{l=0,1}j_{l}(i)\cdot j_{l}(i)^{-1}\cdot \Lambda_n$. In particular, we can enable $g_i=0$ on $j_{2}(i)\cdot j_{2}(i)^{-1}\cdot \Lambda_n$.

\noindent\textit{Case 3}. $\rho_i>\rho_{j_1(i)}=\rho_{j_1(2)}$. This is similar to Case $2$. This time, we apply Theorem \ref{thm61} and symmetry to $\bigcup_{l=1,2}j_{l}(i)\cdot j_{l}(i)^{-1}\cdot \Lambda_n$. In particular, we can enable $g_i=1$ on $i\cdot i^{-1}\cdot \Lambda_n$.

\noindent\textit{Case 4}. $\rho_i>\rho_{j_1(i)}>\rho_{j_2(i)}$. Just  as in Case 1, we apply Lemma \ref{lemma76} to $j_{1}(i)\cdot j_{1}(i)^{-1}\cdot \Lambda_n$.
\end{proof}

\begin{proof}[Proof of Lemma \ref{lemma73} for $N_0=3$]
Still as before, we only need to consider $\Lambda_n$, $n\geq 1$. The constructions of $h_{\Lambda_n}$ and $h'_{\Lambda_n}$ are essentially same as that of the $N_0=4$ case. The only difference is that we need to use Lemma \ref{lemma77} to adjust functions on upside-down triangles.
\end{proof}

\appendix
\renewcommand{\appendixname}{Appendix~\Alph{section}}
\section{Proof of Proposition \ref{prop27}} \label{AppendixA}

In this appendix, we prove Proposition \ref{prop27}. First, we present some geometric properties of the regular polygon carpets.

\begin{lemma}\label{lemmaa1}
  There exists $c > 0$ such that  for any $n>m\geq 0$ and $v\in \Lambda_m$,
  \begin{equation}\label{eq22}
  \min\big\{\dist(\Psi_wK,\Psi_vL_1) > 0:w\in \mcB_{n - m}(v)\big\} \geq c\rho_*^{n},
  \end{equation}
  where $\min\emptyset = +\infty$.
\end{lemma}
\begin{proof}
  By the symmetry condition, there is at least one $w\in \Lambda_2$ such that $\Psi_wK\cap L_1 = \emptyset$, which gives $c_1: = \min\big\{\dist(\Psi_wK,L_1) > 0:w\in \Lambda_2\big\} \in (0,\infty)$. When $N_0$ is odd, for $k\in S_0$, we define $\ell_k$ to be the straight line passing through $q_k$, parallel to the side $L_i$ opposite to $q_k$. Let
  $c_2 = \min\{\dist(\Psi_wK,\ell_1)>0:w\in \Lambda_2\}\in (0,\infty).$ Then take $c_0 = c_1\wedge c_2$.

First, we consider the case $m=0$.
 By the definition of $c_1$, we have
  $\min\{\dist(\Psi_wK,L_1) > 0:w\in \Lambda_1\} \geq c_1 > c_0\rho_*.$
  So (\ref{eq22}) holds for $n = 1$. Assume  (\ref{eq22}) holds for $n$. By induction, we turn to prove (\ref{eq22}) for $n+1$.
 Let $w\in \Lambda_{n+1}$ and $\Psi_wK\cap L_1 = \emptyset$.  Pick $\tau\in \Lambda_{n}$ so that $w\in \mcB_1(\tau)$, we only need to consider the case that  $\Psi_\tau K\cap L_1 \neq \emptyset$. By Lemma \ref{lemma23}, $\Lambda_2$ is finer than $\tau^{-1}\cdot\mcB_1(\tau)$, so we can choose $w'\in \Lambda_2$ so that $\Psi_{\tau w'}K\subset \Psi_{w}K$ and $\dist(\Psi_{\tau w'}K,L_1) = \dist(\Psi_{w}K,L_1)$. There are two possible cases:  1. $\Psi_{\tau}K\cap L_1 = \Psi_{\tau}L_k$ for some $k\in S_0$, then from the definition of $c_1$, we have
  $\dist(\Psi_{w}K,L_1) =\rho_\tau\dist(\Psi_{w'}K,L_k) > c_1\rho_*^{n+1} \geq c_0\rho_*^{n+1}$;  2. $\Psi_{\tau}K\cap L_1 = \Psi_{\tau}q_k$ for some $k\in S_0$, then from the definition of $c_2$, we still have
  $\dist(\Psi_{w}K,L_1) =\rho_\tau\dist(\Psi_{w'}K,\ell_k) > c_2\rho_*^{n+1} \geq c_0\rho_*^{n+1}$ as desired.

  Next, we consider the general $m>0$ case. Denote $I_n$ the left side of (\ref{eq22}), then we have
  $I_n=\rho_v\cdot\min\{\dist(\Psi_wK,L_1) > 0:w\in v^{-1}\cdot \mcB_{n-m}(v)\}$. By  Lemma \ref{lemma23}, $\Lambda_{n-m+1}$ is finer than $v^{-1}\cdot\mcB_{n-m}(v)$, so we have $I_n\geq \rho_v\cdot\min\{\dist(\Psi_wK,L_1) > 0:w\in \Lambda_{n-m+1}\}> \rho_*^{m+1}\cdot c_0\rho_*^{n-m+1}$, where the last inequality follows from the previous argument. This proves that $I_n>c_0\rho_*^{n+2}$ as desired.
\end{proof}

\begin{lemma}\label{lemmaa2}
  Suppose $K$ is a bordered polygon carpet. Let $n > m > 0, w\sm v\in \Lambda_m$ and $w'\in \mcB_{n - m}(w), v'\in \mcB_{n - m}(v)$. There exists $c > 0$ such that $\dist(\Psi_{w'}K,\Psi_{v'}K) < c\rho_*^{n}$ implies $d_n(w',v') \leq 2$.
\end{lemma}
\begin{proof}

By the boundary included condition, $N_0=3,4$. Consider the case $N_0=4$.
Let $c=\tilde{c}\wedge \rho_*$, where $\tilde{c}$ is the constant in Lemma \ref{lemmaa1}. By Lemma \ref{lemmaa1}, both $\Psi_{w'}K\cap\ell\neq \emptyset$ and $\Psi_{v'}K\cap \ell \neq \emptyset$, where $\ell$ is a straight line passing through $\Psi_w \partial K\cap\Psi_v\partial K$. In addition,  we could pick $x\in \Psi_{w'}K\cap\ell$ and $y\in \Psi_{v'}K\cap\ell$ so that $|x-y|\leq \dist(\Psi_{w'}K,\Psi_{v'}K) <c\rho_*^n\leq \rho_*^{n+1}$. Then by the boundary included condition, we could pick $w''$ in $\partial\mcB_{n-m}(w)\cup\partial\mcB_{n-m}(v)$ such that $\{x,y\}\cap\Psi_{w''}K\neq \emptyset$ and $\Psi_{w''}K\cap\ell\neq \emptyset$. Since $\rho_{w''}>\rho_*^{n+1}$, we further have both $x,y\in\Psi_{w''}K$. This gives that $w''\sn w'$ and $w''\sn v'$, thus $d_n(w',v') \leq 2$. The $N_0=3$ case follows in a similar way by a suitable adjustment of $c$.
\end{proof}

\begin{proof}[Proof of Proposition \ref{prop27}]

\textbf{(A1).} This condition is just the open set condition. $\mcA^{o}$ is an open set that satisfies the requirement.

\textbf{(A2).} Since $K$ is connected, $(\Lambda_n,\sn)$ is also connected for any $n\geq 0$.

\textbf{(A4).} Since by Lemma \ref{lemma23}, $w^{-1}\cdot\mcB_{m}(w)$ is finer than $\Lambda_{m-1}$, we need only to prove that $\partial \Lambda_n \neq \Lambda_n$ for some $n\geq 1$.

If $\Lambda_n = \partial \Lambda_n$ for any $n\geq 1$, we have $K\subset \partial \mcA$. So $K = \partial K$ and $\dim_H(\partial K) = d_H$, a contraction to Proposition \ref{prop26}.

Note that if the boundary included condition holds, it is direct to check that $\partial \Lambda_2\neq \Lambda_2$.

\textbf{(A3).} We divide the proof into two cases.\vspace{0.2cm}

\noindent\textit{Case 1. $K$ is a perfect polygon carpet.}\vspace{0.2cm}

 In this case $\rho_i = \rho_*$ for any $i\in S$, so $\Lambda_n = W_n$. Define
  $$c_0 =\min\big\{\rho_*^{-n -1}\dist(\Psi_{wi}K,\Psi_{vj}K) > 0: w\sn v \text{ in } \Lambda_n, n\geq 0, i,j\in S\big\},$$
where $c_0>0$ since there are only finite intersection types of $\Psi_wK$ and $\Psi_vK$. Then we see that $\min\big\{\dist(\Psi_wK,\Psi_vK)>0:w,v\in\Lambda_n\big\} = c_0\rho_*^{n}$, which implies \textbf{(A3)}.

\vspace{0.2cm}

\noindent\textit{Case 2. $K$ is a bordered polygon carpet.}\vspace{0.2cm}

We define
$$c_1 = \min\big\{\dist(\Psi_wK,\Psi_vK)>0:w,v\in \Lambda_2\big\},$$
and $c_0 = c_1\wedge c$, where $c$ is the same constant in Lemma \ref{lemmaa2}.

When $w,v\in \Lambda_1$ with $\dist(\Psi_wK,\Psi_vK) > 0$, then by the definition of $c_1$, we have
$$\dist(\Psi_wK,\Psi_vK) \geq c_1 > c_0\cdot\rho_*.$$
So \textbf{(A3)} holds with $n = 1$. Suppose \textbf{(A3)} holds for all $k \leq n$. By induction, we will prove \textbf{(A3)} with $k = n+1$. Otherwise, there should exist $w,v\in \Lambda_{n+1}$ such that $d_{n+1}(w,v)>2$ and $\dist(\Psi_wK,\Psi_vK) < c_0\rho_*^{n+1}$. We will consider two cases to see it is impossible. \vspace{0.2cm}

\noindent\textit{Case 2.1. There is  $\tau\in \Lambda_n$ so that $w,v\in \mcB_{1}(\tau)$.}\vspace{0.2cm}

In this case, we have
$$\dist\big(\Psi_\tau^{-1}\Psi_wK,\Psi_\tau^{-1}\Psi_vK\big) = \rho_\tau^{-1}\dist\big(\Psi_wK,\Psi_vK\big) < c_0\rho_*^{-(n+1)}\rho_*^{n+1} = c_0.$$

Since $\Lambda_2$ is finer than $\tau^{-1}\cdot\mcB_{1}(\tau)$ by Lemma \ref{lemma23}, we can choose $w',v'\in \Lambda_2$ so that $\Psi_{\tau w'}K\subset \Psi_wK,\Psi_{\tau v'}K\subset \Psi_vK$ and
$$ \dist(\Psi_{w'}K,\Psi_{v'}K) = \dist(\Psi_\tau^{-1}\Psi_wK,\Psi_\tau^{-1}\Psi_vK) < c_0 \leq c_1.$$
By the definition of $c_1$, we then have $\dist(\Psi_{w'}K,\Psi_{v'}K) = 0$, which implies $\dist(\Psi_{w}K,\Psi_{v}K) = 0$, a contraction to $d_{n+1}(w,v)>2$.\vspace{0.2cm}

\noindent\textit{Case 2.2. There is no $\tau\in \Lambda_{n}$ so that $w,v\in \mcB_{1}(\tau)$.}\vspace{0.2cm}

In this case, we choose $m$ to be the largest number such that there is $\tau\in \Lambda_m$ with $w,v\in \mcB_{n + 1 - m}(\tau)$. Then we have $m \leq n-1$. Next we choose $\tau^{(1)}\neq \tau^{(2)}\in \mcB_1(\tau)$ such that $w\in \mcB_{n-m}(\tau^{(1)})$ and $v\in \mcB_{n-m}(\tau^{(2)})$. Write $\tau^{(1)}=\tau w'$ and $\tau^{(2)}=\tau v'$, then
 $$ \dist(\Psi_{w'}K,\Psi_{v'}K) \leq \dist(\Psi_{\tau}^{-1}\Psi_wK,\Psi_{\tau}^{-1}\Psi_vK) < c_0\rho_{\tau}^{-1}\rho_*^{n+1} < c_0\rho_*^{n-m}<c_1.$$
Since by Lemma \ref{lemma23}, $\Lambda_2$ is finer than $\tau^{-1}\cdot\mcB_1(\tau)$, so we have $\tau^{(1)}\stackrel{m+1}{\sim}\tau^{(2)}\in \Lambda_{m+1}$ by the definition of $c_1$, a contraction to Lemma \ref{lemmaa2}.

\end{proof}

\section{Proof of Proposition \ref{prop32} and \ref{prop33}}\label{AppendixB}
In this appendix, we provide a proof of Proposition \ref{prop32} and \ref{prop33} following Kusuoka and Zhou's strategy \cite{KZ}, for a regular polygon carpet $K$. It should be pointed out that the proof does not involve the symmetry condition of $K$.

\subsection{Basic facts.}
For convenience of readers, we collect some basic facts in this subsection.

From time to time, throughout this section, for $w,v\in \Lambda_n$, we will abbreviate $w\sim v$ instead of $w\sn v$ when no confusion caused.
Also, for $w\in \Lambda_*$, recall that $\mathcal{N}_k(w)=\{v\in \Lambda_{\|w\|}: d_{\|w\|}(w,v)\leq k\}$. Clearly, for each fixed $k$, there is a uniform upper bound of $\#\mathcal{N}_k(w)$.

\begin{lemma}\label{lemmab1}
	There is $M_0<\infty$ such that  $\#\mcN_k(w)\leq M_0^k$ for any $k\geq 1$ and $w\in \Lambda_*$.
\end{lemma}
\begin{proof}
	It suffices to choose $M_0=\frac{4\pi \diam^2(\mathcal A)}{\rho_*^2m(\mathcal A)}$, where $m$ is the Lebesgue measure on $\mathbb{R}^2$.
	
	For $k=1$, we let $n\geq 0,w\in \Lambda_n$ and $x\in \Psi_wK$. Noticing that $\rho_*^{n} \geq \rho_v,\forall v\in \Lambda_n$,  we have
	$$\bigcup_{v\in \mcN_{1}(w)}\Psi_v\mathcal {A}^o\subset \bigcup_{v\in \mcN_1(w)}B\big(x,\diam(\Psi_v\mathcal {A}) + \diam(\Psi_w\mathcal {A})\big) \subset B\big(x,2\rho_*^{n}\diam(\mathcal {A})\big),$$
	where $B(x,r) = \{y\in \R^2:|x - y| < r\}$. Also noticing that $\rho_v>\rho_*^{n+1},\forall v\in \Lambda_n$, we see
	\[\#\mcN_1(w)\cdot \rho_*^{2(n+1)}m(\mcA)\leq \sum_{v\in \mcN_1(w)}m(\Psi_v\mcA)\leq m\big(B\big(x,2\rho_*^{n}\diam(\mathcal {A})\big)\big)=\pi\big(2\rho_*^{n}\diam(\mathcal {A})\big)^2.\]
	If follows that $\#\mcN_1(w)\leq M_0$.
	
	For general $k\geq 1$, it suffices to notice that $\mathcal{N}_k(w)=\bigcup_{v\in \mathcal{N}_{k-1}(w)}\mcN_1(v)$.
\end{proof}

\begin{lemma}\label{lemmab2}
	For $k\geq 1$, there exists $C>0$ depending only on $k$ and $M_0$ such that
	$$\sum_{w\in\Lambda_n}\sum_{v\in \mcN_{k}(w)}\big(f(w) - f(v)\big)^2\leq C\mcD_n(f),\qquad \forall n\geq 1,f\in l(\Lambda_n).$$
\end{lemma}
\begin{proof}
	For any $w,v\in \Lambda_n$ with $d_n(w,v) \leq k$, choose a path $\boldsymbol {\tau}:=\{\tau^{(i)}\}_{i = 0}^{l}\subset \Lambda_n$ so that $\tau^{(0)} = w, \tau^{(l)} = v$ and $\tau^{(i)}\sim \tau^{(i - 1)}$ for $i = 1,\cdots,l$, $l\leq k$. Then
	$$\big(f(w) - f(v)\big)^2 \leq k\cdot\sum_{i = 1}^{l}\big(f(\tau^{(i)} - f(\tau^{(i - 1)})\big)^2.$$
	Since each pair $\tau\sim \tau'$ appears in at most $C_1=k\cdot M_0^{k-1}$ different paths $\boldsymbol\tau$ with length at most $k$, we then have
	$\sum_{w\in \Lambda_n}\sum_{v\in \mcN_k(w)}\big(f(w) - f(v)\big)^2\leq 2kC_1\mcD_n(f),$ where $2$ appears since we count each pair $w,v\in A$ with $d_n(w,v)\leq k$ twice in the left side of the inequality.
\end{proof}

\subsection{Proof of $R_m \geq C\rho_*^{(d_H-2)m}$}
The estimate $R_m \geq C\rho_*^{(d_H-2)m}$ in Proposition \ref{prop32} is linked with the well known estimate of walk dimension $d_W\geq 2$. In particular, the estimate can be derived as an immediate consequence of \cite[Lemma 4.6.15]{ki5}. For the convenience of the readers, we still provide a direct proof here.

\begin{proposition}\label{propb3} All the Poincare constants $\lambda_m$, $R_m$ and $\sigma_m$ for $m\geq 1$ are positive and finite. In addition,
  there is $C > 0$ such that $R_m \geq C\rho_*^{(d_H-2)m}$ for any $m\geq 1$.
\end{proposition}
\begin{proof}
It is straightforward to see $\lambda_m, \sigma_m>0$ for all $m\geq 1$ from their definitions.

Now we prove $R_m \geq C\rho_*^{(d_H-2)m}$.  Let $w\in \Lambda_n$, $n\geq 1$. For each $w'\in\mcN_2^c(w)$, by \textbf{(A3)}, we have $\dist(\Psi_{w'}K,\Psi_wK) \geq c_0\rho_*^n$, where $c_0$ is the same constant in  \textbf{(A3)}. Let $f\in L^2(K,\mu)$ be defined as
  $$f(x) = 1\wedge\big(\dist(x, \Psi_wK)/(c_0\rho_*^n)\big).$$
Immediately, $f|_{\Psi_wK} = 0$ and $f|_{\bigcup_{w'\in\mathcal{N}_2^c(w)}\Psi_{w'}K} = 1$, and $f$ is a Lipschitz function with $Lip(f) = (c_0\rho_*^n)^{-1}$. Define $\tf\in l(\Lambda_{n+m})$ as $\tf = \pi_{n+m}^{-1}\circ P_{n+m}f,$ and  then we have $\tf|_{\mcB_m(w)} = 0,\tf|_{\mcB_m(\mcN_2^c(w))} = 1$.
When $v\sim v'\in \Lambda_{n + m}$, we have $\diam(\Psi_vK\cup \Psi_{v'}K) \leq (\rho_v + \rho_{v'})\diam(K) \leq 2\rho_*^{n+m}\diam(K)$. So
  $$
  \big|\tf(v) - \tf(v')\big|
  \leq \int_K|f\circ\Psi_v - f\circ\Psi_{v'}|d\mu \leq Lip(f)\cdot \diam(\Psi_vK\cup \Psi_{v'}K) \\
  \leq 2 c_0^{-1}\rho_*^{m}\diam(K).$$
Then by Lemma \ref{lemmab1} and Proposition \ref{prop24}, we have
  $$\begin{aligned}\mcD_{n+m}(\tf)
  &\leq\sum_{v\in \mcB_m(\mcN_2(w))}\sum_{v'\sim v\in \Lambda_{n + m}}\big(\tf(v') - \tf(v)\big)^2 \\
  &\leq M_0^2\cdot C_1\rho_*^{-md_H}\cdot M_0 \cdot(2 c_0^{-1}\rho_*^{m}\diam(K))^2 = 4C_1 M_0^3 c_0^{-2}\diam^2(K)\rho_*^{(2-d_H)m},
  \end{aligned}$$
 where $C_1$ is the constant in Proposition  \ref{prop24}. This gives $R_{n + m}\big(\mcB_m(w),\mcB_m(\mcN_2^c(w))\big) \geq C\rho_*^{(d_H-2)m}$ with $C = (4C_1M_0^3c_0^{-2}\diam^2(K))^{-1}$.
\end{proof}

\subsection{Proof of $\lambda_{n+m+k_0} \geq CR_m\lambda_n$} The most difficult part of Proposition \ref{prop32} is the inequality $\lambda_{n+m} \geq CR_m\lambda_n$. We will closely follow Kusuoka-Zhou's idea \cite{KZ} in this part to prove a very close estimate $\lambda_{n+m+k_0} \geq CR_m\lambda_n$, where $k_0\geq 3$ is a fixed number independent of $m,n$.

For convenience, we will always consider $\lambda_m(\emptyset)$ in the proof, which is feasible by the following lemma.

\begin{lemma}\label{lemmab4}
 $\rho_*^{10d_H}\lambda_m(v)\leq \lambda_{m+2}(w)$ for any $m\geq 1$ and $w,v\in \Lambda_*$.
\end{lemma}
\begin{proof}
Take $f\in l(\mcB_m(v))$ so that $\rho_*^{md_H}\sum_{\tau\in \mcB_m(v)}\big(f(\tau)-[f]_{\mcB_m(v)}\big)^2=\lambda_m(v)$ and
$\mcD_{\|v\|+m,\mcB_m(v)}(f)=1$. By Lemma \ref{lemma23}, $w^{-1}\cdot\mcB_{m}(w)$ is finer than $\Lambda_{m+1}$, and $\Lambda_{m+1}$ is finer than $v^{-1}\cdot\mcB_m(v)$. So we can define $g\in l(\mcB_{m+2}(w))$ by $g(w\cdot\tau)=f(v\cdot \tau')$, where $\tau'$ is uniquely determined by $\tau\in \tau'\cdot W_*$. Still by Lemma \ref{lemma23}, $\rho_{\tau}\geq\rho_*^{m+3}$, $\rho_{\tau'}\leq\rho_*^{m-1}$, and so each $\tau'$ determines at most $\rho_*^{-4d_H}$ different  $\tau$'s. One can see
\[\mcD_{\|w\|+m+2,\mcB_{m+2}(w)}(g)\leq \rho_*^{-8d_H}\mcD_{\|v\|+m,\mcB_m(v)}(f)=\rho_*^{-8d_H}.\]
In addition, noticing that $[g]_{\mcB_{m+2}(w)}=[f]_{\mcB_m(v)}$, we have \[\rho_*^{(m+2)d_H}\sum_{\iota\in \mcB_{m+2}(w)}\big(g(\iota)-[g]_{\mcB_{m+2}(w)}\big)^2\geq \rho_*^{2d_H}\lambda_m(v).\]
The lemma follows immediately.
\end{proof}

The rest of the proof will be similar to that of Kusuoka-Zhou \cite{KZ}, with slight modifications.

\begin{lemma}\label{lemmab5}
  Let $n,m\geq 1$, and $\{\varphi_w\}_{w\in\Lambda_n}$ be a collection of non-negative functions in $l(\Lambda_{n+m})$ such that  $supp(\varphi_w)\subset \mcB_m(\mcN_2(w))$ and $\sum_{w\in \Lambda_n}\varphi_w = 1$. For any $f\in l(\Lambda_n)$, we define $\tf\in l(\Lambda_{n+m})$ as
  $\tf = \sum_{w\in \Lambda_n}f(w)\varphi_w.$
Then there exists $C>0$ such that
  $$\mcD_{n + m}(\tf) \leq C\max\big\{\mcD_{n + m}(\varphi_w):w\in \Lambda_n\big\}\cdot\mcD_n(f).$$
\end{lemma}
\begin{proof}
  First, if $v\in \mcB_m(w)$ for some $w\in \Lambda_n$ and $v'\sim v$, we can see that $\{v,v'\}\subset \mcB_m\big(\mcN_1(w)\big)$. Hence, $\varphi_{w'}(v)\vee \varphi_{w'}(v')>0$ only if $w'\in \mcN_3(w)$, noticing that $ \varphi_{w'}$ supports in $\mcB_m\big(\mcN_2(w')\big)$. As a consequence, $\tilde{f}(v)-\tilde{f}(v')=\sum_{w'\in \mcN_3(w)}f(w')\big( \varphi_{w'}(v)- \varphi_{w'}(v')\big)$. Also, noticing that $\sum_{w'\in \mcN_3(w)}f(w) \varphi_{w'}(v)=\sum_{w'\in \mcN_3(w)}f(w) \varphi_{w'}(v')=f(w)$, one finally get
  \[\tilde{f}(v)-\tilde{f}(v')=\sum_{w'\in \mcN_3(w)}\big(f(w')-f(w)\big)\big( \varphi_{w'}(v)- \varphi_{w'}(v')\big),\qquad \forall w\in \Lambda_n,v\in \mcB_m(w)\text{ and }v'\sim v.\]
  Hence
  \[\begin{aligned}
  \mcD_{n + m}(\tf)&=\frac12\sum_{w\in \Lambda_n}\sum_{v\in \mcB_m(w)}\sum_{v'\sim v}\big(\tilde{f}(v)-\tilde{f}(v')\big)^2\\
  &=\frac{1}{2} \sum_{w\in \Lambda_n}\sum_{v\in \mcB_m(w)}\sum_{v'\sim v}\big(\sum_{w'\in \mcN_{3}(w)}(f(w') - f(w))(\varphi_{w'}(v)-\varphi_{w'}(v'))\big)^2.
  \end{aligned}\]
  Next we apply the Cauchy-Schwarz inequality to see
  \begin{align*}
  \mcD_{n + m}(\tf)
  &\leq \frac{1}{2}\sum_{w\in \Lambda_n}\sum_{v\in \mcB_m(w)}\sum_{v'\sim v}\bigg(\sum_{w'\in \mcN_{3}(w)}\big(f(w') - f(w)\big)^2\sum_{w''\in \mcN_{3}(w)}\big(\varphi_{w''}(v) - \varphi_{w''}(v')\big)^2\bigg)\\
  &=\frac{1}{2}\sum_{w\in \Lambda_n}\big(\sum_{w'\in \mcN_{3}(w)}\big(f(w') - f(w)\big)^2\big)\cdot\big(\sum_{w''\in \mcN_{3}(w)}\sum_{v\in \mcB_m(w)}\sum_{v'\sim v}\big(\varphi_{w''}(v) - \varphi_{w''}(v')\big)^2\big)\\
  &\leq \frac{1}{2}\sum_{w\in\Lambda_n}\big(\sum_{w'\in \mcN_{3}(w)}\big(f(w') - f(w)\big)^2\big)\cdot\big(\sum_{w''\in \mcN_{3}(w)}2\mcD_{n+m}(\varphi_{w'})\big)\\
  &\leq \sum_{w\in\Lambda_n}\sum_{w'\in \mcN_{3}(w)}\big(f(w') - f(w)\big)^2\cdot \#\mcN_{3}(w)\cdot\max\big\{\mcD_{n + m}(\varphi_{w''}):w''\in \Lambda_n\big\}.
  \end{align*}
  The desired estimate follows, since $\sum_{w\in \Lambda_n}\sum_{w'\in \mcN_{3}(w)}\big(f(w') - f(w)\big)^2 \leq C_1\mcD_n(f)$ for some $C_1>0$ by Lemma \ref{lemmab2}, and $\#\mcN_3(w)\leq M_0^3$ by Lemma \ref{lemmab1}.
\end{proof}

\begin{lemma}\label{lemmab6}
  Let $n, m\geq 1$. For any $w\in \Lambda_n$, let $u_w\in l(\Lambda_{n+m})$ satisfy $u_w|_{\mcB_m(w)} = 1, u_w|_{\mcB_m(\mcN_2^c(w))} = 0$, and $\mcD_{n+m}(u_w) = R_{n+m}\big(\mcB_m(w),\mcB_m(\mcN_2^c(w))\big)^{-1}$. Define $u = \sum_{w\in \Lambda_n}u_w$ and $\varphi_w = u_w/u$ for any $w\in l(\Lambda_n)$.
Then there exists $C>0$ such that
  $$\mcD_{n+m}(\varphi_w) \leq C R_{m}^{-1}, \quad \forall w\in \Lambda_n.$$
\end{lemma}
\begin{proof}
Immediately, we have $\mcD_{n+m}(u_w) \leq R_m^{-1}$ for any $w\in \Lambda_n$, $n\geq 1$. So for any  $w\in \Lambda_n$, we have
$\mcD_{n+m}(\varphi_w)
  = \sum_{v\sim v'\in \Lambda_{n+m}}\big(\frac{u_w(v) - u_w(v')}{u(v)} - u_w(v')\cdot\frac{u(v) - u(v')}{u(v)u(v')}\big)^2,$ and thus

  $$\begin{aligned}\mcD_{n+m}(\varphi_w)
  &\leq 2\mcD_{n+m}(u_w) + 2\sum_{v\sim v'\in \Lambda_{n+m}}\big(\sum_{w'\in \mcN_{3}(w)}(u_{w'}(v') - u_{w'}(v))\big)^2\\
  &\leq 2\mcD_{n+m}(u_w) + 2\#\mcN_{3}(w)\cdot\sum_{w'\in \mcN_{3}(w)}\sum_{v\sim v'\in \Lambda_{n+m}}\big(u_{w'}(v') - u_{w'}(v)\big)^2\\
  &\leq 2\mcD_{n+m}(u_w) + 2(\#\mcN_{3}(w))^2\cdot\max\{\mcD_{n+m}(u_{w'}):w'\in \Lambda_n\}\leq 2(1 + M_0^6)R_m^{-1},
  \end{aligned}$$
  where the first inequality follows from $u\geq 1$ and $u_w\leq 1$, and the  last inequality follows from Lemma \ref{lemmab1}.
\end{proof}

The following is a weak version of the second inequality of Proposition \ref{prop32}.
\begin{proposition}\label{propb7}
  There exists $C > 0$, $k_0\geq 3$ such that $\lambda_{n + m + k_0} \geq CR_m\lambda_n$ for any $n,m\geq 1$.
\end{proposition}
\begin{proof}
  First, we let $f\in l(\Lambda_{n+2})$ so that $\mcD_{n+2}(f)=1$ and $\rho_*^{(n+2)d_H}\sum_{w\in \Lambda_{n+2}}\big(f(w)-[f]_{\Lambda_{n+2}}\big)^2=\lambda_{n+2}(\emptyset)$.

  Next, we fix $l\geq 1$ so that for any $w\in W_*$ there exists $\tau(w)\in \mcB_l(w)$ such that $\mcN_2(\tau(w))\subset \mcB_l(w)$. Define $f'\in l(\Lambda_{n+l+2})$ as $f'=\pi_{n+l+2}^{-1}\circ \pi_{n+2}f$. As in the proof of Lemma \ref{lemmab4}, one can easily see that $\mcD_{n+l+2}(f')\leq C_1\mcD_{n+2}(f)=C_1$ for some constant $C_1>0$ depending on $l$.

  Finally, we define $f''=\sum_{w\in \Lambda_{n+l+2}}f'(w)\varphi_w$ with each $\varphi_w\in l(\Lambda_{n+m+l+2})$ as in Lemma \ref{lemmab5}, where the existence of good $\varphi_w$'s is guaranteed by Lemma \ref{lemmab6}.  Hence, we have
  \[\mcD_{n+m+l+2}(f'')\leq C_2\mcD_{n+l+2}(f')R_m^{-1}\leq C_1C_2R_m^{-1},\]
  for some constant $C_2>0$.  On the other hand, by the choice of $l$, we have $f''(v)=f(w)$ for each $w\in \Lambda_{n+2}$ and $v\in \mcB_m\big(\tau(w)\big)$, where $\tau(w)$ is the same as in the previous paragraph. Noticing that $\rho_*^{n+3}<\rho_w\leq\rho_*^{n+2},\forall w\in \Lambda_{n+2}$, and by Proposition \ref{prop24}, $\#\mcB_m(\tau)\geq C_3\rho_*^{-md_H},\forall m\geq 1,\tau\in W_*$ for some $C_3>0$, we can see that
  \[\begin{aligned}
  &\rho_*^{(n+m+l+2)d_H}\sum_{v\in \Lambda_{n+m+l+2}}\big(f''(v)-[f'']_{\Lambda_{n+m+l+2}}\big)^2\\
  \geq&
  \rho_*^{ld_H}\sum_{w\in \Lambda_{n+2}}\rho_*^{(n+2)d_H}\cdot\rho_*^{md_H}\sum_{v\in \mcB_m(\tau(w))}\big(f(w)-[f'']_{\Lambda_{n+m+l+2}}\big)^2\\
  \geq& C_3\rho_*^{ld_H}\sum_{w\in \Lambda_{n+2}}\rho_w^{d_H}\big(f(w)-[f'']_{\Lambda_{n+m+l+2}}\big)^2\geq C_3 \rho_*^{ld_H}\sum_{w\in \Lambda_{n+2}}\rho_w^{d_H}\big(f(w)-[f]_{\Lambda_{n+2}}\big)^2 \\
  \geq &C_3\rho_*^{(l+1)d_H}\rho_*^{(n+2)d_H}\sum_{w\in \Lambda_{n+2}}\big(f(w)-[f]_{\Lambda_{n+2}}\big)^2=C_3\rho_*^{(l+1)d_H}\lambda_{n+2}(\emptyset)\geq C_3\rho_*^{(l+11)d_H}\lambda_n,
  \end{aligned}\]
  where the last inequality is due to Lemma \ref{lemmab4}. The proposition follows immediately.
\end{proof}

\subsection{Proof of Proposition \ref{prop32}}
To fulfill the proof of Proposition \ref{prop32}, we need two more lemmas.

\begin{lemma}\label{lemmab8}
  Let $n,m\geq 1$, and $f\in l(\Lambda_{n + m})$. Define $\tf\in l(\Lambda_n)$ as
  $\tf  = \pi_n^{-1}\circ P_{n}\circ \pi_{n + m}f$, i.e. $\tilde{f}(w)= [f]_{\mcB_{m}(w)}$ for any $w\in \Lambda_n.$
  Then there exists a constant $C>0$ independent of $n,m$ and $f$ such that
  $$\mcD_{n,A}(\tf) \leq C\sigma_m\mcD_{n + m,\mcB_m(A)}(f)$$
  for any non-empty connected set $A \subset \Lambda_n$.
\end{lemma}

\begin{proof}
By the definition of $\sigma_m$, for each pair $w\sim v\in A$, we have
$\big(\tf(w) - \tf(v)\big)^2\leq \sigma_m\mcD_{n + m,\mcB_m(\{w,v\})}(f)$. Summing up the both sides over $w\sim v\in A$, we get the desired estimate.
\end{proof}

\begin{lemma}\label{lemmab9}
  There exists $C >0$ such that $C\lambda_{n + m} \leq \rho_*^{nd_H}\lambda_m +\lambda_n \sigma_m$ for any $m,n\geq 1$.
\end{lemma}
\begin{proof}
Let $m,n\geq 1$, $l\geq 0$,  $w\in \Lambda_l$, and $f\in l\big(\mcB_{n + m}(w)\big)$. Define $\tf\in l(\mcB_{n}(w))$ as
  $\tf(v) = [f]_{\mcB_{m}(v)}$ for any $v\in \mcB_{n}(w).$
Then we have
  $$\begin{aligned}
  &\sum_{\tau\in \mcB_{n + m}(w)}\big(f(\tau) - [f]_{\mcB_{n + m}(w)}\big)^2\\
  &= \sum_{v\in \mcB_{n}(w)}\sum_{\tau\in \mcB_{m}(v)}\big(f(\tau) - [f]_{\mcB_{m}(v)} + \tf(v) - [\tf]_{\mcB_{n}(w)}\big)^2\\
   &\leq 2\sum_{v\in \mcB_{n}(w)}\sum_{\tau\in \mcB_m(v)}\big(f(\tau) - [f]_{\mcB_{m}(v)}\big)^2 + C_1\rho_*^{-md_H}\sum_{v\in \mcB_n(w)}\big(\tf(v) - [\tf]_{\mcB_n(w)}\big)^2\\
  &\leq 2\sum_{v\in \mcB_{n}(w)}\rho_*^{-md_H}\lambda_m(v)\mcD_{l + n + m,\mcB_m(v)}(f) + C_1\rho_*^{-(m+n)d_H}\lambda_n(w)\mcD_{l + n,\mcB_n(w)}(\tf)\\
  & \leq C \rho_*^{-(m + n)d_H}\cdot (\rho_*^{nd_H}\lambda_m + \lambda_n\sigma_m)\cdot \mcD_{l + n + m,\mcB_{n+m}(w)}(f),
  \end{aligned}$$
 for some $C_1, C>0$, where the first inequality follows from Proposition \ref{prop24}, and the last inequality follows from Lemma \ref{lemmab8}. This gives
  $\lambda_{n + m} \leq C(\rho_*^{nd_H}\lambda_m + \lambda_n\sigma_m).$
\end{proof}

Now we give the proof of Proposition \ref{prop32}.

\begin{proof}[Proof of Proposition \ref{prop32}]
By Proposition \ref{propb3}, it remains to prove $C^{-1}R_m\lambda_n\leq\lambda_{n+m}\leq C\lambda_n\sigma_m$ for any $m,n\geq 1$, for some constant $C>0$.

 Combining Proposition \ref{propb7} and Lemma \ref{lemmab9} together, we have
 $$(C_1R_{m-k_0} - \rho_*^{md_H})\lambda_n \leq \lambda_{m}\sigma_n.$$
 for some $C_1>0$, for all $m>k_0$, $n\geq 1$, where $k_0$ is the same constant in Proposition \ref{propb7}. By Proposition \ref{propb3}, we can choose $m_0$ large enough so that $C_1R_{m_0 - k_0} - \rho_*^{m_0d_H} > 0$, which gives that for all $n\geq 1$,
\begin{equation}\label{eqa1}
\lambda_n \leq (C_1R_{m_0 - k_0} - \rho_*^{m_0d_H})^{-1}\lambda_{m_0}\sigma_n:=C_2\sigma_n.
\end{equation}

On the other hand, again by Proposition \ref{propb7} and \ref{propb3}, we have $\lambda_{n + k_0 + 1} \geq C_3R_{n}\geq C_4\rho_*^{n d_H - 2n}$, and thus $\lambda_n \geq C_5\rho_*^{nd_H}$ for some $C_3$-$C_5 > 0$. Combining this with Lemma \ref{lemmab9} and (\ref{eqa1}), we get
 \begin{equation}\label{eqa2}
 \lambda_{n+m} \leq C_6(\lambda_n + \rho_*^{nd_H})\sigma_m \leq C_7\lambda_n\sigma_m
 \end{equation}
 for some $C_6, C_7>0$.
Still by Proposition \ref{propb7}, we then have
 \begin{equation}\label{eqa3}
 C_8R_m\lambda_n \leq \lambda_{n + m + k_0} \leq C_7\lambda_{n + m}\sigma_{k_0}
 \end{equation}
for some $C_8>0$. The desired result follows from (\ref{eqa2}) and (\ref{eqa3}).
\end{proof}

\subsection{Proof of Proposition \ref{prop33}} At last, we prove Proposition \ref{prop33}.

\begin{proof}[Proof of Proposition \ref{prop33}]
For $l = 0,\cdots, m$, we denote $w^{(l)}$ as the unique element in $\Lambda_{n + m -l}$ such that $v\in \mcB_l(w^{(l)})$. In particular, $w^{(0)} = v$ and $w^{(m)} = w$. Define $f_l\in l(\mcB_{ m-l}(w))$ as
$f_l(\tau) = [f]_{\mcB_l(\tau)}$
for any $\tau\in \mcB_{m - l}(w)$. Then for each $l = 1,\cdots,m$, we have
$$\begin{aligned}
\big(f_{l-1 }(w^{(l-1 )}) - f_{l}(w^{(l)})\big)^2
& = \big([f]_{\mcB_{l-1}(w^{(l-1)})}- [f]_{\mcB_{l}(w^{(l)})}\big)^2\\
& \leq \rho_{w^{(l-1)}}^{-d_H}\sum_{\tau\in\mcB_{l-1}(w^{(l-1)})}\rho_{\tau}^{d_H}\big(f(\tau)-[f]_{\mcB_{l}(w^{(l)})}\big)^2\\
&\leq \rho_*^{(l-2)d_H}\sum_{\tau\in\mcB_l(w^{(l)})}\big(f(\tau)-[f]_{\mcB_{l}(w^{(l)})}\big)^2 \leq \rho_*^{-2}\lambda_{l}\mcD_{n + m,\mcB_{m}(w)}(f).
\end{aligned}$$
Since by Proposition \ref{prop32}, $ \lambda_{l} \leq C_1\lambda_mR_{m - l}^{-1}\leq C_2\rho_*^{(m-l)(2 - d_H)}\lambda_m$ for some $C_1, C_2>0$, we then have
  $$\big|f(v) - [f]_{\mcB_m(w)}\big| \leq C_3\big(\lambda_m\mcD_{n + m,\mcB_{m}(w)}(f)\big)^{\frac{1}{2}} \sum_{l = 1}^{m}(\rho_*^{1 - d_H/2})^{m-l} \leq C^{\frac12}\big(\lambda_m\mcD_{n + m,\mcB_{m}(w)}(f)\big)^{\frac{1}{2}}$$ for some $C_3,C>0$.
This gives that
  $\big(f(v) - [f]_{\mcB_m(w)}\big)^2 \leq C\lambda_m\mcD_{n + m,\mcB_{m}(w)}(f)$ as desired.
\end{proof}

\section{Proof of Proposition \ref{prop35}}\label{AppendixC}

First, we recall the concept of $\Gamma$-convergence. Please refer to the book \cite{D} for general discussion on \textit{$\Gamma$-convergence}.

\begin{definition}[$\Gamma$-convergence]\label{defc1}
  Let $(X,d)$ be a metric space, and $f$, $f_n$, $n\geq 1$ be functions from $X$ to $[0,+\infty]$. If for any $x\in X$,

  (a). for any sequence $x_n$ converging to $x$ in $(X,d)$,
  $f(x)\leq \liminf_{n\to \infty}f_n(x_n);$

  (b). there exists a sequence $x_n$ converging to $x$ in $(X,d)$, such that
  $f(x) = \lim_{n\to \infty}f_n(x_n),$
  we say \emph{$f_n$ $\Gamma$-converges to $f$}.
\end{definition}

\begin{proof}[Proof of Proposition \ref{prop35}]
It follows from Proposition \ref{prop32} and Assumption \textbf{(B)}, there is  $0 < r< 1$ such that
  $R_m\asymp \lambda_m\asymp \sigma_m \asymp r^{-m}.$ Denote $\bar{\mcD}_n = r^{-n}\mcD_n$ for $n \geq 1$.

By Proposition \ref{prop33}, there is $C_1 > 0$ such that, for  $n,m \geq 1$ and $f\in L^2(K,\mu)$,
  $$\mcD_n(f) = \sum_{w\sn v}\big(\pi_n^{-1}\circ P_nf(w) - \pi_n^{-1}\circ P_nf(v)\big)^2 \leq  C_1\lambda_m\sum_{w\sn v}\mcD_{n+m,\mcB_m(\{w,v\})}(f) \leq C_1r^{-m}\mcD_{n+m}(f),$$
  which gives $\bar{\mcD}_n(f) \leq C_1\bar{\mcD}_{n + m}(f)$.
By \cite[Proposition 3.8]{CQ3}, there is a subsequence $\{\bar{\mcD}_{\gamma(n)}\}_{n = 1}^{\infty}$ of $\{\bar{\mcD}_n\}_{n = 1}^{\infty}$ $\Gamma$-converging to a closed symmetric non-negative quadratic form, denoted as $\bar{\mathcal{E}}$.  Let $\mathcal{F} = \big\{f\in L^2(K,\mu):\bar{\mathcal{E}}(f) < \infty\big\}$.\vspace{0.2cm}

 \noindent \textit{Claim 1. For any $f\in\mathcal F$, $C_1^{-1}\sup_{n\geq 1}\bar{\mcD}_n(f) \leq \bar{\mathcal{E}}(f)\leq C_1\liminf_{n\to \infty}\bar{\mcD}_n(f)$.}\vspace{0.2cm}

On one hand, by the definition of $\Gamma$-convergence,
\[\bar{\mathcal{E}}(f) \leq \liminf_{n\to\infty}\bar{\mcD}_{\gamma(n)}(f) \leq C_1\liminf_{n\to\infty}\inf_{m \geq \gamma(n)}\bar{\mcD}_{m}(f)\leq C_1\liminf_{m\to\infty}\bar{\mcD}_{m}(f).\]
On the other hand, take $f_n\in L^2(K,\mathscr{F}_{\gamma(n)},\mu)$ such that $f_n\to f$ in $L^2(K,\mu)$ and $\bar{\mathcal{E}}(f) = \lim_{n\to\infty}\bar{\mcD}_{\gamma(n)}(f_n)$. Then for any $m\geq 1$,
$\bar{\mcD}_m(f) = \lim_{n\to\infty}\bar{\mcD}_m(f_n)\leq C_1\lim_{n\to\infty}\bar{\mcD}_{\gamma(n)}(f_n) = C_1\bar{\mathcal{E}}(f)$. So Claim 1 follows.\vspace{0.2cm}

For $n\geq 1,f\in L^2(K,\mu)$, as $P_nf$ is the orthogonal projection of $f$ to $L^2(K,\mathscr{F}_n,\mu)$, we may choose  the value of $P_nf(x)$ to be $\pi_n^{-1}\circ P_nf(w)$ for each $x\in \Psi_w(K\setminus \partial K), w\in \Lambda_n$. Let $K_0 := K\setminus \bigcup_{w\in \Lambda_*}\Psi_{w}\partial K$. Note that $\mu(K_0) = 1$.\vspace{0.2cm}

\noindent \textit{Claim 2. There is $C_2 > 0$ such that $\big|P_nf(x) - P_nf(y)\big|^2 \leq C_2\bar{\mathcal{D}}_n(f)\cdot |x - y|^{\theta}$ for any $n \geq 2,f\in L^2(K,\mu)$ and $x,y\in K_0$ with $|x - y| \geq c_0\rho_*^n$, where $c_0$ is the same constant in \textbf{(A3)}.}\vspace{0.2cm}

For $|x - y| \geq c_0\rho_*$, by Proposition \ref{prop33}, for some $C_3>0$, 
\[\big|P_nf(x) - P_nf(y)\big|^2 \leq C_3\bar{\mcD}_n(f) \leq C_3(c_0\rho_*)^{-\theta}\bar{\mcD}_n(f)|x - y|^{\theta}.\]
For $|x - y| < c_0\rho_*$, choose $m\geq 1$ so that $c_0\rho_*^{m + 1} \leq |x-y| < c_0\rho_*^m$. It follows that $m < n$. By \textbf{(A3)}, there are $w\sm w'\sm w''\in \Lambda_m$ so that $x\in \Psi_wK_0,y\in \Psi_{w''}K_0$, giving $|P_mf(x) - P_mf(y)|^2 \leq 2\mcD_m(f) \leq 2C_1r^{m}\bar{\mcD}_n(f)$. Still by Proposition \ref{prop33}, there is a constant $C_4>0$ such that $|P_mf(z) - P_nf(z)|^2 \leq C_4r^{-(n-m)}\mcD_n(f) = C_4r^{m}\bar{\mcD}_n(f)$  for $z = x,y$. Combine the estimates above, noticing the choice of $m$, for some $C_5>0$,
\[\big|P_nf(x) - P_nf(y)\big|^2 \leq 3\big(2C_1r^{m}\bar{\mcD}_n(f) +2 C_4r^{m}\bar{\mcD}_n(f)\big) \leq  C_5\bar{\mcD}_n(f)|x - y|^{\theta}.\]
So Claim 2 follows by taking $C_2 = C_3(c_0\rho_*)^{-\theta}\vee C_5$.\vspace{0.2cm}

For each $f\in \mathcal F$, since $P_nf\to f$ in $L^2(K,\mu)$, there is a subsequence $P_{n_k}f$ converging to $f$ a.e. $\mu$, i.e. there exists $A\subset K$ with $\mu(A) = 1$ so that $P_{n_k}f(x) \to f(x)$ for any $x\in A$. Thus by Claim 1 and 2, for each $x\neq y\in A\cap K_0$, we have $|f(x) - f(y)|^2 = \lim_{k\to\infty}|P_{n_k}f(x) - P_{n_k}f(y)|^2 \leq C_2\sup_{n_k\geq 1}\bar{\mcD}_{n_k}(f)|x - y|^{\theta} \leq C_6\bar{\mathcal{E}}(f)|x - y|^{\theta}$, where $C_6 = C_1C_2$.
In particular,  $\mu(A\cap K_0) = 1$ implies that  $f\in C(K)$ and $|f(x) - f(y)|^2 \leq C_6\bar{\mathcal{E}}(f)|x - y|^{\theta}$ for any $x,y\in K$.\vspace{0.2cm}

It remains to verify the Markov property and regular property of $(\bar{\mathcal{E}},\mathcal{F})$. Denote $\bar{f} = (f\vee 0)\wedge 1$ for any $f\in L^2(K,\mu)$. For each $f\in \mathcal{F}$, choose a sequence $f_n\in L^2(K,\mcM_{\gamma(n)},\mu)$ such that $f_n\to f$ in $L^2(K,\mu)$ and $\bar{\mathcal{E}}(f) = \lim_{n\to\infty}\bar{\mcD}_{\gamma(n)}(f_n)$. Then we have $\bar{f_n}\to \bar{f}$ in $L^2(K,\mu)$, and by the Markov property of $\bar{\mcD}_n$,
$\bar{\mathcal{E}}(f) \geq \liminf_{n\to\infty}\bar{\mcD}_{\gamma(n)}(\bar{f_n}) \geq \bar{\mathcal{E}}(\bar{f}).$
This gives the Markov property of $(\bar{\mathcal{E}},\mathcal{F})$.

As for the regular property of $(\bar{\mathcal{E}},\mathcal{F})$, it is enough to prove $\mathcal{F}$ is dense in $C(K)$. To this end, we will prove $\mathcal{F}$ is an algebra and separates points in $K$, and then apply the Stone-Weierstrass theorem.
First, for $f,g\in \mathcal{F}\subset C(K)$, since $P_nf\to f$, $P_ng\to g$ in $L^{\infty}(K,\mu)$ and $P_nf,P_ng$ are uniformly bounded, we have $P_nf\cdot P_ng\to f\cdot g$ in $L^{\infty}(K,\mu)$, thus in $L^2(K,\mu)$. So by Claim 1,
\[\begin{aligned}\bar{\mathcal{E}}(f\cdot g)
&\leq \liminf_{n\to\infty}\bar{\mcD}_{\gamma(n)}\big(P_{\gamma(n)}f\cdot P_{\gamma(n)}g\big)\\
&\leq \liminf_{n\to\infty}\big(\|P_{\gamma(n)}f\|^2_{L^\infty(K,\mu)}\cdot\bar{\mcD}_{\gamma(n)}(g) + \|P_{\gamma(n)}g\|^2_{L^\infty(K,\mu)}\cdot\bar{\mcD}_{\gamma(n)}(f)\big)\\
&\leq C_1\big(\|f\|^2_{L^\infty(K,\mu)}\cdot \bar{\mathcal{E}}(g) + \|g\|^2_{L^\infty(K,\mu)}\cdot\bar{\mathcal{E}}(f)\big) < \infty,
\end{aligned}\]
which gives that $\mathcal{F}$ is an algebra. Next, for any $x\neq y\in K$, we choose $w\in \Lambda_{n_0}$ for some $n_0\geq 1$ so that $x\in \Psi_wK$ and $y\in \bigcup_{v\in \mathcal N_2^c(w)}\Psi_vK$. Then for any $n>n_0$, we pick a function $f_n\in L^2(K,\mu)$ so that $f_n|_{\Psi_wK}=1$, $f_n|_{\bigcup_{v\in \mathcal N_2^c(w)}\Psi_vK}=0$ and $\bar{\mathcal D}_n (f)\leq R_{n-n_0}^{-1}$. Let $f$ be a weak limit of $f_n$. Then $f\in\mathcal F$ since for any $m\geq 1$, $\bar{\mathcal D}_m(f)=\lim_{n\to \infty}\bar{\mathcal D}_m(f_n)\leq C_1\liminf_{n\to\infty}\bar{\mathcal D}_n(f_n)$. Moreover, $f(x)=1$, $f(y)=0$, so that $f$ separates $x$ and $y$.
\end{proof}

\bibliographystyle{amsplain}

\end{document}